\documentclass[10pt]{article}

\usepackage{amssymb, amsthm, mathtools}
\usepackage[hidelinks, citecolor=blue]{hyperref}
\hypersetup{colorlinks=true, urlcolor = blue, linkcolor=black}
\usepackage{verbatim}
\usepackage{bm}
\usepackage{color}

\newcommand{\nc}{\normalcolor}

\numberwithin{equation}{section}

\newtheorem{theorem}{Theorem}
\newtheorem{corollary}[theorem]{Corollary}
\newtheorem{lemma}[theorem]{Lemma}
\newtheorem{prop}[theorem]{Proposition}
\newtheorem{defn}[theorem]{Definition}

\newcommand{\zz}{{\mathbb{Z}}}

\newcommand{\pp}{{\mathbb{P}}}
\newcommand{\oh}{{\mathcal{O}}}
\newcommand{\im}{{\operatorname{Im}\,}}
\newcommand{\re}{{\operatorname{Re }\,}}
\newcommand{\pf}{\operatorname{Pf}}

\newcommand{\tr}{{\operatorname{Tr}\,}}

\newcommand{\beq}[1]{\begin{equation} \label{#1}}
\newcommand{\eeq}{\end{equation}}

\newcommand{\expb}[1]{\operatorname{exp}\left\{#1\right\}}
\newcommand{\ntr}[1]{\left\langle #1 \right\rangle}
\usepackage[top=1.25in, bottom=1.25in, left=1.25in, right=1.25in]{geometry}
\newcommand*\samethanks[1][\value{footnote}]{\footnotemark[#1]}
\renewcommand{\nc}{\newcommand}
\nc{\rnc}{\renewcommand}
\nc{\R}{\mathbb{R}}
\nc{\C}{\mathbb{C}}
\rnc{\d}{\mathrm{d}}
\nc{\E}{\mathbb{E}}

\begin{document}
\title{Bulk universality for complex eigenvalues of real non-symmetric random matrices with i.i.d. entries}

\author{Sofiia Dubova\thanks{Department of Mathematics, Harvard University, One Oxford Street, Cambridge MA 02138, USA}
\and Kevin Yang\samethanks
}
\maketitle

\begin{abstract}
We consider an ensemble of non-Hermitian matrices with independent identically distributed real entries that have finite moments. We show that its $k$-point correlation function in the bulk away from the real line converges to a universal limit.
\end{abstract}

\tableofcontents

\section{Introduction}
Universality of local eigenvalue statistics has long been of interest in random matrix theory. In the case of Hermitian Wigner ensembles this is the content of Wigner-Dyson-Mehta conjecture \cite{mehta1967statistical}, which was resolved in \cite{erdHos2010bulk} using a \emph{three-step strategy}:
\begin{itemize}
    \item Local law for the resolvent $G(z) = (H-z)^{-1}$;
    \item Bulk universality for the perturbed ensemble $H_t = H+\sqrt{t}B$, where $B$ is Gaussian;
    \item Comparison of local statistics of $H$ with local statistics of the perturbed matrix $H_t$.
\end{itemize}
This strategy has since been widely used to show universality of local statistics of other Hermitian ensembles \cite{Huang_2015,aggarwal2019bulk}. An established method of carrying out the second step of the strategy for Hermitian ensembles is Dyson Brownian motion (DBM), a system of stochastic differential equations for the eigenvalues of $H_t$. For non-Hermitian random matrices, the natural analogue of DBM involves the left- and right-eigenvector overlaps of the ensemble, which makes its analysis very complicated; see Appendix A in \cite{BD20}, for example. In \cite{cipolloni2021edge} Cipolloni, Erd{\H{o}}s and Schr{\"o}der side-stepped the non-Hermitian DBM using Girko's formula \cite{girko1984} to show the universality of local statistics in the edge regime for both complex and real non-Hermitian i.i.d. ensembles. Girko's formula allows access to the eigenvalue statistics of a non-Hermitian matrix $A$ through the resolvent $G_z(\eta)$ of the Hermitization of $A$:
\[
G_z(\eta) = \begin{pmatrix}-i\eta&A-z\\A^\ast-\bar{z} & -i\eta\end{pmatrix}^{-1}.
\]
This strategy has not yielded the universality of local statistics in the bulk due to difficulty controlling the correlation of $G_{z_1}$ and $G_{z_2}$ in the intermediate regime $|z_1-z_2|\asymp \frac{1}{\sqrt{N}}$. Current techniques only allow to show independence of $G_{z_1}$ and $G_{z_2}$ in the regime $|z_1-z_2| > N^{-\frac12+\varepsilon}$, see \cite{cipolloni2023mesoscopic}.

Recently, Maltsev and Osman proved bulk universality for complex non-Hermitian i.i.d. matrices in \cite{MO} using a different approach based on supersymmetry for the second step of the three-step strategy. They derived an exact integral formula for the $k$-point correlation function of a random complex non-Hermitian matrix $A$ perturbed by a Gaussian $\sqrt{t}B$ and performed asymptotic analysis to show its convergence to the $k$-point correlation function of the complex Ginibre ensemble with Gaussian entries. This method was also used to show bulk universality for weakly non-Hermitian complex matrices in \cite{M1}. For the purposes of bulk universality, it is enough to take $t=N^{-\epsilon_{0}}$ (see Section 7 in \cite{MO}), though the analysis in \cite{MO} holds until the time-scale $N^{-1/2+\delta}$ assuming the initial data matrix satisfies technical (third-order) local law estimates. (Here, $\epsilon_{0},\delta>0$ are any small parameters independent of $N$.)

In this work we establish bulk universality for the real non-Hermitian i.i.d. ensemble away from the real line using the strategy of \cite{MO}. This builds on \cite{taovu}, in which the first four moments of the entries of $A$ are assumed to match those of the Gaussian distribution. (We note that \cite{taovu}, which builds on the methods in \cite{taovuhermitian} for Hermitian matrices, applies to both real and complex matrix ensembles and both real and complex eigenvalue statistics.) We discuss the additional complications that arise in the real case compared to the complex case after we state our main results.

To be precise, the random matrix ensemble of interest in this paper is given by $A=(A_{ij})_{i,j=1}^{N}$, where $N\gg1$ is the matrix size, and $A_{ij}$ are i.i.d. real-valued random variables that satisfy $\E A_{ij}=0$ and $\E|A_{ij}|^{2}=N^{-1}$ and $\E|A_{ij}|^{2p}\lesssim_{p}N^{-p}$ for all $p\geq1$. (Throughout this paper, we write $a\lesssim b$ if $|a|\leq C|b|$, with subscripts to indicate what parameters the constant $C$ depends on.)

Let us now introduce the main objects and constants appearing in this paper.
\begin{itemize}
\item First, we define $t=N^{-\epsilon_{0}}$ to be the time-scale of the Gaussian perturbation, where $\epsilon_{0}>0$ is independent of $N$.
\item For any $z\in\C$, we define $H_{z}(\eta):=[(A-z)(A-z)^{\ast}+\eta^{2}]^{-1}$ and $\tilde{H}_{z}(\eta):=[(A-z)^{\ast}(A-z)+\eta^{2}]^{-1}$. We also define the Hermitization
\begin{align*}
\mathcal{H}_{z}:=\begin{pmatrix}0&A-z\\A^{\ast}-\overline{z}&0\end{pmatrix}
\end{align*}
and recall the resolvent $G_{z}(\eta)=[\mathcal{H}_{z}-i\eta]^{-1}$. It turns out (see \cite{MO}) that 
\begin{align}
G_{z}(\eta)=\begin{pmatrix}i\eta H_{z}(\eta)&(A-z)H_{z}(\eta)\\H_{z}(\eta)(A-z)^{\ast}&i\eta \tilde{H}_{z}(\eta)\end{pmatrix}.\label{eq:gfformula}
\end{align}
\item Let $\eta_{z,t}$ solve $t\langle H_{z}(\eta_{z,t})\rangle=1$, where $\langle H\rangle=\frac{1}{\dim H}\tr H$ is normalized trace for any square matrix $H$. See the proof of Theorem 1.1 in \cite{MO} for existence and uniqueness of such $\eta_{z,t}$ for $t=N^{-\epsilon_{0}}$ (and for the estimate $\frac{1}{C}t\leq\eta_{z,t}\leq Ct$, which holds whenever the local law in Lemma \ref{1GE} holds).
\item Set $g_{z,t}=\eta_{z,t}\langle H_{z}(\eta_{z,t})\rangle$ and $\alpha_{z,t}=\eta_{z,t}^{2}\langle H_{z}(\eta_{z,t})\tilde{H}_{z}(\eta_{z,t})\rangle$. Set $\beta_{z,t}=\eta_{z,t}\langle H_{z}(\eta_{z,t})^{2}(A-z)\rangle$ and $\gamma_{z,t}=\eta_{z,t}^{2}\langle H_{z}(\eta_{z,t})^{2}\rangle$. Define $\sigma_{z,t}=\alpha_{z,t}+\gamma_{z,t}^{-1}|\beta_{z,t}|^{2}$. 
\end{itemize}
Now, recall that the $k$-point correlation function (for any integer $k\geq1$) of a point process $\{\zeta_{j}\}_{j}$ on $\C$ is the function $\rho(z_{1},\ldots,z_{k})$ satisfying
\begin{align*}
\E\left[\sum_{i_{1}\neq i_{2}\neq\ldots\neq i_{k}}\varphi(\zeta_{i_{1}},\ldots,\zeta_{i_{k}})\right]&=\int_{\C^{k}}\varphi(z_{1},\ldots,z_{k})\rho(z_{1},\ldots,z_{k})dz_{1}\ldots dz_{k}.
\end{align*}
The goal of this paper is compute the $k$-point correlation function in the large $N$ limit for eigenvalues of $A$ in the complex plane, i.e. away from the real line. To state this result precisely, let us fix $k\geq2$ and define 
\begin{align*}
\rho_{\mathrm{GinUE}}^{(k)}(z_{1},\ldots,z_{k}):=\det\left[\frac{1}{\pi}\exp\left\{-\frac{|z_{i}|^{2}+|z_{j}|^{2}}{2}+z_{i}\overline{z}_{j}\right\}\right]_{i,j=1}^{k}
\end{align*}
This yields essentially the local distribution of $k$-many eigenvalues for a matrix whose entries are (normalized) \emph{complex} Gaussians. Remarkably, as we further explain shortly, it also gives the local eigenvalue distribution for real, non-symmetric Gaussian ensembles away from the real line.
\begin{theorem}\label{theorem:main}
Fix any $k\geq2$. For any $O\in C^{\infty}_{c}(\C^{k})$ and $z\in\C$ such that $|z|<1$ and $\mathrm{Im}(z)\neq0$, we have that
\begin{align*}
&\E\left[\sum_{i_{1}\neq i_{2}\neq\ldots\neq i_{k}}O(N^{\frac12}\sigma_{\mathrm{univ},z}^{\frac12}[z-\lambda_{i_{1}}],\ldots,N^{\frac12}\sigma_{\mathrm{univ},z}^{\frac12}[z-\lambda_{i_{k}}])\right]\\
&-\int_{\C^{k}}O(z_{1},\ldots,z_{k})\rho_{\mathrm{univ},z}^{(k)}(z_{1},\ldots,z_{k})dz_{1}\ldots dz_{k}
\end{align*}
vanishes as $N\to\infty$, where $\sigma_{\mathrm{univ},z}>0$ is independent of the distribution of the entries of $A$, where $\{\lambda_{j}\}_{j}$ are eigenvalues of $A$, where
\begin{align*}
\rho_{\mathrm{univ},z}^{(k)}(z_{1},\ldots,z_{k})&=\Phi_{z}(z_{1},\ldots,z_{k})\rho_{\mathrm{GinUE}}^{(k)}(z_{1},\ldots,z_{k}),
\end{align*}
and where $\Phi_{z}(z_{1},\ldots,z_{k})$ is independent of the distribution of the entries of $A$.
\end{theorem}
By Theorem 11 in \cite{BS}, we know $\Phi_{z}(z_{1},\ldots,z_{k})\equiv1$, i.e. that local bulk eigenvalue statistics away from the real line agree with those in the complex case. We only stated Theorem \ref{theorem:main} in this way to parallel it with Theorem \ref{theorem:mainwitht}. We note a similar analysis of local laws as in Section 7 of \cite{MO} shows $\sigma_{\mathrm{univ},z}=1$, but we omit this extra computation. 

The main step in proving Theorem \ref{theorem:main} is to first prove the same statement but for matrices with a Gaussian perturbation. In what follows, $B$ is a real Ginibre matrix of size $N\times N$, i.e. its entries are independent real Gaussians with mean zero and variance $N^{-1}$.
\begin{theorem}\label{theorem:mainwitht}
Fix any $k\geq2$. If we set $t=N^{-\epsilon_{0}}$, then for $\epsilon_{0}>0$ sufficiently small and for any $O\in C^{\infty}_{c}(\C^{k})$ and $z\in\C$ such that $|z|<1$ and $\mathrm{Im}(z)\neq0$, we have that
\begin{align*}
&\E\left[\sum_{i_{1}\neq i_{2}\neq\ldots\neq i_{k}}O(N^{\frac12}\sigma_{z,t}^{\frac12}[z-\lambda_{i_{1}}(t)],\ldots,N^{\frac12}\sigma_{z,t}^{\frac12}[z-\lambda_{i_{k}}(t)])\right]\\
&-\int_{\C^{k}}O(z_{1},\ldots,z_{k})\rho_{\mathrm{univ},z,t}^{(k)}(z_{1},\ldots,z_{k})dz_{1}\ldots dz_{k}
\end{align*}
vanishes as $N\to\infty$, where $\{\lambda_{j}(t)\}_{j}$ are eigenvalues of $A+\sqrt{t}B$, where
\begin{align*}
\rho_{\mathrm{univ},z,t}^{(k)}(z_{1},\ldots,z_{k})&=\Phi_{z,t}(z_{1},\ldots,z_{k})\rho_{\mathrm{GinUE}}^{(k)}(z_{1},\ldots,z_{k}),
\end{align*}
and where $\Phi_{z,t}(z_{1},\ldots,z_{k})$ is independent of the distribution of the entries of $A$.
\end{theorem}
We give proofs of Theorems \ref{theorem:main} and \ref{theorem:mainwitht} in the case $k=2$ for (notational) convenience. For general $k\geq2$, the same argument with minor adjustments works; see \cite{MO} for details.

Now we discuss additional difficulties that arise in case of real non-symmetric i.i.d. ensembles in comparison with the complex case. In \cite{MO} the authors obtain the integral formula for the $k$-point correlation function through a change of variables in the space of complex non-Hermitian $N\times N$ matrices $M_N(\C)$ obtained by consecutive Householder transformations with respect to the eigenvectors. In our case, since real matrices still have complex eigenvalues, we have to adapt each step of the change of variable to isolate conjugate pairs of eigenvalues. This leads to another integration parameter $\theta$, where $2\theta$ is the angle between the corresponding right eigenvectors. It turns out $\theta$ concentrates around the value $\pi/4$. Showing this is one of the key steps, and the main challenge is that local law estimates do not seem to help; instead, we control the region $|\theta-\pi/4|\gg N^{-1/2+\delta}$ directly (here, $\delta>0$ is any small parameter). Additionally, to carry out the asymptotic analysis we need to show that certain resolvent quantities have universal deterministic approximations. In particular, we derive the local law and the two-resolvent local laws at scales proportional to $t=N^{-\epsilon_{0}}$ for the Hermitization of $I_2\otimes A - \Lambda\otimes I_N$, where $\Lambda$ is a deterministic $2\times 2$ matrix. This is a $2\times 2$ analogue of the $1\times1$ results in \cite{bourgade2014local,AET18,AET21,CEHS2023}. We prove this using the cumulant expansion argument commonly used for the proofs of two resolvent local laws, see \cite{CES2020,cipolloni2023eigenstate,cipolloni2022optimal,adhikari2023eigenstate}. Our final step is to remove the Gaussian component $\sqrt{t}B$, which amounts to a standard ``three-and-a-half" moment comparison approach and the Girko formula as in \cite{EYY11,CES20}, respectively.

Finally, let us mention that approximately a week before posting to the arXiv, \cite{M2} was posted to the arXiv, in which the same result was proven by the same general strategy, but with different approaches to some of the technical challenges alluded to in the previous paragraph. (These technical differences potentially affect the smallest time-scale $t$ achievable.) It is possible that combining the methods of our paper with those of \cite{M2} may prove Theorem \ref{theorem:mainwitht} for the optimal time-scale $t=N^{-1/2+\delta}$. Osman \cite{M2} also proves bulk universality on the real line; for this, another method based on ``spin variables" is employed to handle additional difficulties occuring therein. See \cite{M2} for details.
\subsection{Notation}
We use big-O notation, i.e. $a=O(b)$ means $|a|\leq C|b|$ for some constant $C>0$. Any subscripts in the big-O notation indicate what parameters the constant $C$ depends on. We also write $a\lesssim b$ to mean $a=O(b)$ and $a\gtrsim b$ to mean $b=O(a)$, with the same disclaimer about subscripts. We use the notation $[[a,b]] = [a,b] \cap \zz$ for an integer range.
\subsection{Acknowledgements}
We thank Benjamin McKenna, Horng-Tzer Yau, and Jun Yin for very helpful discussions (and in particular Jun Yin for helping with details in the moment comparison step). We would like to thank Mohammed Osman for helpful comments on an earlier draft (and pointing out a mistake). K.Y. was supported in part by the NSF under Grant No. DMS-2203075.
%
%
%
\section{Change-of-variables}
The goal of this section is to reproduce ideas in \cite{MO} but for real non-symmetric matrices.
\subsection{Preliminary steps}
Define the manifolds
\begin{align*}
\Omega&:=\Omega_{1}\times\Omega_{2}\times M_{(N-4)\times(N-4)}(\R)\\
\Omega_{1}&:=\R\times\R_{+}\times[0,\frac{\pi}{2})\times V^{2}(\R^{N})\times M_{(N-2)\times2}(\R),\\
\Omega_{2}&:=\R\times\R_{+}\times[0,\frac{\pi}{2})\times V^{2}(\R^{N-2})\times M_{(N-4)\times2}(\R).
\end{align*}
Above, $V^{2}(\R^{d})$ is the Stiefel manifold
\begin{align*}
V^{2}(\R^{d}):=\mathbf{O}(d)/\mathbf{O}(d-2)=\{(v_{1},v_{2})\in\mathbb{S}^{d-1}\times\mathbb{S}^{d-1}: v_{1}^{\ast}v_{2}=0\},
\end{align*}
where $\mathbb{S}^{d-1}$ is the unit sphere in $\R^{d}$. Define the map $\Phi:\Omega\to M_{N\times N}(\R)$ given by 
\begin{align*}
\Phi(a_{1},b_{1},\theta_{1},\mathbf{v},W_{1},a_{2},b_{2},\theta_{2},\mathbf{w},W_{2},M^{(2)})&=R_{1}(\mathbf{v})\begin{pmatrix}\Lambda_{a_{1},b_{1},\theta_{1}}&W_{1}^{\ast}\\0&M^{(1)}\end{pmatrix}R_{1}(\mathbf{v})^{\ast},
\end{align*}
where 
\begin{align*}
M^{(1)}=R_{2}(\mathbf{u})\begin{pmatrix}\Lambda_{a_{2},b_{2},\theta_{2}}&W_{2}^{\ast}\\0&M^{(2)}\end{pmatrix}R_{2}(\mathbf{u})^{\ast}.
\end{align*}
Above, $R_{1}:V^{2}(\R^{N})\to\mathbf{O}(N)$ is a smooth map such that for any $\mathbf{v}=(v_{1},v_{2})\in V^{2}(\R^{N})$, we have $R_{1}(\mathbf{v})\mathbf{e}_{i}=v_{i}$ for $i=1,2$. Similarly, $R_{2}:V^{2}(\R^{N-2})\to\mathbf{O}(N-2)$ is a smooth map such that for any $\mathbf{u}=(u_{1},u_{2})\in V^{2}(\R^{N-2})$, we have $R_{2}(\mathbf{u})\mathbf{e}_{i}=u_{i}$ for $i=1,2$. Lastly, the matrix $\Lambda_{a,b,\theta}$ is given by 
\begin{align*}
\Lambda_{a,b,\theta}=\begin{pmatrix}a&b\tan\theta\\-\frac{b}{\tan\theta}&a\end{pmatrix}.
\end{align*}
Throughout, we will use the notation $\lambda_{j}=a_{j}+ib_{j}$.
\begin{lemma}\label{lemma:jacobian}
The Jacobian of the map $\Phi$ is given by 
\begin{align*}
J(\Phi)&=\frac{256b_{1}^{2}b_{2}^{2}|\cos2\theta_{1}||\cos2\theta_{2}|}{|\sin^{2}2\theta_{1}|\sin^{2}2\theta_{2}|}|\lambda_{1}-\lambda_{2}|^{2}|\lambda_{1}-\overline{\lambda}_{2}|^{2}\\
&\times\left|\det\left(M^{(2)}-\lambda_{1}\right)\right|^{2}\left|\det\left(M^{(2)}-\lambda_{2}\right)\right|^{2}.
\end{align*}
\end{lemma}
\begin{proof}
See Appendix \ref{section:jacobian}.
\end{proof}
Define the following probability measure on $M_{N\times N}(\R)$, in which $A$ is a deterministic matrix:
\begin{align*}
\rho(M)dM:=\left(\frac{N}{2\pi t}\right)^{\frac{N^{2}}{2}}\exp\left\{-\frac{N}{2t}\tr\left[(M-A)^{\ast}(M-A)\right]\right\}dM.
\end{align*}
By the change-of-variables formula and Lemma \ref{lemma:jacobian}, we have 
\begin{align*}
&\rho(M)dM=\left(\frac{N}{2\pi t}\right)^{\frac{N^{2}}{2}}\frac{256b_{1}^{2}b_{2}^{2}|\cos2\theta_{1}||\cos2\theta_{2}|}{|\sin^{2}2\theta_{1}|\sin^{2}2\theta_{2}|}|\lambda_{1}-\lambda_{2}|^{2}|\lambda_{1}-\overline{\lambda}_{2}|^{2}\\
&\times\left|\det\left(M^{(2)}-\lambda_{1}\right)\right|^{2}\left|\det\left(M^{(2)}-\lambda_{2}\right)\right|^{2}\exp\left\{-\frac{N}{2t}\tr(\Phi^{-1}M)^{\ast}(\Phi^{-1}M)\right\}\\
&\times da_{1}db_{1}d\theta_{1} d\mathbf{v}dW_{1}da_{2}db_{2}d\theta_{2} d\mathbf{u}dW_{2}dM^{(2)},
\end{align*}
where $\Phi^{-1}M$ is formal notation for the following matrix:
\begin{align*}
\Phi^{-1}M&=\begin{pmatrix}\Lambda_{a_{1},b_{1},\theta}&W_{1}^{\ast}\\0&M^{(1)}\end{pmatrix},\\
M^{(1)}&=\begin{pmatrix}\Lambda_{a_{2},b_{2},\theta_{2}}&W_{2}^{\ast}\\0&M^{(2)}\end{pmatrix}.
\end{align*}
We will integrate out all variables except $a_{1},b_{1},a_{2},b_{2}$. In this subsection, we focus on integrating out just $W_{1},W_{2}$, and $M^{(2)}$.

Now, we write $A$ in terms of the basis induced by $R_{1}(\mathbf{v})$ and $R_{2}(\mathbf{u})$. Precisely, we write
\begin{align*}
A&=R_{1}(\mathbf{v})\begin{pmatrix}A_{11}&B_{1}^{\ast}\\C_{1}&A^{(1)}\end{pmatrix}R_{1}(\mathbf{v})^{\ast}\\
A^{(1)}&=R_{2}(\mathbf{u})\begin{pmatrix}A^{(1)}_{11}&B_{2}^{\ast}\\C_{2}&A^{(2)}\end{pmatrix}R_{2}(\mathbf{u}).
\end{align*}
We clarify that $A_{11}$ and $A^{(1)}_{11}$ are blocks of size $2\times2$. Finally, set $B^{(2)}:=M^{(2)}-A^{(2)}$. An elementary computation shows that
\begin{align*}
\tr(\Phi^{-1}M)^{\ast}(\Phi^{-1}M)&=\left\|\left[I_{2}\otimes A-\Lambda_{a_{1},b_{1},\theta_{1}}\otimes I_{N}\right]\mathbf{v}\right\|^{2}\\
&+\left\|\left[I_{2}\otimes A^{(1)}-\Lambda_{a_{2},b_{2},\theta_{2}}\otimes I_{N-2}\right]\mathbf{u}\right\|^{2}\\
&+\tr\left[(W_{1}-B_{1})^{\ast}(W_{1}-B_{1})\right]\\
&+\tr\left[(W_{2}-B_{2})^{\ast}(W_{2}-B_{2})\right]\\
&+\tr(B^{(2)})^{\ast}B^{(2)}.
\end{align*}
From this, we get
\begin{align*}
\rho(M)dM&=\tilde{\rho}(a_{1},b_{1},\theta_{1},\mathbf{v},W_{1},a_{2},b_{2},\theta_{2},\mathbf{u},W_{2},B^{(2)})\\
&\times da_{1}db_{1}d\theta _{1}d\mathbf{v}dW_{1}da_{2}db_{2}d\theta_{2} d\mathbf{u}dW_{2}dB^{(2)},
\end{align*}
where $\tilde{\rho}(a_{1},b_{1},\theta_{1},\mathbf{v},W_{1},a_{2},b_{2},\theta_{2},\mathbf{u},W_{2},B^{(2)})$ is given by
\begin{align*}
&\left(\frac{N}{2\pi t}\right)^{\frac{N^{2}}{2}}\frac{256b_{1}^{2}b_{2}^{2}|\cos2\theta_{1}||\cos2\theta_{2}|}{|\sin^{2}2\theta_{1}|\sin^{2}2\theta_{2}|}|\lambda_{1}-\lambda_{2}|^{2}|\lambda_{1}-\overline{\lambda}_{2}|^{2}\\
&\times\left|\det[A^{(2)}+B^{(2)}-\lambda_{1}]\right|^{2}\left|\det[A^{(2)}+B^{(2)}-\lambda_{2}]\right|^{2}e^{-\frac{N}{2t}\tr(B^{(2)})^{\ast}B^{(2)}}\\
&\times\exp\left\{-\frac{N}{2t}\left\|\left[I_{2}\otimes A-\Lambda_{a_{1},b_{1},\theta_{1}}\otimes I_{N}\right]\mathbf{v}\right\|^{2}\right\}\\
&\times\exp\left\{-\frac{N}{2t}\left\|\left[I_{2}\otimes A^{(1)}-\Lambda_{a_{2},b_{2},\theta_{2}}\otimes I_{N-2}\right]\mathbf{u}\right\|^{2}\right\}\\
&\times\exp\left\{-\frac{N}{2t}\tr\left[(W_{1}-B_{1})^{\ast}(W_{1}-B_{1})\right]\right\}\\
&\times\exp\left\{-\frac{N}{2t}\tr\left[(W_{2}-B_{2})^{\ast}(W_{2}-B_{2})\right]\right\}.
\end{align*}
We note that $dM^{(2)}$ has turned into $dB^{(2)}$; the change-of-variable factor here is $1$ because the map $M^{(2)}\mapsto B^{(2)}$ is translation by $A^{(2)}$. Also, we clarify that $B$ and $B_{1}$ are functions of all parameters except for $W_{1},W_{2}$, respectively. Thus, by Gaussian integration, we have 
\begin{align*}
\int_{M_{(N-2)\times2}(\R)}\exp\left\{-\frac{N}{2t}\tr\left[(W_{1}-B_{1})^{\ast}(W_{1}-B_{1})\right]\right\}dW_{1}&=\left(\frac{2\pi t}{N}\right)^{N-2}\\
\int_{M_{(N-4)\times2}(\R)}\exp\left\{-\frac{N}{2t}\tr\left[(W_{2}-B_{2})^{\ast}(W_{2}-B_{2})\right]\right\}dW_{2}&=\left(\frac{2\pi t}{N}\right)^{N-4}.
\end{align*}
So, by integrating out $W_{1},W_{2}$, which appear only through the Gaussian weights in the last two lines, we get
\begin{align*}
&\int_{M_{(N-2)\times2}(\R)}\int_{M_{(N-4)\times2}(\R)}\tilde{\rho}(a_{1},b_{1},\theta_{1},\mathbf{v},W_{1},a_{2},b_{2},\theta_{2},\mathbf{u},W_{2},B^{(2)})dW_{1}dW_{2}\\
&=\left(\frac{N}{2\pi t}\right)^{\frac{N^{2}-4N+12}{2}}\frac{256b_{1}^{2}b_{2}^{2}|\cos2\theta_{1}||\cos2\theta_{2}|}{|\sin^{2}2\theta_{1}|\sin^{2}2\theta_{2}|}|\lambda_{1}-\lambda_{2}|^{2}|\lambda_{1}-\overline{\lambda}_{2}|^{2}\\
&\times\left|\det[A^{(2)}+B^{(2)}-\lambda_{1}]\right|^{2}\left|\det[A^{(2)}+B^{(2)}-\lambda_{2}]\right|^{2}e^{-\frac{N}{2t}\tr(B^{(2)})^{\ast}B^{(2)}}\\
&\times\exp\left\{-\frac{N}{2t}\left\|\left[I_{2}\otimes A-\Lambda_{a_{1},b_{1},\theta_{1}}\otimes I_{N}\right]\mathbf{v}\right\|^{2}\right\}\\
&\times\exp\left\{-\frac{N}{2t}\left\|\left[I_{2}\otimes A^{(1)}-\Lambda_{a_{2},b_{2},\theta_{2}}\otimes I_{N-2}\right]\mathbf{u}\right\|^{2}\right\}.
\end{align*}
We now integrate the previous expression over $B^{(2)}$. First, let $M_{k}(\R)$ be the space of $k\times k$ matrices with real entries.
\begin{lemma}\label{lemma:b2integration}
We have the identity
\begin{align}
&\left(\frac{N}{2\pi t}\right)^{\frac{N^{2}-4N+12}{2}}\int_{M_{N-4}(\R)}\prod_{j=1,2}|\det[A^{(2)}+B^{(2)}-\lambda_{j}]|^{2}e^{-\frac{N}{2t}\tr(B^{(2)})^{\ast}B^{(2)}}dB^{(2)}\nonumber\\
&=2^{6}\left(\frac{N}{2\pi t}\right)^{2N+4}\int_{M^{skew}_{4}(\C)}e^{-\frac{N}{2t}\tr X^{\ast}X}\operatorname{Pf}[\mathbf{M}(X)]dX,
\end{align}
where $M_{4}^{skew}(\C)$ is the space of $4\times4$ skew-symmetric complex matrices, and
\begin{align}
\mathbf{M}(X)&:=\begin{pmatrix}X\otimes I_{N-4}& A^{(2)}_\mathbf{w}\\-\left(A^{(2)}_{\mathbf{w}}\right)^{T}&X^\ast\otimes I_{N-4}\end{pmatrix},\label{eq:b2integrationMmatrix}\\
A^{(2)}_\mathbf{w} &:= \begin{bmatrix}I_{2}\otimes A^{(2)}-\mathbf{w}\otimes I_{N-4}&0\\0&\left(I_{2}\otimes A^{(2)}-\mathbf{w}\otimes I_{N-4}\right)^{\ast}\end{bmatrix}\\
\mathbf{w}&:=\begin{pmatrix}\lambda_{1}&0\\0&\lambda_{2}\end{pmatrix}.
\end{align}
\end{lemma}
\begin{proof}
This is very similar to the proof of Lemma 5.1 in \cite{MO}. We first use vectors $\chi_{1},\chi_{2},\psi_{1},\psi_{2}$ of anti-commuting Grassmann variables to write
\begin{align}
|\det[A^{(2)}+B^{(2)}-\lambda_{j}]|^{2}&=\int e^{-\chi_{j}^{\ast}(\lambda_{j}-A^{(2)}-B^{(2)})\chi_{j}-\psi_{j}^{\ast}(\overline{\lambda}_{j}-(A^{(2)})^{\ast}-(B^{(2)})^{\ast})\psi_{j}}d\chi_{j}d\psi_{j}. \nonumber
\end{align}
From this, we get
\begin{align*}
&\left|\det[A^{(2)}+B^{(2)}-\lambda_{1}]\right|^{2}\left|\det[A^{(2)}+B^{(2)}-\lambda_{2}]\right|^{2}e^{-\frac{N}{2t}\tr(B^{(2)})^{\ast}B^{(2)}}\\
&=\int\exp\left\{-\frac{N}{2t}\tr(B^{(2)})^{\ast}B^{(2)}-\tr B^{(2)}\sum_{j=1,2}\chi_{j}\chi_{j}^{\ast}-\tr(B^{(2)})^{\ast}\sum_{j=1,2}\psi_{j}\psi_{j}^{\ast}\right\}\\
&\times\exp\left\{-\sum_{j=1,2}\chi_{j}^{\ast}(\lambda_{j}-A^{(2)})\chi_{j}-\sum_{j=1,2}\psi_{j}^{\ast}(\overline{\lambda}_{j}-(A^{(2)})^{\ast})\psi_{j}\right\}d\chi_{1}d\chi_{2}d\psi_{1}d\psi_{2}.
\end{align*}
We can integrate the first exponential factor over $B^{(2)}$, since it is a Gaussian integral. We have
\begin{align}
&\left(\frac{N}{2\pi t}\right)^{\frac{(N-4)^{2}}{2}}\int_{M_{N-4}(\R)}e^{-\frac{N}{2t}\tr(B^{(2)})^{\ast}B^{(2)}-\tr B^{(2)}\sum_{j=1,2}\chi_{j}\chi_{j}^{\ast}-\tr(B^{(2)})^{\ast}\sum_{j=1,2}\psi_{j}\psi_{j}^{\ast}}dB^{(2)}\nonumber\\
&=\expb{\frac{t}{2N}\sum_{k,m=1}^{N-4}\left(
\sum_{j=1,2}\left(\chi_{j,k}\bar{\chi}_{j,m} + \psi_{j,m}\bar{\psi}_{j,k}\right)\right)^2}\nonumber\\
&=\exp\left\{-\frac{t}{N}\sum_{j,\ell=1,2}(\chi_{j}^{\ast}\psi_{\ell})(\psi_{\ell}^{\ast}\chi_{j})-\frac{t}{N}(\chi_1^T\chi_2)(\chi_2^\ast\bar{\chi}_1)-\frac{t}{N}(\psi_1^T\psi_2)(\psi_2^\ast\bar{\psi}_1)\right\},\nonumber
\end{align}
where the last identity follows by anti-commuting of Grassmann variables. By Gaussian integration, we can again write
\begin{align}
&\exp\left\{-\frac{t}{N}\sum_{j,\ell=1,2}(\chi_{j}^{\ast}\psi_{\ell})(\psi_{\ell}^{\ast}\chi_{j})\right\}\\
&=\left(\frac{N}{\pi t}\right)^{4}\int_{M_{2}(\C)}e^{-\frac{N}{t}\tr Y^{\ast}Y-i\sum_{j,\ell=1,2}(Y_{j\ell}\chi_{j}^{\ast}\psi_{\ell}+\overline{Y}_{j\ell}\psi_{\ell}^{\ast}\chi_{j})}dY\\
&=2^{4}\left(\frac{N}{2\pi t}\right)^{4}\int_{M_{2}(\C)}e^{-\frac{N}{t}\tr Y^{\ast}Y-i\sum_{j,\ell=1,2}(Y_{j\ell}\chi_{j}^{\ast}\psi_{\ell}+\overline{Y}_{j\ell}\psi_{\ell}^{\ast}\chi_{j})}dY.
\end{align}
and, similarly,
\begin{align}
&\frac{t}{N}(\chi_1^T\chi_2)(\chi_2^\ast\bar{\chi}_1) = 2\frac{N}{2\pi t}\int_{\C}e^{-\frac{N}{t} |S|^2-i\bar{S}\chi_1^T\chi_2-iS\chi_2^\ast\bar{\chi}_1}dS\\
&\frac{t}{N}(\psi_1^T\psi_2)(\psi_2^\ast\bar{\psi}_1) = 2\frac{N}{2\pi t}\int_{\C}e^{-\frac{N}{t} |T|^2-i\bar{T}\psi_1^T\psi_2-iT\psi_2^\ast\bar{\psi}_1}dT.
\end{align}
Now we change the integration variables from $Y, S, T$ to $X\in M^{skew}_4(\C)$ defined by
\[
X = 
\begin{pmatrix}
    \begin{matrix}0& iS\\-iS&0\end{matrix} & iY\\
    -iY^T & \begin{matrix}0& iT\\-iT&0\end{matrix}
\end{pmatrix}.
\]
The Jacobian of this change of variables is $1$. Since all remaining terms involving Grassmann variables are now quadratic in $\chi$, $\psi$, we can combine them into an exponent of a quadratic form and integrate 
\[
\int e^{-\frac12\phi^T \mathbf{M}(X) \phi} d\phi = \pf[\mathbf{M}(X)],
\]
where $\phi^T = (\chi_1^\ast, \chi_2^\ast, \psi_1^T, \psi_2^T, \chi_1^T, \chi_2^T, -\psi_1^\ast,-\psi_2^\ast)$ and $\mathbf{M}(X)$ is defined in the statement of the lemma. It remains to collect the Gaussian weights
\[
\expb{-\frac{N}{t}\left(\tr Y^\ast Y+|S|^2+|T|^2\right)} = \expb{-\frac{N}{2t}\tr X^\ast X}
\]
and the proof is finished.
\end{proof}
\subsection{Re-parameterizing \texorpdfstring{$\lambda_{1}$}{lambda1} and \texorpdfstring{$\lambda_{2}$}{lambda2}}
By Lemma \ref{lemma:b2integration}, we have 
\begin{align*}
&\int_{\substack{M_{(N-2)\times2}(\R)\\M_{(N-2)\times2}(\R)\\M_{(N-2)\times(N-2)}(\R)}}\tilde{\rho}(a_{1},b_{1},\theta_{1},\mathbf{v},W_{1},a_{2},b_{2},\theta_{2},\mathbf{u},W_{2},B^{(2)})dW_{1}dW_{2}dB^{(2)}\\
&=2^{6}\left(\frac{N}{2\pi t}\right)^{2N+4}\frac{256b_{1}^{2}b_{2}^{2}|\cos2\theta_{1}||\cos2\theta_{2}|}{|\sin^{2}2\theta_{1}|\sin^{2}2\theta_{2}|}|\lambda_{1}-\lambda_{2}|^{2}|\lambda_{1}-\overline{\lambda}_{2}|^{2}\\
&\times\int_{M^{skew}_{4}(\C)}e^{-\frac{N}{2t}\tr X^{\ast}X}\operatorname{Pf}[\mathbf{M}(X)]dX\\
&\times\exp\left\{-\frac{N}{2t}\left\|\left[I_{2}\otimes A-\Lambda_{a_{1},b_{1},\theta_{1}}\otimes I_{N}\right]\mathbf{v}\right\|^{2}\right\}\\
&\times\exp\left\{-\frac{N}{2t}\left\|\left[I_{2}\otimes A^{(1)}-\Lambda_{a_{2},b_{2},\theta_{2}}\otimes I_{N-2}\right]\mathbf{u}\right\|^{2}\right\}\\
&=:\tilde{\rho}(\lambda_{1},\theta_{1},\mathbf{v},\lambda_{2},\theta_{2},\mathbf{u}).
\end{align*}
Recall $\lambda_{j}=a_{j}+ib_{j}$. We now re-parameterize $\lambda_{1},\lambda_{2}$ as to be in a neighborhood of a fixed point $z=a+ib\in\C$ of radius $\sqrt{N}$; we emphasize that we always assume $b>0$. Set $\mathbf{z}=(z_{1},z_{2})$ and define
\begin{align}
\rho_{t}(z,\mathbf{z};A)&:=\frac{1}{N^{2}\sigma_{z,t}^{2}}\tilde{\rho}\left(z+\frac{z_{1}}{\sqrt{N\sigma_{z,t}}},z+\frac{z_{2}}{\sqrt{N\sigma_{z,t}}}\right),\\
\tilde{\rho}(\lambda_{1},\lambda_{2})&:=\int_{[0,\frac{\pi}{2}]}\int_{[0,\frac{\pi}{2}]}\int_{V^{2}(\R^{N})}\int_{V^{2}(\R^{N-2})}\tilde{\rho}(\lambda_{1},\theta_{1},\mathbf{v},\lambda_{2},\theta_{2},\mathbf{u})d\theta_{1}d\theta_{2}d\mathbf{v}d\mathbf{u}.
\end{align}
Here $d\mathbf{u}$ and  $d\mathbf{v}$ are defined as integration with respect to the rotationally invariant volume form on $V^{2}(\R^{N})$ and $V^{2}(\R^{N-2})$, respectively. In view of the previous definitions, we will also always identify $\lambda_{j}=z+N^{-1/2}\sigma_{z,t}^{-1/2}z_{j}$ for $j=1,2$. We now follow \cite{MO} and define
\begin{align*}
\tilde{F}(\mathbf{z};A^{(2)})&:=\frac{4N}{\pi^{4}t^{4}\sigma_{z,t}^{3}}|z_{1}-z_{2}|^{2}{|z-\overline{z}|^{2}}\int_{M^{skew}_{4}(\C)}e^{-\frac{N}{2t}\tr X^{\ast}X}\operatorname{Pf}[\mathbf{M}(X)]dX\\
\tilde{K}_{1}(z_{1})&:=\left(\frac{N}{2\pi t}\right)^{N}\int_{V^{2}(\R^{N})}\exp\left\{-\frac{N}{2t}\left\|\left[I_{2}\otimes A-\Lambda_{a_{1},b_{1},\theta_{1}}\otimes I_{N}\right]\mathbf{v}\right\|^{2}\right\}d\mathbf{v}\\
\tilde{K}_{2}(z_{2})&:=\left(\frac{N}{2\pi t}\right)^{N}\int_{V^{2}(\R^{N-2})}\exp\left\{-\frac{N}{2t}\left\|\left[I_{2}\otimes A^{(1)}-\Lambda_{a_{1},b_{1},\theta_{1}}\otimes I_{N-2}\right]\mathbf{u}\right\|^{2}\right\}d\mathbf{u}\\
d\nu_{1}(\mathbf{v})&:=\tilde{K}_{1}(z_{1})^{-1}\exp\left\{-\frac{N}{2t}\left\|\left[I_{2}\otimes A-\Lambda_{a_{1},b_{1},\theta_{1}}\otimes I_{N}\right]\mathbf{v}\right\|^{2}\right\}d\mathbf{v},\\
d\nu_{2}(\mathbf{u})&:=\tilde{K}_{2}(z_{2})^{-1}\exp\left\{-\frac{N}{2t}\left\|\left[I_{2}\otimes A^{(1)}-\Lambda_{a_{1},b_{1},\theta_{1}}\otimes I_{N-2}\right]\mathbf{u}\right\|^{2}\right\}d\mathbf{u}.
\end{align*}
The scaling above is chosen because the $V^{2}(\R^{N})$ and $V^{2}(\R^{N-2})$ integration resemble Gaussian integration with variance $N^{-1}$; this is why we want $N^{N}$ factors with $\tilde{K}_{j}(z_{j})$ terms. Finite powers of $N$, on the other hand, will not be important to keep track of carefully (since all error terms will be either come from multiplicative factors $1+o(1)$ or be exponentially small in $N$).

Again, by identifying $\lambda_{j}=z+N^{-1/2}\sigma_{z,t}^{-1/2}z_{j}$, in the formula \eqref{eq:b2integrationMmatrix} for the matrix $\mathbf{M}(X)$, we can write
\begin{align*}
\mathbf{w}=\begin{pmatrix}z+\frac{z_{1}}{\sqrt{N\sigma_{z,t}}}&0\\0&z+\frac{z_{2}}{\sqrt{N\sigma_{z,t}}}\end{pmatrix}.
\end{align*}
With this notation, we can write
\begin{align*}
\rho(z;\mathbf{z},A):=&\int_{[0,\frac{\pi}{2}]^{2}}\frac{256b_{1}^{2}b_{2}^{2}|\cos2\theta_{1}||\cos2\theta_{2}|}{|\sin^{2}2\theta_{1}|\sin^{2}2\theta_{2}|}\tilde{F}(\mathbf{z};A^{(2)})\tilde{K}_{1}(z_{1})\tilde{K}_{2}(z_{2})d\nu_{1}(\mathbf{v})d\nu_{2}(\mathbf{u})d\theta_{1}d\theta_{2}\\
&\times\left[1+O(N^{-\frac12}))\right].
\end{align*}
We will estimate $\tilde{F}(\mathbf{z};A^{(2)})$ following the ideas in \cite{MO}. We will then give a preliminary Fourier transform computation for $\tilde{K}_{j}(z_{j})$. Ultimately, the main problem for real non-symmetric matrices is the integration over $\theta$, which we explain further and start dealing with in the next section.
%
%
%
\subsection{Estimating \texorpdfstring{$\tilde{F}(\mathbf{z};A^{(2)})$}{F-tilde}}
The goal of this subsection is to prove the following analog of Lemma 4.1 in \cite{MO}. Before we state this lemma, we first introduce the following notation (in which $j\in\{1,2\}$):
\begin{align}
Z_{j}&:=\begin{pmatrix}0&z_{j}\\\overline{z}_{j}&0\end{pmatrix},\\
\delta_{z,t}&=\langle[H_{z}(\eta_{z,t})(A-z)]^{2}\rangle,\\
\tau_{z,t}&=\frac{\gamma_{z,t}\delta_{z,t}-\beta^{2}_{z,t}}{\gamma_{z,t}\sigma_{z,t}},\\
\psi_{j}&:=\exp\left\{-\frac{1}{\sqrt{N\sigma_{z,t}}}\tr[G_{z}(\eta_{z,t})Z_{j}]-\mathrm{Re}(\overline{\tau}_{z,t}z_{j}^{2})+|z_{j}|^{2}\right\}\\
V_{1}&:=V_{1}(\mathbf{v}):=\begin{pmatrix}\mathbf{v}_{1}&\mathbf{v}_{2}&0&0\\0&0&\mathbf{v}_{1}&\mathbf{v}_{2}\end{pmatrix},\\
V_{2}&:=V_{2}(\mathbf{u}):=\begin{pmatrix}\mathbf{u}_{1}&\mathbf{u}_{2}&0&0\\0&0&\mathbf{u}_{1}&\mathbf{u}_{2}\end{pmatrix}
\end{align}
We clarify that the dimension of $V_{1}$ is $2N\times4$ and the dimension of $V_{2}$ is $2(N-2)\times4$. Also, for the rest of this paper, when we say ``locally uniformly in $a_{j},b_{j}$", we mean for $a_{j},b_{j}=O(1)$, so that the eigenvalues are $\lambda_{j}=z+O(N^{-1/2})$ uniformly in $N$. (We will also use ``locally uniformly in $z_{j}$" to mean the same thing.) Next, let us introduce the constant 
\begin{align*}
\upsilon_{z,t}:=&\pi\left[Nt^{-1}-N\langle(A-z)^{\ast}\tilde{H}_{z}(\eta_{z,t})\tilde{H}_{\overline{z}}(\eta_{z,t})(A-z)\rangle\right]\\
&\times\left[Nt^{-1}-N\langle(A-z)H_{z}(\eta_{z,t})H_{\overline{z}}(\eta_{z,t})(A-\overline{z})\rangle\right].
\end{align*}
Lastly, we emphasize that all estimates hold with high probability with respect to randomness of $A$, i.e. with probability at least $1-o(1)$. We only need a finite number of estimates, so the end result holds with high probability (i.e. on an event which contributes $o(1)$ in the expectations in Theorems \ref{theorem:main} and \ref{theorem:mainwitht}). To avoid being verbose, we will not always explicitly mention this, e.g. when it concerns local law estimates in Lemmas \ref{1GE} and \ref{2GE}.
\begin{lemma}\label{lemma:festimate}
We have the following locally uniformly in $z_{1},z_{2}$, where $\kappa>0$:
\begin{align}
\frac{\tilde{F}(\mathbf{z};A^{(2)})}{\rho_{\mathrm{GinUE}}^{(2)}(z_{1},z_{2})}&=\frac{32\pi^{3}{|z-\overline{z}|^{2}}}{Nt^{4}\gamma_{z,t}^{2}\sigma_{z,t}^{3}\upsilon_{z,t}^{2}}e^{-\frac{2N}{t}\eta_{z,t}^{2}}\det[(A-z)^{\ast}(A-z)+\eta_{z,t}^{2}]^{2}\\
&\times\prod_{j=1,2}|\det[V_{j}^{\ast}G_{z}^{(j-1)}(\eta_{z,t})V_{j}]|^{2}\psi_{j}\left[1+O(N^{-\kappa})\right].
\end{align}
\end{lemma}
\begin{proof}
Before we begin, let us first comment on the strategy. The argument amounts to removing the diagonal $2\times2$ blocks from generic $X\in M_{4}^{skew}(\C)$ in the integration in $\tilde{F}(\mathbf{z};A^{(2)})$. Indeed, in view of the proof of Lemma \ref{lemma:b2integration}, these diagonal blocks correspond to extra terms which differentiate the real case from the complex case as in Section 5 of \cite{MO}. After this, we are left with a very similar computation as in Section 5 of \cite{MO}. (In fact, bulk correlation functions should agree between real and complex matrix ensembles for eigenvalues away from the real line, so we must arrive at the same Harish-Chandra-Itzykson-Zuber integral as in the complex case of \cite{MO}.)

For any $X\in M_{4}^{skew}(\C)$, set $P=XX^{\ast}$ and $Q=X^{\ast}X$. For $\delta>0$ small but fixed, write
\begin{align*}
&\int_{M_{4}^{skew}(\C)}e^{-\frac{N}{2t}\tr X^{\ast}X}\pf[\mathbf{M}(X)]dX\\
&=\int_{\|P-\eta_{z,t}^{2}\|_{\mathrm{max}}\leq N^{\delta-\frac12}}e^{-\frac{N}{2t}\tr X^{\ast}X}\pf[\mathbf{M}(X)]dX + \Upsilon\\
&=:\Upsilon_{0}+\Upsilon,
\end{align*}
where $\|\|_{\max}$ is the max-entry norm, and $\Upsilon$ is an error term that satisfies the bound
\begin{align*}
|\Upsilon|&\leq\sum_{j=1,\ldots,4}\int_{|P_{jj}-\eta_{z,t}^{2}|\geq N^{\delta-\frac12}}e^{-\frac{N}{2t}\tr X^{\ast}X}|\det[\mathbf{M}(X)]|^{\frac12}dX\\
&+\sum_{j=1,\ldots,4}\int_{|Q_{jj}-\eta_{z,t}^{2}|\geq N^{\delta-\frac12}}e^{-\frac{N}{2t}\tr X^{\ast}X}|\det[\mathbf{M}(X)]|^{\frac12}dX\\
&+\sum_{j,\ell,k=1,\ldots,4}\int_{\substack{|P_{j\ell}|,|Q_{j\ell}|\geq N^{\delta-\frac12}\\|P_{kk}-\eta_{z,t}^{2}|\leq N^{\delta-\frac12}\\|Q_{kk}-\eta_{z,t}^{2}|\leq N^{\delta-\frac12}}}e^{-\frac{N}{2t}\tr X^{\ast}X}|\det[\mathbf{M}(X)]|^{\frac12}dX\\
&=:\Upsilon_{1}+\Upsilon_{2}+\Upsilon_{3}.
\end{align*}
(We have used the fact that since $X$ is size $O(1)$, we know $\|P-\eta_{z,t}^{2}\|_{\max}=O(N^{\delta-1/2})$ is equivalent to $\|Q-\eta_{z,t}^{2}\|_{\max}=O(N^{\delta-1/2})$. We also used $|\pf[\mathbf{M}]|\leq|\det[\mathbf{M}]|^{1/2}$ for a skew symmetric matrix $\mathbf{M}$.) We first bound $\Upsilon$. We treat the region $|P_{jj}-\eta_{z,t}^{2}|\geq N^{\delta-1/2}$ for $j=1,\ldots,4$. We claim
\begin{align}
|\det[\mathbf{M}(X)]|&=|\det[\mathbf{M}(X)\mathbf{M}(X)^{\ast}]|^{\frac12}\label{eq:fisherM(X)}\\
&\leq\det[P\otimes I_{N-4}+A_{\mathbf{w}}^{(2)}(A_{\mathbf{w}}^{(2)})^{\ast}]^{\frac12}\det[Q\otimes I_{N-4}+(A_{\mathbf{w}}^{(2)})^{\ast}A_{\mathbf{w}}^{(2)}]^{\frac12}.
\end{align}
The second line follows by computing explicitly the diagonal entries of $\mathbf{M}(X)\mathbf{M}(X)^{\ast}$ and then using Lemma 2.2 in \cite{MO} (Fisher's inequality) to bound $|\det[\mathbf{M}(X)\mathbf{M}(X)^{\ast}]|$ by the product of determinants of its diagonal blocks. If we again compute each matrix in the second line and apply Lemma 2.2 in \cite{MO}, we obtain
\begin{align*}
|\det[\mathbf{M}(X)]|^{\frac12}&\leq\prod_{j=1,\ldots,4}|\det[P_{jj}+|A^{(2)}-\lambda_{j}|^{2}]|^{\frac14}|\det[Q_{jj}+|A^{(2)}-\lambda_{j}|^{2}]|^{\frac14},
\end{align*}
where $\lambda_{3}=\lambda_{1}$ and $\lambda_{4}=\lambda_{2}$. Multiply and divide by $\det[(A-z)^{\ast}(A-z)+\eta_{z,t}^{2}]$. This gives
\begin{align*}
|\det[\mathbf{M}(X)]|^{\frac12}&\leq\det[(A-z)^{\ast}(A-z)+\eta_{z,t}^{2}]^{2}\\
&\times\prod_{j=1,2}\frac{\det[(A^{(2)}-\lambda_{j})^{\ast}(A^{(2)}-\lambda_{j})+\eta_{z,t}^{2}]}{\det[(A-z)^{\ast}(A-z)+\eta_{z,t}^{2}]}\\
&\times\prod_{j=1,\ldots,4}\left\{\frac{\det[(A^{(2)}-\lambda_{j})^{\ast}(A^{(2)}-\lambda_{j})+P_{jj}]}{\det[(A^{(2)}-\lambda_{j})^{\ast}(A^{(2)}-\lambda_{j})+\eta_{z,t}^{2}]}\right\}^{\frac14}\\
&\times\prod_{j=1,\ldots,4}\left\{\frac{\det[(A^{(2)}-\lambda_{j})^{\ast}(A^{(2)}-\lambda_{j})+Q_{jj}]}{\det[(A^{(2)}-\lambda_{j})^{\ast}(A^{(2)}-\lambda_{j})+\eta_{z,t}^{2}]}\right\}^{\frac14}.
\end{align*}
Let us control the first product, i.e. the ratio of determinants. We first have
\begin{align}
&\frac{\det[(A^{(2)}-\lambda_{j})^{\ast}(A^{(2)}-\lambda_{j})+\eta_{z,t}^{2}]}{\det[(A-z)^{\ast}(A-z)+\eta_{z,t}^{2}]}\\
&=\frac{\det[(A^{(2)}-\lambda_{j})^{\ast}(A^{(2)}-\lambda_{j})+\eta_{z,t}^{2}]}{\det[(A-\lambda_{j})^{\ast}(A-\lambda_{j})+\eta_{z,t}^{2}]}\frac{\det[(A-\lambda_{j})^{\ast}(A-\lambda_{j})+\eta_{z,t}^{2}]}{\det[(A-z)^{\ast}(A-z)+\eta_{z,t}^{2}]}\\
&=\frac{\det[(A^{(2)}-\lambda_{j})^{\ast}(A^{(2)}-\lambda_{j})+\eta_{z,t}^{2}]}{\det[(A-\lambda_{j})^{\ast}(A-\lambda_{j})+\eta_{z,t}^{2}]}\frac{\det\begin{pmatrix}i\eta_{z,t}&A-\lambda_{j}\\(A-\lambda_{j})^{\ast}&i\eta_{z,t}\end{pmatrix}}{\det\begin{pmatrix}i\eta_{z,t}&A-z\\(A-z)^{\ast}&i\eta_{z,t}\end{pmatrix}}\\
&=\frac{\det[(A^{(2)}-\lambda_{j})^{\ast}(A^{(2)}-\lambda_{j})+\eta_{z,t}^{2}]}{\det[(A-\lambda_{j})^{\ast}(A-\lambda_{j})+\eta_{z,t}^{2}]}\det\left[1-\frac{1}{\sqrt{N\sigma_{z,t}}}G_{z}(\eta_{z,t})Z_{j}\right].
\end{align}
We now rewrite the last determinant as exponential of trace of log, after which we expand the logarithm to get
\begin{align*}
\det\left[1-\frac{1}{\sqrt{N\sigma_{z,t}}}G_{z}(\eta_{z,t})Z_{j}\right]&=\exp\left\{\tr\log\left[1-\frac{1}{\sqrt{N\sigma_{z,t}}}G_{z}(\eta_{z,t})Z_{j}\right]\right\}\\
&=\exp\left\{-\sum_{k=1}^{\infty}\frac{1}{k(N\sigma_{z,t})^{\frac{k}{2}}}\tr\left[\left(G_{z}(\eta_{z,t})Z_{j}\right)^{k}\right]\right\}.
\end{align*}
We can bound the $k\geq3$ contribution by using the trivial operator norm bound $\|G_{z}(\eta_{z,t})\|_{\mathrm{op}}\lesssim\eta_{z,t}^{-1}\lesssim t^{-1}$. Using also uniform boundedness of $\sigma_{z,t}$ away from $0$, this gives (for some $C=O(1)$)
\begin{align*}
\left|\sum_{k=3}^{\infty}\frac{1}{k(N\sigma_{z,t})^{\frac{k}{2}}}\tr\left[(G_{z}(\eta_{z,t})Z_{j})^{k}\right]\right|\lesssim \sum_{k=3}^{\infty}C^{k}N^{-\frac{k}{2}+1}t^{-\frac{k}{2}}\lesssim N^{-\frac12}\eta_{z,t}^{-\frac32}
\end{align*}
assuming $t=N^{-\epsilon_{0}}$ with $\epsilon_{0}>0$ small. Let us now separate the $k=1,2$ terms:
\begin{align*}
\exp\left\{-N^{\frac12}\sigma_{z,t}^{-\frac12}\langle G_{z}(\eta_{z,t})Z_{j}\rangle-\frac{1}{2\sigma_{z,t}}\langle G_{z}(\eta_{z,t})Z_{j}G_{z}(\eta_{z,t})Z_{j}\rangle\right\}.
\end{align*}
A straightforward computation using \eqref{eq:gfformula} and $(A-z)^{\ast}\tilde{H}_{z}(\eta)=H_{z}(\eta)(A-z)$ shows that
\begin{align*}
-\frac{1}{2\sigma_{z,t}}\langle G_{z}(\eta_{z,t})Z_{j}G_{z}(\eta_{z,t})Z_{j}\rangle&=\frac{1}{\sigma_{z,t}}\eta_{z,t}^{2}|z_{j}|^{2}\langle H_{z}(\eta_{z,t})\tilde{H}_{z}(\eta_{z,t})\rangle\\
&-\frac{z_{j}^{2}\langle[(A-z)^{\ast}H_{z}(\eta_{z,t})]^{2}\rangle}{2\sigma_{z,t}}\\
&-\frac{\overline{z}_{j}^{2}\langle[H_{z}(\eta_{z,t})(A-z)]^{2}\rangle}{2\sigma_{z,t}}.
\end{align*}
By definition of $\sigma_{z,t}$, we know that the first term on the RHS is equal to $|z_{j}|^{2}-\beta_{z,t}^{2}\sigma_{z,t}^{-1}\gamma_{z,t}^{-1}|z_{j}|^{2}$. Next, note $\langle[H_{z}(\eta_{z,t})(A-z)]^{2}\rangle=\overline{\langle(A-z)^{\ast}H_{z}(\eta_{z,t})]^{2}\rangle}$ since $H_{z}(\eta_{z,t})$ is self-adjoint. So, the second and third terms on the RHS of the previous identity combine to $-\mathrm{Re}(\overline{\delta_{z,t}}\sigma_{z,t}^{-1}z_{j}^{2})=-\mathrm{Re}(\overline{\tau_{z,t}}\sigma_{z,t}^{-1}z_{j}^{2})+\beta_{z,t}^{2}\sigma_{z,t}^{-1}\gamma_{z,t}^{-1}\mathrm{Re}(z_{j}^{2})$. Thus, in total, the previous display is $\leq|z_{j}|^{2}-\mathrm{Re}(\overline{\tau_{z,t}}\sigma_{z,t}^{-1}z_{j}^{2})$, and
\begin{align*}
&\exp\left\{-N^{\frac12}\sigma_{z,t}^{-\frac12}\langle G_{z}(\eta_{z,t})Z_{j}\rangle-\frac{1}{2\sigma_{z,t}}\langle G_{z}(\eta_{z,t})Z_{j}G_{z}(\eta_{z,t})Z_{j}\rangle\right\}\\
&\leq\exp\left\{-N^{\frac12}\sigma_{z,t}^{-\frac12}\langle G_{z}(\eta_{z,t})Z_{j}\rangle-\mathrm{Re}(\overline{\tau}_{z,t}\sigma_{z,t}^{-1}z_{j}^{2})+|z_{j}|^{2}\right\}=\psi_{j}.
\end{align*}
Thus, we have
\begin{align*}
&\det\left[1-\frac{1}{\sqrt{N\sigma_{z,t}}}G_{z}(\eta_{z,t})Z_{j}\right]\\
&=\psi_{j}\exp\left\{-\sum_{k=3}^{\infty}\frac{1}{k(N\sigma_{z,t})^{\frac{k}{2}}}\tr\left[\left(G_{z}(\eta_{z,t})Z_{j}\right)^{k}\right]\right\}\\
&=\psi_{j}\left[1+O(N^{-\frac12}\eta_{z,t}^{-3})\right].
\end{align*}
We now compute the remaining ratio of determinants using Cramer's rule to replace $A^{(2)}$ by $A^{(1)}$ and then by $A$. This is the same procedure as in (5.7) of \cite{MO}, but now the subspaces that we project resolvents onto must be spanned by pairs $\mathbf{v}\in V^{2}(\R^{N})$. Specifically, this gives
\begin{align}
\frac{\det[(A^{(2)}-\lambda_{j})^{\ast}(A^{(2)}-\lambda_{j})+\eta_{z,t}^{2}]}{\det[(A-\lambda_{j})^{\ast}(A-\lambda_{j})+\eta_{z,t}^{2}]}&=\prod_{\ell=1,2}\left|\det\left[V_{\ell}^{\ast}G_{\lambda_{j}}^{(\ell-1)}(\eta_{z,t})V_{\ell}\right]\right|,
\end{align}
where $G_{w}^{(\ell)}(\eta)$ is the same as $G_{w}(\eta)$ but replacing $A$ by $A^{(\ell)}$. We now replace $G^{(\ell-1)}_{\lambda_{j}}$ by $G^{(\ell-1)}_{z}$. We have
\begin{align*}
\left|\det\left[V_{\ell}^{\ast}G_{\lambda_{j}}^{(\ell-1)}(\eta_{z,t})V_{\ell}\right]\right|&=\left|\det\left[V_{\ell}^{\ast}G_{z}^{(\ell-1)}(\eta_{z,t})V_{\ell}\right]\right|\frac{\left|\det\left[V_{\ell}^{\ast}G_{\lambda_{j}}^{(\ell-1)}(\eta_{z,t})V_{\ell}\right]\right|}{\left|\det\left[V_{\ell}^{\ast}G_{z}^{(\ell-1)}(\eta_{z,t})V_{\ell}\right]\right|}
\end{align*}
and, by the resolvent identity, we also have the following (in which $G_{w}=G_{w}(\eta_{z,t})$ for convenience):
\begin{align}
&\frac{\left|\det\left[V_{\ell}^{\ast}G_{\lambda_{j}}^{(\ell-1)}(\eta_{z,t})V_{\ell}\right]\right|}{\left|\det\left[V_{\ell}^{\ast}G_{z}^{(\ell-1)}(\eta_{z,t})V_{\ell}\right]\right|}\label{eq:cramerrule}\\
&=\left|\det\left[1+(V_{\ell}^{\ast}G_{z}^{(\ell-1)}V_{\ell})^{-1}V_{\ell}^{\ast}(G_{\lambda_{j}}^{(\ell)}-G_{z}^{(\ell-1)})V_{\ell}^{\ast}\right]\right|\nonumber\\
&=\left|\det\left[1-\frac{1}{\sqrt{N\sigma_{z,t}}}(V_{\ell}^{\ast}G_{z}^{(\ell-1)}V_{\ell})^{-1}V_{\ell}^{\ast}G_{z}^{(\ell-1)}Z_{j}G_{\lambda_{j}}^{(\ell-1)}V_{\ell}\right]\right|.\nonumber
\end{align}
To control this determinant we are left with, we will again write it as exponential of a trace of a log and then expand the log. To this end, we need the following estimate, which uses only Lemma 2.1 of \cite{MO} and some elementary manipulations:
\begin{align*}
&N^{-\frac{k}{2}}\left|\tr\left[(V_{\ell}^{\ast}G_{z}^{(\ell-1)}V_{\ell})^{-1}V_{\ell}^{\ast}G_{z}^{(\ell-1)}Z_{j}G_{\lambda_{j}}^{(\ell-1)}V_{\ell}\right]^{k}\right|\\
&=N^{-\frac{k}{2}}\left|\tr\left[G_{\lambda_{j}}^{(\ell-1)}V_{\ell}(V_{\ell}^{\ast}G_{z}^{(\ell-1)}V_{\ell})^{-1}V_{\ell}^{\ast}G_{z}^{(\ell-1)}Z_{j}\right]^{k}\right|\\
&\lesssim N^{-\frac{k}{2}}\left\|G_{\lambda_{j}}^{(\ell-1)}V_{\ell}(V_{\ell}^{\ast}G_{z}^{(\ell-1)}V_{\ell})^{-1}V_{\ell}^{\ast}G_{z}^{(\ell-1)}Z_{j}\right\|_{\mathrm{op}}^{k}\\
&\lesssim C^{k}N^{-\frac{k}{2}}\left\|G_{\lambda_{j}}^{(\ell-1)}V_{\ell}(V_{\ell}^{\ast}G_{z}^{(\ell-1)}V_{\ell})^{-1}V_{\ell}^{\ast}G_{z}^{(\ell-1)}\right\|_{\mathrm{op}}^{k}\\
&\leq C^{k}N^{-\frac{k}{2}}\left\|G_{z}^{(\ell-1)}V_{\ell}(V_{\ell}^{\ast}G_{z}^{(\ell-1)}V_{\ell})^{-1}V_{\ell}^{\ast}G_{z}^{(\ell-1)}\right\|_{\mathrm{op}}^{k}\|[G_{z}^{(\ell-1)}]^{-1}G_{\lambda_{j}}^{(\ell-1)}\|_{\mathrm{op}}^{k}\\
&\lesssim C^{k}N^{-\frac{k}{2}}\left\|[\mathrm{Im}(G_{z}^{(\ell-1)})^{-1}]^{-1}\right\|_{\mathrm{op}}^{k}\|[G_{z}^{(\ell-1)}]^{-1}G_{\lambda_{j}}^{(\ell-1)}\|_{\mathrm{op}}^{k}\\
&=C^{k}N^{-\frac{k}{2}}\left\|[\mathrm{Im}(G_{z}^{(\ell-1)})^{-1}]^{-1}\right\|_{\mathrm{op}}^{k}\left\|I+N^{-\frac12}\sigma_{z,t}^{-\frac12}Z_{j}G_{\lambda_{j}}^{(\ell-1)}\right\|_{\mathrm{op}}^{k}\\
&=C^{k}N^{-\frac{k}{2}}\eta_{z,t}^{-k}(1+O(N^{-\frac12}\eta_{z,t}^{-1})).
\end{align*}
We clarify that the operator norm bound for the trace follows because the matrix in question rank $O(1)$. Since $\eta_{z,t}\gg N^{-1/2}$, the big-O term in the last line is $\ll1$. Thus, we have 
\begin{align}
&\left|\det\left[1-\frac{1}{\sqrt{N\sigma_{z,t}}}(V_{\ell}^{\ast}G_{z}^{(\ell-1)}V_{\ell})^{-1}V_{\ell}^{\ast}G_{z}^{(\ell-1)}Z_{j}G_{\lambda_{j}}^{(\ell-1)}V_{\ell}\right]\right|\\
&\leq\exp\left\{\sum_{k=1}^{\infty}\frac{1}{k}N^{-\frac{k}{2}}\sigma_{z,t}^{-\frac{k}{2}}\left|\tr\left[(V_{\ell}^{\ast}G_{z}^{(\ell-1)}V_{\ell})^{-1}V_{\ell}^{\ast}G_{z}^{(\ell-1)}Z_{j}G_{\lambda_{j}}^{(\ell-1)}V_{\ell}\right]^{k}\right|\right\}\\
&=\exp\left\{\sum_{k=1}^{\infty}O(C^{k}N^{-\frac{k}{2}}\eta_{z,t}^{-k})\right\}=1+O(N^{-\frac12}\eta_{z,t}^{-1}).
\end{align}
Gathering all of this, we deduce 
\begin{align*}
|\det[\mathbf{M}(X)]|^{\frac12}&\lesssim\det[(A-z)^{\ast}(A-z)+\eta_{z,t}^{2}]^{2}\psi_{1}\psi_{2}\prod_{\ell=1,2}|\det[V_{\ell}^{\ast}G_{z}^{(\ell-1)}(\eta_{z,t})V_{\ell}]|^{2}\\
&\times\prod_{j=1,\ldots,4}\left\{\frac{\det[(A^{(2)}-\lambda_{j})^{\ast}(A^{(2)}-\lambda_{j})+P_{jj}]}{\det[(A^{(2)}-\lambda_{j})^{\ast}(A^{(2)}-\lambda_{j})+\eta_{z,t}^{2}]}\right\}^{\frac14}\\
&\times\prod_{j=1,\ldots,4}\left\{\frac{\det[(A^{(2)}-\lambda_{j})^{\ast}(A^{(2)}-\lambda_{j})+Q_{jj}]}{\det[(A^{(2)}-\lambda_{j})^{\ast}(A^{(2)}-\lambda_{j})+\eta_{z,t}^{2}]}\right\}^{\frac14}.
\end{align*}
It remains to control the last two products over $j$; we treat the first, since the second is handled in the same way. By resolvent perturbation and the inequality $\log(1+x)\leq x-\frac{3x^{2}}{6+4x}$, we have the following, for which we use the notation $\tilde{H}_{\lambda_{j}}^{(2)}(\eta_{z,t}):=A_{\lambda_{j}}^{(2),\ast}A_{\lambda_{j}}^{(2)}+\eta_{z,t}^{2}$
\begin{align*}
&\frac{\det[(A^{(2)}-\lambda_{j})^{\ast}(A^{(2)}-\lambda_{j})+P_{jj}]}{\det[(A^{(2)}-\lambda_{j})^{\ast}(A^{(2)}-\lambda_{j})+\eta_{z,t}^{2}]}\\
&=\det\left[I_{N-4}+(P_{jj}-\eta_{z,t}^{2})\tilde{H}_{\lambda_{j}}^{(2)}(\eta_{z,t})\right]\\
&=\exp\left\{\tr\log\left[I_{N-4}+(P_{jj}-\eta_{z,t}^{2})\tilde{H}_{\lambda_{j}}^{(2)}(\eta_{z,t})\right]\right\}\\
&\leq\exp\left\{(P_{jj}-\eta_{z,t}^{2})\tr \tilde{H}_{\lambda_{j}}^{(2)}(\eta_{z,t})-\frac{3(P_{jj}-\eta_{z,t}^{2})}{6+4\frac{|P_{jj}-\eta_{z,t}|}{\eta_{z,t}^{2}}}\tr[\tilde{H}_{\lambda_{j}}^{(2)}(\eta_{z,t})]^{2}\right\}.
\end{align*}
By combining the previous two displays and by $\frac{t}{N}\tr \tilde{H}_{z}(\eta_{z,t})=1$, we have the following parallel to (5.6) in \cite{MO}:
\begin{align*}
\Upsilon_{1}&\lesssim e^{-\frac{2N\eta_{z,t}^{2}}{t}}\det[(A-z)^{\ast}(A-z)+\eta_{z,t}^{2}]^{2}\psi_{1}\psi_{2}\prod_{\ell=1,2}|\det[V_{\ell}^{\ast}G_{z}^{(\ell-1)}(\eta_{z,t})V_{\ell}]|^{2}\\
&\times\sum_{k=1,\ldots,4}\int_{|P_{kk}-\eta_{z,t}^{2}|\leq N^{\delta-\frac12}}\exp\left\{-\frac{Nt}{4}\sum_{j=1,\ldots,4}h_{j}(\eta_{z,t}^{-2}|p_{j}|)-\frac{Nt}{4}\sum_{j=1,\ldots,4}h_{j}(\eta_{z,t}^{-2}|q_{j}|)\right\}dX,
\end{align*}
where $p_{j}=P_{jj}-\eta_{z,t}^{2}$ and $q_{j}=Q_{jj}-\eta_{z,t}^{2}$ and 
\begin{align*}
h_{j}(x):=\frac{\eta_{z,t}^{4}}{Nt}\tr[\tilde{H}_{\lambda_{j}}^{(2)}(\eta_{z,t})]^{2}\frac{3x^{2}}{6+4x}-\frac{\eta_{z,t}^{2}}{Nt}\left(\tr[\tilde{H}_{\lambda_{j}}^{(2)}(\eta_{z,t})]-\tr[\tilde{H}_{z}(\eta_{z,t})]\right)x.
\end{align*}
By interlacing of eigenvalues and the bound $\|\tilde{H}_{\lambda_{j}}^{(2)}(\eta_{z,t})\|_{\mathrm{op}}\leq\eta_{z,t}^{-2}$, we replace $N^{-1}\tr[\tilde{H}_{\lambda_{j}}^{(2)}(\eta_{z,t})]^{2}$ by $N^{-1}\tr[\tilde{H}_{\lambda_{j}}(\eta_{z,t})]^{2}+O(N^{-1}\eta_{z,t}^{-4})$. By the same token, we can also replace $N^{-1}\tr[\tilde{H}_{\lambda_{j}}^{(2)}(\eta_{z,t})]$ by $N^{-1}\tr[\tilde{H}_{\lambda_{j}}(\eta_{z,t})]+O(N^{-1}\eta_{z,t}^{-2})$. We can then use a resolvent perturbation to show
\begin{align*}
\tr[\tilde{H}_{\lambda_{j}}(\eta_{z,t})]-\tr[\tilde{H}_{z}(\eta_{z,t})]&=\frac{1}{2N\eta_{z,t}}\tr\mathrm{Im}(G_{\lambda_{j}}(\eta_{z,t})-G_{z}(\eta_{z,t}))\\
&=\frac{1}{2N\eta_{z,t}}\mathrm{Im}\tr\left[N^{-\frac12}\sigma_{z,t}^{-\frac12}G_{\lambda_{j}}(\eta_{z,t})Z_{j}G_{z}(\eta_{z,t})\right]\\
&=O(N^{-\frac12}\eta_{z,t}^{-3}).
\end{align*}
Lastly, we have the lower bound $N^{-1}\tr[\tilde{H}_{\lambda{j}}(\eta_{z,t})]^{2}\geq C>0$, since $\tilde{H}_{\lambda_{j}}(\eta_{z,t})$ is the resolvent of a bounded operator. So, we obtain the lower bound $h_{j}(x)\geq C\eta_{z,t}^{3}(\frac{x^{2}}{1+x}-\frac{1}{N\eta_{z,t}^{3}}x)\geq C\eta_{z,t}^{3}\frac{x^{2}}{1+x}$ because $\eta=N^{-\epsilon_{0}}$ with $\epsilon_{0}>0$ small. (The discrepancy in the powers of $\eta_{z,t}$ compared to \cite{MO} will not be important; they come from using trivial bounds for resolvents as opposed to local law estimates.) In particular, if we take $\epsilon_{0}>0$ small enough in $\eta_{z,t}=N^{-\epsilon_{0}}$, then as in \cite{MO}, the integration over $|P_{kk}-\eta_{z,t}^{2}|\geq N^{-1/2+\delta}$ is exponentially small, i.e. for some $\kappa>0$, we have
\begin{align}
\Upsilon_{1}\lesssim e^{-CN^{\kappa}}e^{-\frac{2N\eta_{z,t}^{2}}{t}}\det[(A-z)^{\ast}(A-z)+\eta_{z,t}^{2}]^{2}\psi_{1}\psi_{2}\prod_{\ell=1,2}|\det[V_{\ell}^{\ast}G_{z}^{(\ell-1)}(\eta_{z,t})V_{\ell}]|^{2}.
\end{align}
By reversing the roles of $X$ and $X^{\ast}$, the same argument implies 
\begin{align}
\Upsilon_{2}\lesssim e^{-CN^{\kappa}}e^{-\frac{2N\eta_{z,t}^{2}}{t}}\det[(A-z)^{\ast}(A-z)+\eta_{z,t}^{2}]^{2}\psi_{1}\psi_{2}\prod_{\ell=1,2}|\det[V_{\ell}^{\ast}G_{z}^{(\ell-1)}(\eta_{z,t})V_{\ell}]|^{2}.
\end{align}
So, we are left to bound $\Upsilon_{3}$; let us treat the case $(j,\ell)=(1,2)$, since the other terms in $\Upsilon_{3}$ are treated in the same way. For this, we again start with \eqref{eq:fisherM(X)}. Now, let $P_{1}$ denote the upper left $2\times2$ block of $P=XX^{\ast}$. We now use Fisher's inequality (Lemma 2.2 in \cite{MO}) again, but now with the top left $2\times2$ block, the $(3,3)$ entry, and the $(4,4)$ entry as our diagonal blocks. This gives
\begin{align*}
&\det[P\otimes I_{N-4}+A_{\mathbf{w}}^{(2)}(A_{\mathbf{w}}^{(2)})^{\ast}]\\
&\leq|\det[P_{1}\otimes I_{N-4}+(I_{2}\otimes A^{(2)}-\mathbf{w}\otimes I_{N-4})(I_{2}\otimes A^{(2)}-\mathbf{w}\otimes I_{N-4})^{\ast}]|\\
&\times|\det[P_{33}+(A^{(2)}-\lambda_{1})(A^{(2)}-\lambda_{1})^{\ast}]\det[P_{44}+(A^{(2)}-\lambda_{2})(A^{(2)}-\lambda_{2})^{\ast}].
\end{align*}
We now use the Schur complement formula to get the first line below, after which we factor out $P_{22}+|A^{(2)}-\lambda_{2}|^{2}$ from the second determinant:
\begin{align*}
&\det[P_{1}\otimes I_{N-4}+(I_{2}\otimes A^{(2)}-\mathbf{w}\otimes I_{N-4})(I_{2}\otimes A^{(2)}-\mathbf{w}\otimes I_{N-4})^{\ast}]\\
&=\det[P_{11}+|A^{(2)}-\lambda_{1}|^{2}]\det[P_{22}+|A^{(2)}-\lambda_{2}|^{2}-P_{12}(P_{11}+|A^{(2)}-\lambda_{1}|^{2})^{-1}P_{12}]\\
&=\det\left[1-|P_{12}|^{2}H_{\lambda_{j}}^{(2)}(P_{11})H_{\lambda_{j}}^{(2)}(P_{22})\right]\det[P_{11}+|A^{(2)}-\lambda_{1}|^{2}]\det[P_{22}+|A^{(2)}-\lambda_{2}|^{2}].
\end{align*}
(Recall $H_{w}^{(2)}(\eta)=[(A-w)(A-w)^{\ast}+\eta^{2}]^{-1}$.) Now, we use the inequality $\det(I+B)\leq e^{\tr B}$ and the lower bounds $H^{(2)}_{\lambda_{j}}(P_{11}),H^{(2)}_{\lambda_{j}}(P_{22})\geq C>0$ (since they are resolvents of covariance matrices with bounded operator norm) to get
\begin{align}
&\det\left[1-|P_{12}|^{2}H_{\lambda_{j}}^{(2)}(P_{11})H_{\lambda_{j}}^{(2)}(P_{22})\right]\label{eq:fisher2x2}\\
&\leq\exp\left\{-|P_{12}|^{2}\tr H_{\lambda_{j}}^{(2)}(P_{11})H_{\lambda_{j}}^{(2)}(P_{22})\right\}\nonumber\\
&=\exp\left\{-|P_{12}|^{2}\tr [H_{\lambda_{j}}^{(2)}(P_{22})]^{\frac12}H_{\lambda_{j}}^{(2)}(P_{11})[H_{\lambda_{j}}^{(2)}(P_{22})]^{\frac12}\right\}\nonumber\\
&\leq\exp\{-CN|P_{12}|^{2}\}.\nonumber
\end{align}
A similar bound holds for $Q$ in place of $P$. Let us clarify that the effect of using Fisher's inequality with one $2\times2$ matrix is to produce the determinant factor in \eqref{eq:fisher2x2}, which produces an a priori exponential bound on the off-diagonal entry of said $2\times2$ matrix. We shortly apply this principle again. We can now follow our computations for the bounds on $\Upsilon_{1}$ and $\Upsilon_{2}$; the only difference is the presence of $\exp\{-CN|P_{12}|^{2}\}$ and $\exp\{-CN|Q_{12}|^{2}\}$. But we restrict to the set $|P_{12}|,|Q_{12}|\geq N^{\delta-1/2}$. Ultimately, we deduce 
\begin{align}
\Upsilon_{3}&\lesssim e^{-CN^{\kappa}}e^{-\frac{2N\eta_{z,t}^{2}}{t}}\det[(A-z)^{\ast}(A-z)+\eta_{z,t}^{2}]^{2}\psi_{1}\psi_{2}\prod_{\ell=1,2}|\det[V_{\ell}^{\ast}G_{z}^{(\ell-1)}(\eta_{z,t})V_{\ell}]|^{2}.
\end{align}
We are left with $\Upsilon_{0}$. We must compute $\mathbf{M}(X)\mathbf{M}(X)^{\ast}$ more precisely as follows:
\begin{align}
&\mathbf{M}(X)\mathbf{M}(X)^{\ast}=\begin{pmatrix}X\otimes I_{N-4}&A_{\mathbf{w}}^{(2)}\\(-A_{\mathbf{w}}^{(2)})^{T}&X^{\ast}\otimes I_{N-4}\end{pmatrix}\begin{pmatrix}X^{\ast}\otimes I_{N-4}&-A^{(2)}_{\overline{\mathbf{w}}}\label{eq:mmstarformulacutoff}\\(A_{\mathbf{w}}^{(2)})^{\ast}&X\otimes I_{N-4}\end{pmatrix}\\
&=\begin{pmatrix}XX^{\ast}\otimes I_{N-4}+A_{\mathbf{w}}^{(2)}[A_{\mathbf{w}}^{(2)}]^{\ast}&-X\otimes I_{N-4}A_{\overline{\mathbf{w}}}^{(2)}+A_{\mathbf{w}}^{(2)}X\otimes I_{N-4}\\-[A_{\overline{\mathbf{w}}}^{(2)}]^{\ast}X^{\ast}\otimes I_{N-4}+X^{\ast}\otimes I_{N-4}[A_{\mathbf{w}}^{(2)}]^{\ast}&X^{\ast}X\otimes I_{N-4}+[A_{\overline{\mathbf{w}}}^{(2)}]^{\ast}A_{\overline{\mathbf{w}}}^{(2)}\end{pmatrix}.\nonumber
\end{align}
In particular, we must compute the off-diagonal blocks. Note that the bottom left block is the adjoint of the upper right block (indeed, $\mathbf{M}(X)\mathbf{M}(X)^{\ast}$ must be Hermitian). So, we compute the top right block. First, since $X\in M_{4}^{skew}(\C)$, we can write 
\begin{align*}
X&=\begin{pmatrix}X^{1}&Y\\-Y^{T}&X^{2}\end{pmatrix}=\begin{pmatrix}X^{1}_{11}&X^{1}_{12}&Y_{11}&Y_{12}\\-X^{1}_{12}&X^{1}_{22}&Y_{21}&Y_{22}\\-Y_{11}&-Y_{21}&X^{2}_{11}&X^{2}_{12}\\-Y_{12}&-Y^{22}&-X^{2}_{12}&X^{2}_{22}\end{pmatrix}
\end{align*}
where $X^{1},X^{2}\in M_{2}^{skew}(\C)$ and $Y\in M_{2\times2}(\C)$. By definition, $A_{\mathbf{w}}^{(2)}X\otimes I_{N-4}$ is equal to
\begin{align*}
&\begin{pmatrix}A^{(2)}-\lambda_{1}&0&0&0\\0&A^{(2)}-\lambda_{2}&0&0\\0&0&A^{(2)}-\overline{\lambda}_{1}&0\\0&0&0&A^{(2)}-\overline{\lambda}_{2}\end{pmatrix}\begin{pmatrix}X^{1}_{11}&X^{1}_{12}&Y_{11}&Y_{12}\\-X^{1}_{12}&X^{1}_{22}&Y_{21}&Y_{22}\\-Y_{11}&-Y_{21}&X^{2}_{11}&X^{2}_{12}\\-Y_{12}&-Y^{22}&-X^{2}_{12}&X^{2}_{22}\end{pmatrix}.
\end{align*}
The top left $2\times2$ block of this $4\times4$ matrix is given by
\begin{align*}
\begin{pmatrix}X_{11}^{1}[A^{(2)}-\lambda_{1}]&X^{1}_{12}[A^{(2)}-\lambda_{1}]\\-X_{12}^{1}[A^{(2)}-\lambda_{2}]&X^{1}_{22}[A^{(2)}-\lambda_{2}]\end{pmatrix}.
\end{align*}
We can compute $X\otimes I_{N-4}A_{\overline{\mathbf{w}}}^{(2)}$ in the exact same way. Its top left $2\times2$ block is equal to
\begin{align*}
\begin{pmatrix}X_{11}^{1}[A^{(2)}-\overline{\lambda}_{1}]&X^{1}_{12}[A^{(2)}-\overline{\lambda}_{2}]\\-X_{12}^{1}[A^{(2)}-\overline{\lambda}_{1}]&X^{1}_{22}[A^{(2)}-\overline{\lambda}_{2}]\end{pmatrix}.
\end{align*}
Thus, the top left $2\times2$ block of $-X\otimes I_{N-4}A_{\overline{\mathbf{w}}}^{(2)}+A_{\mathbf{w}}^{(2)}X\otimes I_{N-4}$ is equal to 
\begin{align}
\begin{pmatrix}X^{1}_{11}[\overline{\lambda}_{1}-\lambda_{1}]&X^{1}_{12}[\overline{\lambda}_{2}-\lambda_{1}]\\-X^{1}_{12}[\overline{\lambda}_{1}-\lambda_{2}]&X^{1}_{22}[\overline{\lambda}_{2}-\lambda_{2}]\end{pmatrix}.\label{eq:topleft2x2}
\end{align}
The top left $2\times 2$ block of $-[A_{\overline{\mathbf{w}}}^{(2)}]^{\ast}X^{\ast}\otimes I_{N-4}+X^{\ast}\otimes I_{N-4}[A_{\mathbf{w}}^{(2)}]^{\ast}$ is just the adjoint of the previous display, as we mentioned earlier. Now, return to \eqref{eq:mmstarformulacutoff}. We apply Fisher's inequality (Lemma 2.2 in \cite{MO}) once again with the following choices for blocks. First, we choose the $2\times2$ block made up of the $(1,1)$ entry of $XX^{\ast}\otimes I_{N-4}+A_{\mathbf{w}}^{(2)}[A_{\mathbf{w}}^{(2)}]^{\ast}$, the $(2,2)$ entry of $X^{\ast}X\otimes I_{N-4}+[A_{\overline{\mathbf{w}}}^{(2)}]^{\ast}A_{\overline{\mathbf{w}}}^{(2)}$, the $(1,2)$ entry of \eqref{eq:topleft2x2}, and the $(2,1)$ entry of its adjoint (in the bottom left block of the second line of \eqref{eq:mmstarformulacutoff}). Precisely, this block is
\begin{align*}
\begin{pmatrix}P_{11}+(A^{(2)}-\lambda_{1})(A^{(2)}-\lambda_{1})^{\ast}&X_{12}^{1}[\overline{\lambda}_{2}-\lambda_{1}]\\\overline{X}_{12}[\lambda_{2}-\overline{\lambda}_{1}]&Q_{22}+(A^{(2)}-\lambda_{1})^{\ast}(A^{(2)}-\lambda_{1})\end{pmatrix}.
\end{align*}
This picks out a $2\times2$ diagonal block in \eqref{eq:mmstarformulacutoff}. Note that its determinant is given by 
\begin{align*}
&\det\begin{pmatrix}P_{11}+(A^{(2)}-\lambda_{1})(A^{(2)}-\lambda_{1})^{\ast}&X_{12}^{1}[\overline{\lambda}_{2}-\lambda_{1}]\\\overline{X}_{12}[\lambda_{2}-\overline{\lambda}_{1}]&Q_{22}+(A^{(2)}-\lambda_{1})^{\ast}(A^{(2)}-\lambda_{1})\end{pmatrix}\\
&=\det[P_{11}+|A^{(2)}-\lambda_{1}|^{2}]\det[Q_{22}+|A^{(2)}-\lambda_{2}|^{2}]\det[1-|X_{12}^{1}|^{2}|\overline{\lambda}_{2}-\lambda_{1}|^{2}H_{\lambda_{j}}^{(2)}(\sqrt{P_{11}})\tilde{H}_{\lambda_{j}}^{(2)}(\sqrt{Q_{22}})].
\end{align*}
For applying Fisher's inequality, we now take the $1\times1$ blocks given by the remaining diagonal entries of the second line in \eqref{eq:mmstarformulacutoff}. We get the following (in which $|L|^{2}=LL^{\ast}$):
\begin{align*}
\det[\mathbf{M}(X)\mathbf{M}(X)^{\ast}]&\leq\prod_{j=1,\ldots,4}|\det[P_{jj}+|A^{(2)}-\lambda_{j}|^{2}]||\det[Q_{jj}+|A^{(2)}-\lambda_{j}|^{2}]|\\
&\times\det[1-|X_{12}^{1}|^{2}|\overline{\lambda}_{2}-\lambda_{1}|^{2}H_{\lambda_{j}}^{(2)}(\sqrt{P_{11}})\tilde{H}_{\lambda_{j}}^{(2)}(\sqrt{Q_{22}})].
\end{align*}
Now, because $P_{11},Q_{11}=O(1)$ by our cutoff in the integration domain of $\Upsilon_{0}$, as in \eqref{eq:fisher2x2}, we know that the second line of the previous display is $\leq\exp[-CN|X_{12}^{1}|^{2}|\overline{\lambda}_{2}-\lambda_{1}|^{2}]$; this follows by the same argument as in \eqref{eq:fisher2x2}. But $\lambda_{j}=z+O(N^{-1/2})$ locally uniformly in $a_{j},b_{j}$, and $z\not\in\R$ is fixed, so $|\overline{\lambda}_{2}-\lambda_{1}|^{2}\gtrsim1$ locally uniformly in $a_{j},b_{j}$, and thus the second line of the previous display is $\leq\exp[-CN|X_{12}^{1}|^{2}]$ locally uniformly in $a_{j},b_{j}$. We can now proceed as in our bound on $\Upsilon_{3}$ to restrict further the integration domain in $\Upsilon_{0}$ to the set where $|X_{12}|\lesssim N^{-1/2+\delta}$ for any small $\delta>0$. We can do the same for every other entry of $X^{1}$ and every entry of $X^{2}$ as well. Ultimately, we have 
\begin{align*}
\Upsilon_{0}=\Upsilon_{\mathrm{main}}+\Upsilon_{4},
\end{align*}
where for some $\kappa>0$ small but fixed, we have 
\begin{align}
\Upsilon_{\mathrm{main}}&:=\int_{\substack{\|P-\eta_{z,t}^{2}\|_{\mathrm{max}}\leq N^{\delta-\frac12}\\\|X^{1}\|_{\max},\|X^{2}\|_{\max}\leq N^{\delta-\frac12}}}e^{-\frac{N}{2t}\tr X^{\ast}X}\pf[\mathbf{M}(X)]dX,\label{eq:upsmain}\\
|\Upsilon_{4}|&\lesssim e^{-CN^{\kappa}}e^{-\frac{2N\eta_{z,t}^{2}}{t}}\det[(A-z)^{\ast}(A-z)+\eta_{z,t}^{2}]^{2}\psi_{1}\psi_{2}\prod_{\ell=1,2}|\det[V_{\ell}^{\ast}G_{z}^{(\ell-1)}(\eta_{z,t})V_{\ell}]|^{2}.\nonumber
\end{align}
In view of the cutoff of $\Upsilon_{\mathrm{main}}$, we now decompose $\mathbf{M}(X)$ as 
\begin{align*}
\mathbf{M}(X)&=\tilde{\mathbf{M}}(Y)+\tilde{\mathbf{M}}(X^{1},X^{2}),
\end{align*}
where $\tilde{\mathbf{M}}(X^{1},X^{2})$ is a diagonal $4\times4$ block matrix with diagonal entries $X^{1}\otimes I_{N-4},X^{2}\otimes I_{N-4},X^{1,\ast}\otimes I_{N-4},X^{2,\ast}\otimes I_{N-4}$, and where $\tilde{\mathbf{M}}(Y)$ is the following block matrix:
{\small
\begin{align*}
\begin{pmatrix}0&Y\otimes I_{N-4}&I_{2}\otimes A^{(2)}-\mathbf{w}\otimes I_{N-4}&0\\-Y^{T}\otimes I_{N-4}&0&0&I_{2}\otimes A^{(2)}-\overline{\mathbf{w}}\otimes I_{N-4}\\-I_{2}\otimes A^{(2),\ast}+\mathbf{w}\otimes I_{N-4}&0&0&-\overline{Y}\otimes I_{N-4}\\0&-I_{2}\otimes A^{(2),\ast}+\overline{\mathbf{w}}\otimes I_{N-4}&Y^{\ast}\otimes I_{N-4}&0\end{pmatrix}
\end{align*}
}Because we take Pfaffian, We can change basis by sending $\mathbf{e}_{1}\mapsto\mathbf{e}_{1},\mathbf{e}_{3}\mapsto\mathbf{e}_{2},\mathbf{e}_{4}\mapsto\mathbf{e}_{3},\mathbf{e}_{2}\mapsto\mathbf{e}_{4}$. In this new basis, we then have 
\begin{align}
|\pf[\mathbf{M}(X)]|&=|\pf[\mathbf{M}(Y)+\mathbf{M}(X^{1},X^{2})]|=|\det[\mathbf{M}(Y)]|^{\frac12}|\det[1+\mathbf{M}(Y)^{-1}\mathbf{M}(X^{1},X^{2})]|^{\frac12},\label{eq:pfaffexpandmy}
\end{align}
where 
\begin{align*}
\mathbf{M}(X^{1},X^{2})&=\begin{pmatrix}X^{1}\otimes I_{N-4}&0&0&0\\0&X^{2,\ast}\otimes I_{N-4}&0\\0&0&X^{2}\otimes I_{N-4}&0\\0&0&0&X^{1,\ast}\otimes I_{N-4}\end{pmatrix},\\
\mathbf{M}(Y)&=\begin{pmatrix}0&M_{0}(Y)\\-M_{0}(Y)^{T}&0\end{pmatrix},\\
M_{0}(Y)&=\begin{pmatrix}Y\otimes I_{N-4}&I_{2}\otimes A^{(2)}-\mathbf{w}\otimes I_{N-4}\\-I_{2}\otimes A^{(2),\ast}+\overline{\mathbf{w}}\otimes I_{N-4}&Y^{\ast}\otimes I_{N-4}\end{pmatrix}.
\end{align*}
We now expand
\begin{align}
&|\det[1+\mathbf{M}(Y)^{-1}\mathbf{M}(X^{1},X^{2})]|^{\frac12}\label{eq:dettoexpmy}\\
&=\exp\left\{\frac12\tr\mathbf{M}(Y)^{-1}\mathbf{M}(X^{1},X^{2})-\frac14\tr[\mathbf{M}(Y)^{-1}\mathbf{M}(X^{1},X^{2})]^{2}+O\left(N^{-\frac12+3\delta+C\epsilon_{0}}\right)\right\}\nonumber\\
&=\exp\left\{-\frac14\tr[\mathbf{M}(Y)^{-1}\mathbf{M}(X^{1},X^{2})]^{2}+O\left(N^{-\frac12+3\delta+C\epsilon_{0}}\right)\right\},\nonumber
\end{align}
where the bound on the third and higher order terms follows by $\|X^{1}\|_{\max},\|X^{2}\|_{\max}\lesssim N^{-1/2+\delta}$ and $\|\mathbf{M}(Y)^{-1}\|_{\mathrm{op}}\lesssim N^{C\epsilon_{0}}$, which we verify immediately below. The second identity follows because $\mathbf{M}(Y)^{-1}\mathbf{M}(X^{1},X^{2})$ has zero diagonal; this can be checked directly. We now collect a few properties of $\mathbf{M}(Y)^{-1}$. First, we define the following modification of $\mathbf{M}(Y)$ obtained by replacing $\mathbf{w}$ by $\mathbf{z}=zI_{2}$:
\begin{align*}
\mathbf{M}_{z}(Y)&:=\begin{pmatrix}0&M_{0,z}(Y)\\-M_{0,z}(Y)^{T}&0\end{pmatrix},\\
M_{0,z}(Y)&:=\begin{pmatrix}Y\otimes I_{N-4}&I_{2}\otimes A^{(2)}-\mathbf{z}\otimes I_{N-4}\\-I_{2}\otimes A^{(2),\ast}+\overline{\mathbf{z}}\otimes I_{N-4}&Y^{\ast}\otimes I_{N-4}\end{pmatrix}.
\end{align*}
An elementary computation shows
\begin{align*}
\mathbf{M}_{z}(Y)^{-1}&=\begin{pmatrix}0&-M_{0,z}(Y)^{-1,T}\\M_{0,z}(Y)^{-1}&0\end{pmatrix}\\
M_{0,z}(Y)^{-1}&=\begin{pmatrix}Y^{\ast}\otimes I_{N-4}{H}^{(2)}_{\mathbf{z}}(Y)&-[I_{2}\otimes A^{(2)}-\mathbf{z}\otimes I_{N-4}]\tilde{H}_{\mathbf{z}}^{(2)}(Y)\\ [I_{2}\otimes A^{(2)}-\mathbf{z}\otimes I_{N-4}]^{\ast}{H}^{(2)}_{\mathbf{z}}(Y)&Y\otimes I_{N-4}\tilde{H}_{\mathbf{z}}^{(2)}(Y)\end{pmatrix},\\
{H}^{(2)}_{\mathbf{z}}(Y)&=[YY^{\ast}\otimes I_{N-4}+(I_{2}\otimes A^{(2)}-\mathbf{z}\otimes I_{N-4})(I_{2}\otimes A^{(2)}-\mathbf{z}\otimes I_{N-4})^{\ast}]^{-1}\\
\tilde{H}_{\mathbf{z}}^{(2)}(Y)&=[Y^{\ast}Y\otimes I_{N-4}+(I_{2}\otimes A^{(2)}-\mathbf{z}\otimes I_{N-4})^{\ast}(I_{2}\otimes A^{(2)}-\mathbf{z}\otimes I_{N-4})]^{-1}.
\end{align*}
Because we restrict to $\|XX^{\ast}-\eta_{z,t}^{2}\|_{\max},\|X^{\ast}X-\eta_{z,t}^{2}\|_{\max},\|X^{1}\|_{\max},\|X^{2}\|_{\max}\lesssim N^{-1/2+\delta}$ for $\delta>0$ small, and because $\eta_{z,t}^{2}\gtrsim N^{-2\epsilon_{0}}$ with $\epsilon_{0}$ small, we know that $\|YY^{\ast}-\eta_{z,t}^{2}\|_{\max},\|Y^{\ast}Y-\eta_{z,t}^{2}\|_{\max}\lesssim N^{-1/2+\delta}$ and thus $\tilde{H}^{(2)}_{\mathbf{z}}(Y),H_{\mathbf{z}}^{(2)}(Y)\lesssim \eta_{z,t}^{-2}\lesssim N^{2\epsilon_{0}}$. This implies $\|\mathbf{M}(Y)^{-1}\|_{\mathrm{op}}\lesssim N^{C\epsilon_{0}}$ for some $C=O(1)$ with high probability (note that $A^{(2)}$ is bounded in operator norm with high probability; indeed, use interlacing to remove the superscript $(2)$ and by the local law in Lemma \ref{1GE}, for example). On the other hand, we know $\|\mathbf{M}_{z}(Y)-\mathbf{M}(Y)\|_{\mathrm{op}}\lesssim N^{-1/2}$ locally uniformly in $z_{1},z_{2}$. Thus, resolvent perturbation shows $\mathbf{M}(Y)^{-1}=\mathbf{M}_{z}(Y)^{-1}[1+O(N^{-1/2+C\epsilon_{0}})]$, where the big-O is in operator norm. Finally, we define the matrices below obtained by evaluating resolvents at $\eta_{z,t}$ instead of $Y$:
\begin{align*}
\mathbf{M}_{z,\eta_{z,t}}(Y)^{-1}&=\begin{pmatrix}0&-M_{0,z,\eta_{z,t}}(Y)^{-1,T}\\M_{0,z,\eta_{z,t}}(Y)^{-1}&0\end{pmatrix}\\
M_{0,z,\eta_{z,t}}(Y)^{-1}&=\begin{pmatrix}Y^{\ast}\otimes I_{N-4}{H}^{(2)}_{\mathbf{z}}(\eta_{z,t})&-[I_{2}\otimes A^{(2)}-\mathbf{z}\otimes I_{N-4}]\tilde{H}_{\mathbf{z}}^{(2)}(\eta_{z,t})\\ [I_{2}\otimes A^{(2)}-\mathbf{z}\otimes I_{N-4}]^{\ast}{H}^{(2)}_{\mathbf{z}}(\eta_{z,t})&Y\otimes I_{N-4}\tilde{H}_{\mathbf{z}}^{(2)}(\eta_{z,t})\end{pmatrix},\\
{H}^{(2)}_{\mathbf{z}}(\eta_{z,t})&=[\eta_{z,t}^{2}+(I_{2}\otimes A^{(2)}-\mathbf{z}\otimes I_{N-4})(I_{2}\otimes A^{(2)}-\mathbf{z}\otimes I_{N-4})^{\ast}]^{-1}\\
\tilde{H}_{\mathbf{z}}^{(2)}(\eta_{z,t})&=[\eta_{z,t}^{2}+(I_{2}\otimes A^{(2)}-\mathbf{z}\otimes I_{N-4})^{\ast}(I_{2}\otimes A^{(2)}-\mathbf{z}\otimes I_{N-4})]^{-1}.
\end{align*}
Before we proceed, we clarify that $\tilde{H}^{(2)}_{\mathbf{z}}(\eta_{z,t})$ is block diagonal with entries $\tilde{H}^{(2)}_{z}(\eta_{z,t})$ and $\tilde{H}^{(2)}_{\overline{z}}(\eta_{z,t})$; a similar statement holds for $H_{\mathbf{z}}^{(2)}(\eta_{z,t})$. Thus, its products with $I_{2}\otimes A^{2}-\mathbf{z}\otimes I_{N-4}$ can be analyzed via the local laws in Lemmas \ref{1GE} and \ref{2GE}. We will use this shortly.

The bounds $\|YY^{\ast}-\eta_{z,t}^{2}\|_{\max},\|Y^{\ast}Y-\eta_{z,t}^{2}\|_{\max}\lesssim N^{-1/2+\delta}$ and resolvent perturbation as before show $\|\mathbf{M}_{z}(Y)-\mathbf{M}_{z,\eta_{z,t}}(Y)\|_{\mathrm{op}}\lesssim N^{-1/2+\delta}$ and $\|\mathbf{M}_{z}(Y)^{-1}-\mathbf{M}_{z,\eta_{z,t}}(Y)^{-1}\|_{\mathrm{op}}\lesssim N^{-1/2+\delta+C\epsilon_{0}}$. Combining this with $\mathbf{M}(Y)^{-1}=\mathbf{M}_{z}(Y)^{-1}[1+O(N^{-1/2+C\epsilon_{0}})]$, we deduce $\|\mathbf{M}(Y)^{-1}-\mathbf{M}_{z,\eta_{z,t}}(Y)^{-1}\|_{\mathrm{op}}\lesssim N^{-1/2+\delta+C\epsilon_{0}}$. Also, by the a priori bounds $\|X_{1}\|_{\max},\|X_{2}\|_{\max}\lesssim N^{-1/2+\delta}$, we also have the estimate $\|\mathbf{M}(X^{1},X^{2})\|_{\mathrm{op}}\lesssim N^{-1/2+\delta}$. Thus, we have 
\begin{align*}
&\exp\left\{-\frac14\left(\tr[\mathbf{M}(Y)^{-1}\mathbf{M}(X^{1},X^{2})]^{2}-\tr[\mathbf{M}_{z,\eta_{z,t}}(Y)^{-1}\mathbf{M}(X^{1},X^{2})]^{2}\right)\right\}\\
&=\exp\left\{-\frac14\tr\mathbf{M}(Y)^{-1}\mathbf{M}(X^{1},X^{2})[\mathbf{M}(Y)^{-1}-\mathbf{M}_{z,\eta_{z,t}}(Y)^{-1}]\mathbf{M}(X^{1},X^{2})\right\}\\
&\times\exp\left\{-\frac14\tr[\mathbf{M}(Y)^{-1}-\mathbf{M}_{z,\eta_{z,t}}(Y)^{-1}]\mathbf{M}(X^{1},X^{2})\mathbf{M}_{z,\eta_{z,t}}(Y)^{-1}\mathbf{M}(X^{1},X^{2})\right\}\\
&=\exp\left\{O(N^{-\frac12+2\delta+C\epsilon_{0}})\right\},
\end{align*}
where $\delta,\epsilon_{0}>0$ are small. In particular, we deduce
\begin{align}
&\exp\left\{-\frac14\tr[\mathbf{M}(Y)^{-1}\mathbf{M}(X^{1},X^{2})]^{2}\right\}\label{eq:mtomzeta}\\
&=\exp\left\{-\frac14\tr[\mathbf{M}_{z,\eta_{z,t}}(Y)^{-1}\mathbf{M}(X^{1},X^{2})]^{2}+O\left(N^{-\frac12+2\delta+C\epsilon_{0}}\right)\right\}.\nonumber
\end{align}
An elementary computation shows the following, in which $A^{(2)}(\mathbf{z}):=I_{2}\otimes A^{(2)}-\mathbf{z}\otimes I_{N-4}$, and in which $Y,X^{1},X^{2}$ are identified with $Y\otimes I_{N-4},X^{1}\otimes I_{N-4},X^{2}\otimes I_{N-4}$, respectively, for convenience:
\begin{align*}
-\frac14\tr[\mathbf{M}_{z,\eta_{z,t}}(Y)^{-1}\mathbf{M}(X^{1},X^{2})]^{2}&=-\frac12\tr Y^{\ast}{H}^{(2)}_{\mathbf{z}}(\eta_{z,t})X^{1}{H}^{(2)}_{\overline{\mathbf{z}}}(\eta_{z,t})\overline{Y}X^{2}\\
&-\frac12\tr A^{(2)}(\mathbf{z})\tilde{H}_{\mathbf{z}}^{(2)}(\eta_{z,t})X^{2,\ast}\tilde{H}_{\overline{\mathbf{z}}}^{(2)}(\eta_{z,t})A^{(2)}(\overline{\mathbf{z}})^{T}X^{2}\\
&-\frac12\tr A^{(2)}(\mathbf{z})^{\ast}{H}^{(2)}_{\mathbf{z}}(\eta_{z,t})X^{1}{H}^{(2)}_{\overline{\mathbf{z}}}(\eta_{z,t})A^{(2)}(\overline{\mathbf{z}})X^{1,\ast}\\
&-\frac12\tr Y\tilde{H}_{\mathbf{z}}^{(2)}(\eta_{z,t})X^{2,\ast}\tilde{H}_{\overline{\mathbf{z}}}^{(2)}(\eta_{z,t})Y^{T}X^{1,\ast}.
\end{align*}
Now, observe that $H_{\mathbf{z}}^{(2)}(\eta_{z,t}),\tilde{H}^{(2)}_{\mathbf{z}}(\eta_{z,t}),H_{\overline{\mathbf{z}}}^{(2)}(\eta_{z,t}),\tilde{H}^{(2)}_{\overline{\mathbf{z}}}(\eta_{z,t}),A^{(2)}(\mathbf{z}),A^{(2)}(\overline{\mathbf{z}}),A^{(2)}(\mathbf{z})^{\ast}$ are all of the form $I_{2}\otimes K$ for some $K\in M_{N-4}(\C)$. Thus, we can factorize the traces on the RHS to get
\begin{align*}
-\frac14\tr[\mathbf{M}_{z,\eta_{z,t}}(Y)^{-1}\mathbf{M}(X^{1},X^{2})]^{2}&=-\frac12\tr {H}^{(2)}_{\mathbf{z}}(\eta_{z,t}){H}^{(2)}_{\overline{\mathbf{z}}}(\eta_{z,t})\tr Y^{\ast}X^{1}\overline{Y}X^{2}\\
&-\frac12\tr A^{(2)}(\mathbf{z})\tilde{H}_{\mathbf{z}}^{(2)}(\eta_{z,t})\tilde{H}_{\overline{\mathbf{z}}}^{(2)}(\eta_{z,t})A^{(2)}(\overline{\mathbf{z}})\tr X^{2,\ast}X^{2}\\
&-\frac12\tr A^{(2)}(\mathbf{z})^{\ast}{H}^{(2)}_{\mathbf{z}}(\eta_{z,t}){H}^{(2)}_{\overline{\mathbf{z}}}(\eta_{z,t})A^{(2)}(\overline{\mathbf{z}})\tr X^{1}X^{1,\ast}\\
&-\frac12\tr \tilde{H}_{\mathbf{z}}^{(2)}(\eta_{z,t})\tilde{H}_{\overline{\mathbf{z}}}^{(2)}(\eta_{z,t})\tr YX^{2,\ast}Y^{T}X^{1,\ast}.
\end{align*}
Next, we compute
\begin{align*}
\exp\left\{-\frac{N}{2t}\tr XX^{\ast}\right\}=\exp\left\{-\frac{N}{2t}\tr X^{1}X^{1,\ast}\right\}\exp\left\{-\frac{N}{2t}\tr X^{2}X^{2,\ast}\right\}\exp\left\{-\frac{N}{t}\tr YY^{\ast}\right\}.
\end{align*}
Therefore, we have 
\begin{align*}
&\exp\left\{-\frac14\tr[\mathbf{M}_{z,\eta_{z,t}}(Y)^{-1}\mathbf{M}(X^{1},X^{2})]^{2}\right\}\exp\left\{-\frac{N}{2t}\tr XX^{\ast}\right\}\\
&=\exp\left\{-\left(\frac{N}{t}+\tr A^{(2)}(\mathbf{z})\tilde{H}_{\mathbf{z}}^{(2)}(\eta_{z,t})\tilde{H}_{\overline{\mathbf{z}}}^{(2)}(\eta_{z,t})A^{(2)}(\overline{\mathbf{z}})\right)\frac12\tr X^{2,\ast}X^{2}\right\}\\
&\times\exp\left\{-\left(\frac{N}{t}+\tr A^{(2)}(\mathbf{z})^{\ast}{H}^{(2)}_{\mathbf{z}}(\eta_{z,t}){H}^{(2)}_{\overline{\mathbf{z}}}(\eta_{z,t})A^{(2)}(\overline{\mathbf{z}})\right)\frac12\tr X^{1,\ast}X^{1}\right\}\\
&\times\exp\left\{-\frac12\tr {H}^{(2)}_{\mathbf{z}}(\eta_{z,t}){H}^{(2)}_{\overline{\mathbf{z}}}(\eta_{z,t})\tr Y^{\ast}X^{1}\overline{Y}X^{2}-\frac12\tr \tilde{H}_{\mathbf{z}}^{(2)}(\eta_{z,t})\tilde{H}_{\overline{\mathbf{z}}}^{(2)}(\eta_{z,t})\tr YX^{2,\ast}Y^{T}X^{1,\ast}\right\}\\
&\times\exp\left\{-\frac{N}{t}\tr YY^{\ast}\right\}.
\end{align*}
Because $X^{1},X^{2}$ are $2\times2$ skew-symmetric, we know that $\tr X^{2,\ast}X^{2}=2|x_{2}|^{2}$ and $\tr X^{1,\ast}X^{1}=2|x_{1}|^{2}$ for $x_{1},x_{2}\in\C$. Moreover, by the a priori estimate $\|Y\|_{\mathrm{op}}\lesssim \eta_{z,t}\lesssim t$ that we restrict to in $\Upsilon_{\mathrm{main}}$, we also know that $\tr Y^{\ast}X^{1}\overline{Y}X^{2}=c_{1}x_{1}x_{2}$ and $\tr YX^{2,\ast}Y^{T}X^{1,\ast}=c_{2}\overline{x}_{1}\overline{x}_{2}$ for constants $c_{1},c_{2}=O(\eta_{z,t}^{2})=t^{2}$. Next, we use Lemma \ref{2GE} to bound from above both $\tr\tilde{H}^{(2)}_{\mathbf{z}}(\eta_{z,t})\tilde{H}^{(2)}_{\overline{\mathbf{z}}}(\eta_{z,t})$ and $\tr H_{\mathbf{z}}^{(2)}(\eta_{z,t})H_{\overline{\mathbf{z}}}^{(2)}(\eta_{z,t})$ by $O(N\eta_{z,t}^{-2})$ with high probability. In particular, we deduce 
\begin{align*}
&\exp\left\{-\frac12\tr {H}^{(2)}_{\mathbf{z}}(\eta_{z,t}){H}^{(2)}_{\overline{\mathbf{z}}}(\eta_{z,t})\tr Y^{\ast}X^{1}\overline{Y}X^{2}-\frac12\tr \tilde{H}_{\mathbf{z}}^{(2)}(\eta_{z,t})\tilde{H}_{\overline{\mathbf{z}}}^{(2)}(\eta_{z,t})\tr YX^{2,\ast}Y^{T}X^{1,\ast}\right\}\\
&=\exp\left\{c_{3}x_{1}x_{2}+c_{4}\overline{x}_{1}\overline{x}_{2}\right\},
\end{align*}
where $c_{3},c_{4}=O(N)$. By combining the previous two displays, we obtain
\begin{align}
&\exp\left\{-\frac14\tr[\mathbf{M}_{z,\eta_{z,t}}(Y)^{-1}\mathbf{M}(X^{1},X^{2})]^{2}\right\}\exp\left\{-\frac{N}{2t}\tr XX^{\ast}\right\}\label{eq:mtquadform}\\
&=\exp\left\{-\begin{pmatrix}x_{1}\\x_{2}\end{pmatrix}^{\ast}\mathbf{T}_{N,t,\mathbf{z}}(\eta_{z,t})\begin{pmatrix}x_{1}\\x_{2}\end{pmatrix}\right\}\nonumber
\end{align}
where $\mathbf{T}_{N,t,\mathbf{z}}(\eta_{z,t})$ is the $2\times2$ matrix
\begin{align*}
\mathbf{T}_{N,t,\mathbf{z}}&=\begin{pmatrix}\mathbf{T}_{N,t,\mathbf{z},11}&\mathbf{T}_{N,t,\mathbf{z},12}\\\mathbf{T}_{N,t,\mathbf{z},21}&\mathbf{T}_{N,t,\mathbf{z},22}\end{pmatrix},\\
\mathbf{T}_{N,t,\mathbf{z},11}&=\frac{N}{t}+\tr A^{(2)}(\mathbf{z})^{\ast}{H}^{(2)}_{\mathbf{z}}(\eta_{z,t}){H}^{(2)}_{\overline{\mathbf{z}}}(\eta_{z,t})A^{(2)}(\overline{\mathbf{z}}),\\
\mathbf{T}_{N,t,\mathbf{z},22}&=\frac{N}{t}+\tr A^{(2)}(\mathbf{z})\tilde{H}_{\mathbf{z}}^{(2)}(\eta_{z,t})\tilde{H}_{\overline{\mathbf{z}}}^{(2)}(\eta_{z,t})A^{(2)}(\overline{\mathbf{z}}),\\
|\mathbf{T}_{N,t,\mathbf{z},12}|&+|\mathbf{T}_{N,t,\mathbf{z},21}|\lesssim N.
\end{align*}
We now combine \eqref{eq:upsmain}, \eqref{eq:pfaffexpandmy}, \eqref{eq:dettoexpmy}, \eqref{eq:mtomzeta}, and \eqref{eq:mtquadform} to deduce
\begin{align}
\Upsilon_{\mathrm{main}}&:=\int_{\substack{\|P-\eta_{z,t}^{2}\|_{\mathrm{max}}\leq N^{\delta-\frac12}\\\|X^{1}\|_{\max},\|X^{2}\|_{\max}\leq N^{\delta-\frac12}}}dY e^{-\frac{N}{t}\tr YY^{\ast}}|\det[\mathbf{M}(Y)]|^{\frac12}\label{eq:upsilonmainreduction}\\
&\times \exp\left\{-\begin{pmatrix}x_{1}\\x_{2}\end{pmatrix}^{\ast}\mathbf{T}_{N,t,\mathbf{z}}(\eta_{z,t})\begin{pmatrix}x_{1}\\x_{2}\end{pmatrix}\right\}dx_{1}dx_{2}\left\{1+O\left(N^{-\frac12+3\delta+C\epsilon_{0}}\right)\right\}.\nonumber
\end{align}
For this, we note the change of variables $dX$ to $dYdx_{1}dx_{2}$ has Jacobian $1$. Because of the constraints $\|X^{1}\|_{\max},\|X^{2}\|_{\max}\leq N^{\delta-1/2}$ above, we can replace $\|P-\eta_{z,t}^{2}\|_{\max}\leq N^{\delta-1/2}$ by $\|YY^{\ast}-\eta_{z,t}^{2}\|_{\max}\leq N^{\delta-1/2}$ in \eqref{eq:upsilonmainreduction}. Next, by Lemma \ref{2GE}, the on-diagonal terms of $\mathbf{T}_{N,t,\mathbf{z}}$ are $\gtrsim Nt^{-1}$. Thus, $\mathbf{T}_{N,t,\mathbf{z}}\gtrsim Nt^{-1}$. Since $t=N^{-\epsilon_{0}}$, we can remove the constraint $\|X^{1}\|_{\max},\|X^{2}\|_{\max}\leq N^{\delta-1/2}$ from \eqref{eq:upsilonmainreduction} by using the Gaussian weight in the second line. In particular, we can perform the Gaussian integration over $x_{1},x_{2}$ and obtain
\begin{align*}
\Upsilon_{\mathrm{main}}&=\int_{\|YY^{\ast}-\eta_{z,t}^{2}\|_{\max}\leq N^{\delta-1/2}}e^{-\frac{N}{t}\tr YY^{\ast}}|\det[\mathbf{M}(Y)]|^{\frac12}|\det\mathbf{T}_{N,t,z}|^{-2}dY\left[1+O(N^{-\kappa})\right]
\end{align*}
for some $\kappa>0$. As mentioned earlier, the on-diagonal entries of $\mathbf{T}_{N,t,z}$ are $\gtrsim N^{-1}$, whereas the off-diagonal entries are $O(N)$. Thus, for some $\kappa>0$, we have $\det\mathbf{T}_{N,t,z}=\mathbf{T}_{N,t,\mathbf{z},11}\mathbf{T}_{N,t,\mathbf{z},22}[1+O(N^{-\kappa})]=\upsilon_{z,t}^{-2}[1+O(N^{-\kappa})]$ (recall $\upsilon_{z,t}$ from before the statement of this lemma). Therefore, we have 
\begin{align}
\Upsilon_{\mathrm{main}}&=\upsilon_{z,t}^{-2}\int_{\|YY^{\ast}-\eta_{z,t}^{2}\|_{\max}\leq N^{\delta-1/2}}e^{-\frac{N}{t}\tr YY^{\ast}}|\det[\mathbf{M}(Y)]|^{\frac12}dY\left[1+O(N^{-\kappa})\right].\label{eq:upsilonbacktocomplexcase}
\end{align}
Now, observe $|\det[\mathbf{M}(Y)]|^{1/2}=\det M_{0}(Y)$. At this point, we proceed as in Section 5 of \cite{MO} to compute the integral in \eqref{eq:upsilonbacktocomplexcase}. Ultimately, after changing variables $Y\mapsto U_{1}SU_{2}^{\ast}$ via singular value decomposition with $U_{1},U_{2}\in\mathbf{U}(2)$ and $S\geq0$ positive semi-definite, we deduce the following (which is the $\tilde{F}_{0}$ formula in Section 5 of \cite{MO}):
\begin{align*}
\Upsilon_{\mathrm{main}}&=\upsilon_{z,t}^{-2}\det[(A-z)^{\ast}(A-z)+\eta_{z,t}^{2}]^{2}e^{-\frac{2N}{t}\eta_{z,t}^{2}}\\
&\times\prod_{j=1,2}|\det[V_{j}^{\ast}G_{z}^{(j-1)}(\eta_{z,t})V_{j}]|^{2}\exp\left\{-\sqrt{\frac{N}{\sigma_{z,t}}}\langle G_{z}(\eta_{z,t})Z_{j}\rangle\right\}\\
&\times|\mathrm{Vol}(\mathbf{U}(2))|^{2}\int_{\mathbf{U}(2)\times\mathbf{U}(2)}d\mu(U_{1},U_{2})e^{-\frac{1}{2\sigma_{z,t}}\sum_{j,k=1,2}\langle G_{z}(\eta_{z,t})Z_{jk}G_{z}(\eta_{z,t})Z_{kj}\rangle}\\
&\times\int_{\R^{2}}\prod_{j=1,2}e^{-2N\gamma_{z,t}(s_{j}-\eta_{z,t})^{2}}e^{-i\frac{1}{\sqrt{N\sigma_{z,t}}}(s_{j}-\eta_{z,t})\tr[G_{z}(\eta_{z,t})^{2}Z_{jj}]}|s_{1}-s_{2}|^{2}ds_{1}ds_{2}\\
&\times\left[1+O(N^{-\frac12+\kappa})\right],
\end{align*}
where $Z$ has $2\times2$ blocks $Z_{jk}$ indexed by $j,k=1,2$, which are defined as
\begin{align*}
Z_{jk}=\begin{pmatrix}0&(U_{1}^{\ast}\mathbf{z}U_{1})_{jk}\\(U_{2}^{\ast}\overline{\mathbf{z}}U_{2})_{jk}&0\end{pmatrix}.
\end{align*}
We omit the details because they follow exactly as in the complex case (the entries of $A$ essentially have nothing to do with this argument, as long as $\|A\|_{\mathrm{op}}=O(1)$). Now, we change variables $\sqrt{4N\gamma_{z,t}}(s_{j}-\eta_{z,t})\mapsto s_{j}$. This introduces a Jacobian of $\frac{1}{4N\gamma_{z,t}}$, and the Vandermonde determinant $|s_{1}-s_{2}|^{2}$ gives another factor of $\frac{1}{4N\gamma_{z,t}}$. Hence, we have 
\begin{align*}
\Upsilon_{\mathrm{main}}&=\frac{1}{16N^{2}\gamma_{z,t}^{2}\upsilon_{z,t}^{2}}\det[(A-z)^{\ast}(A-z)+\eta_{z,t}^{2}]^{2}e^{-\frac{2N}{t}\eta_{z,t}^{2}}\\
&\times\prod_{j=1,2}|\det[V_{j}^{\ast}G_{z}^{(j-1)}(\eta_{z,t})V_{j}]|^{2}\exp\left\{-\sqrt{\frac{N}{\sigma_{z,t}}}\langle G_{z}(\eta_{z,t}Z_{j}\rangle\right\}\\
&\times|\mathrm{Vol}(\mathbf{U}(2))|^{2}\int_{\mathbf{U}(2)\times\mathbf{U}(2)}d\mu(U_{1},U_{2})e^{-\frac{1}{2\sigma_{z,t}}\sum_{j,k=1,2}\langle G_{z}(\eta_{z,t})Z_{jk}G_{z}(\eta_{z,t})Z_{kj}\rangle}\\
&\times\int_{\R^{2}}\prod_{j=1,2}e^{-\frac12s_{j}^{2}}e^{-\frac{i}{2\sqrt{\gamma_{z,t}\sigma_{z,t}}}s_{j}\langle G_{z}(\eta_{z,t})^{2}Z_{jj}\rangle}\left[1+O(N^{-\frac12+C\epsilon_{0}+\kappa})\right]|s_{1}-s_{2}|^{2}ds_{1}ds_{2}.
\end{align*}
Let us first compute $\langle G_{z}(\eta_{z,t})Z_{jk}G_{z}(\eta_{z,t})Z_{kj}\rangle$. To this end, we compute
\begin{align*}
G_{z}(\eta_{z,t})Z_{jk}&=\begin{pmatrix}i\eta_{z,t}H_{z}(\eta_{z,t})&H_{z}(\eta_{z,t})(A-z)\\(A-z)^{\ast}H_{z}(\eta_{z,t})&i\eta_{z,t}\tilde{H}_{z}(\eta_{z,t})\end{pmatrix}\begin{pmatrix}0&(U_{1}^{\ast}\mathbf{z}U_{1})_{jk}\\(U_{2}^{\ast}\overline{\mathbf{z}}U_{2})_{jk}&0\end{pmatrix}\\
&=\begin{pmatrix}(U_{2}^{\ast}\overline{\mathbf{z}}U_{2})_{jk}H_{z}(\eta_{z,t})(A-z)&i\eta_{z,t}(U_{1}^{\ast}\mathbf{z}U_{1})_{jk}H_{z}(\eta_{z,t})\\i\eta_{z,t}(U_{2}^{\ast}\overline{\mathbf{z}}U_{2})_{jk}\tilde{H}_{z}(\eta_{z,t})&(U_{1}^{\ast}\mathbf{z}U_{1})_{jk}(A-z)^{\ast}H_{z}(\eta_{z,t})\end{pmatrix}.
\end{align*}
The same formula holds for $G_{z}(\eta_{z,t})Z_{kj}$ if we swap $j,k$. Hence, we have 
\begin{align*}
&\tr[G_{z}(\eta_{z,t})Z_{jk}G_{z}(\eta_{z,t})Z_{kj}]\\
&=-\eta_{z,t}^{2}(U_{1}^{\ast}\mathbf{z}U_{1})_{jk}(U_{2}^{\ast}\overline{\mathbf{z}}U_{2})_{kj}\tr H_{z}(\eta_{z,t})\tilde{H}_{z}(\eta_{z,t})\\
&+(U_{2}^{\ast}\overline{\mathbf{z}}U_{2})_{jk}(U_{2}^{\ast}\overline{\mathbf{z}}U_{2})_{kj}\tr[(H_{z}(\eta_{z,t})(A-z))^{2}]\\
&-\eta_{z,t}^{2}(U_{2}^{\ast}\overline{\mathbf{z}}U_{2})_{jk}(U_{1}^{\ast}\mathbf{z}U_{1})_{kj}\eta_{z,t}^{2}\tr H_{z}(\eta_{z,t})\tilde{H}_{z}(\eta_{z,t})\\
&+(U_{1}^{\ast}\mathbf{z}U_{1})_{jk}(U_{1}^{\ast}\mathbf{z}U_{1})_{kj}\tr[((A-z)^{\ast}H_{z}(\eta_{z,t}))^{2}]\\
&=-2N\alpha_{z,t}\left[(U_{2}^{\ast}\overline{\mathbf{z}}U_{2})_{jk}(U_{1}^{\ast}\mathbf{z}U_{1})_{kj}+(U_{1}^{\ast}\mathbf{z}U_{1})_{jk}(U_{2}^{\ast}\overline{\mathbf{z}}U_{2})_{kj}\right]\\
&+(U_{2}^{\ast}\overline{\mathbf{z}}U_{2})_{jk}(U_{2}^{\ast}\overline{\mathbf{z}}U_{2})_{kj}\tr[(H_{z}(\eta_{z,t})(A-z))^{2}]\\
&+(U_{1}^{\ast}\mathbf{z}U_{1})_{jk}(U_{1}^{\ast}\mathbf{z}U_{1})_{kj}\tr[((A-z)^{\ast}H_{z}(\eta_{z,t}))^{2}].
\end{align*}
When we sum over $j,k$ and divide by $N$ to take normalized trace, we get
\begin{align*}
&-\frac{1}{2\sigma_{z,t}}\sum_{j,k=1,2}\langle G_{z}(\eta_{z,t})Z_{jk}G_{z}(\eta_{z,t})Z_{jk}\rangle\\
&=\frac{\alpha_{z,t}}{\sigma_{z,t}}\tr U_{2}^{\ast}\overline{\mathbf{z}}U_{2}U_{1}^{\ast}\mathbf{z}U_{1}-\frac{1}{2\sigma_{z,t}}\langle[H_{z}(\eta_{z,t})(A-z)]^{2}\rangle(\overline{z}_{1}^{2}+\overline{z}_{2}^{2})\\
&-\frac{1}{2\sigma_{z,t}}\langle[(A-z)^{\ast}H_{z}(\eta_{z,t})]^{2}\rangle(z_{1}^{2}+z_{2}^{2})\\
&=\frac{\alpha_{z,t}}{\sigma_{z,t}}\tr U_{2}^{\ast}\overline{\mathbf{z}}U_{2}U_{1}^{\ast}\mathbf{z}U_{1}-\sum_{j=1,2}\mathrm{Re}\left(\frac{\overline{\delta}_{z,t}z_{j}^{2}}{\sigma_{z,t}}\right).
\end{align*}
Next, we compute $\langle G_{z}(\eta_{z,t})^{2}Z_{jj}\rangle$. To this end, we first note that 
\begin{align*}
\langle G_{z}(\eta_{z,t})^{2}Z_{jj}\rangle&=(U_{1}^{\ast}\mathbf{z}U_{1})_{jj}\langle[G_{z}(\eta_{z,t})^{2}]_{21}\rangle+(U_{2}^{\ast}\overline{\mathbf{z}}U_{2})_{jj}\langle[G_{z}(\eta_{z,t})^{2}]_{12}\rangle,
\end{align*}
so only the off-diagonal blocks of $G_{z}(\eta_{z,t})^{2}$ matter. These off-diagonal blocks are 
\begin{align*}
[G_{z}(\eta_{z,t})^{2}]_{12}&=i\eta_{z,t}H_{z}(\eta_{z,t})^{2}(A-z)+i\eta_{z,t}H_{z}(\eta_{z,t})(A-z)\tilde{H}_{z}(\eta_{z,t})\\
&=2i\eta_{z,t}H_{z}(\eta_{z,t})^{2}(A-z),\\
[G_{z}(\eta_{z,t})^{2}]_{21}&=i\eta_{z,t}(A-z)^{\ast}H_{z}(\eta_{z,t})^{2}+i\eta_{z,t}\tilde{H}_{z}(\eta_{z,t})(A-z)^{\ast}H_{z}(\eta_{z,t})\\
&=2i\eta_{z,t}(A-z)^{\ast}H_{z}(\eta_{z,t})^{2}.
\end{align*}
Note that $\langle[G_{z}(\eta_{z,t})^{2}]_{12}\rangle=\overline{\langle[G_{z}(\eta_{z,t})^{2}]_{21}\rangle}=2i\beta_{z,t}$ by invariance of trace under transpose since $H_{z}$ is self-adjoint. We deduce
\begin{align*}
\langle G_{z}(\eta_{z,t})^{2}Z_{jj}\rangle&=2i\beta_{z,t}(U_{2}^{\ast}\overline{\mathbf{z}}U_{2})_{jj}+2i\overline{\beta}_{z,t}(U_{1}^{\ast}\mathbf{z}U_{1})_{jj}.
\end{align*}
We can now rewrite the $\R^{2}$ integration in $\Upsilon_{\mathrm{main}}$ as follows (where $\mathbf{s}=(s_{1},s_{2})$):
\begin{align*}
&\int_{\R^{2}}\prod_{j=1,2}e^{-\frac12s_{j}^{2}}e^{\frac{1}{\sqrt{\gamma_{z,t}\sigma_{z,t}}}s_{j}[\beta_{z,t}(U_{2}^{\ast}\mathbf{z}U_{2})_{jj}+\overline{\beta}_{z,t}(U_{1}^{\ast}\mathbf{z}U_{1})_{jj}]}|s_{1}-s_{2}|^{2}ds_{1}ds_{2}\\
&=\int_{\R^{2}}e^{-\frac12\mathbf{s}^{2}}e^{\frac{1}{\sqrt{\gamma_{z,t}\sigma_{z,t}}}\tr\mathbf{s}[\beta_{z,t}(U_{2}^{\ast}\overline{\mathbf{z}}U_{2})+\overline{\beta}_{z,t}(U_{1}^{\ast}\mathbf{z}U_{1})]}|s_{1}-s_{2}|^{2}ds_{1}ds_{2}.
\end{align*}
Note that the $U_{1},U_{2}$ dependence in the integrand of this $\R^{2}$ integral is invariant under right translation by unitary matrices. Thus, we can replace $U_{1}$ by $U_{1}V$ and $U_{2}$ by $U_{2}V$, and average over $V\in\mathbf{U}(2)$ according to Haar measure. Then, with a change-of-variables factor of $\frac{|\mathrm{Vol}(\mathbf{U}(2))|}{8\pi^{2}}|s_{1}-s_{2}|^{2}$, we can replace integration on $\R^{2}\times\mathbf{U}(2)$ by integration on the space $M_{2}^{sa}(\C)$ of Hermitian $2\times2$ matrices. Ultimately, we deduce that the integral in the previous display is equal to
\begin{align*}
&\frac{8\pi^{2}}{|\mathrm{Vol}(\mathbf{U}(2))|}\int_{M_{2}^{sa}(\C)}e^{-\frac12\tr Q^{2}}e^{\frac{1}{\sqrt{\gamma_{z,t}\sigma_{z,t}}}\tr[\beta_{z,t}(U_{2}^{\ast}\overline{\mathbf{z}}U_{2})+\overline{\beta}_{z,t}(U_{1}^{\ast}\mathbf{z}U_{1})]Q}dQ\\
&=\frac{16\pi^{4}}{|\mathrm{Vol}(\mathbf{U}(2))|}\exp\left\{\frac{|\beta_{z,t}|^{2}}{\gamma_{z,t}\sigma_{z,t}}\tr[U_{2}^{\ast}\overline{\mathbf{z}}U_{2}U_{1}^{\ast}\mathbf{z}U_{1}]\right\}\\
&\times\exp\left\{\frac{\beta_{z,t}^{2}}{2\gamma_{z,t}\sigma_{z,t}}\tr[U_{2}^{\ast}\overline{\mathbf{z}}U_{2}U_{2}^{\ast}\overline{\mathbf{z}}U_{2}]+\frac{\overline{\beta}_{z,t}^{2}}{2\gamma_{z,t}\sigma_{z,t}}\tr[U_{1}^{\ast}\mathbf{z}U_{1}U_{1}^{\ast}\mathbf{z}U_{1}]\right\}\\
&=\frac{16\pi^{4}}{|\mathrm{Vol}(\mathbf{U}(2))|}\exp\left\{\left(\frac{\alpha_{z,t}}{\sigma_{z,t}}+\frac{|\beta_{z,t}|^{2}}{\gamma_{z,t}\sigma_{z,t}}\right)\tr[U_{2}^{\ast}\overline{\mathbf{z}}U_{2}U_{1}^{\ast}\mathbf{z}U_{1}]+\sum_{j=1,2}\mathrm{Re}\left(\frac{\overline{\beta}_{z,t}^{2}z_{j}^{2}}{\gamma_{z,t}\sigma_{z,t}}\right)\right\}
\end{align*}
where the integral is computed by Gaussian integration. Putting it altogether to compute $\Upsilon_{\mathrm{main}}$, we deduce 
\begin{align*}
\Upsilon_{\mathrm{main}}&=\frac{1+O(N^{-\kappa})}{16N^{2}\gamma_{z,t}^{2}\upsilon_{z,t}^{2}}\det[(A-z)^{\ast}(A-z)+\eta_{z,t}^{2}]^{2}e^{-\frac{2N}{t}\eta_{z,t}^{2}}\\
&\times\prod_{j=1,2}|\det[V_{j}^{\ast}G_{z}^{(j-1)}(\eta_{z,t})V_{j}]|^{2}\exp\left\{-\sqrt{\frac{N}{\sigma_{z,t}}}\langle G_{z}(\eta_{z,t}Z_{j}\rangle\right\}\\
&\times16\pi^{4}|\mathrm{Vol}(\mathbf{U}(2))|e^{\frac{\alpha_{z,t}}{\sigma_{z,t}}-\sum_{j=1,2}\mathrm{Re}\left(\frac{\overline{\delta}_{z,t}z_{j}^{2}}{\sigma_{z,t}}\right)}e^{\sum_{j=1,2}\mathrm{Re}\left(\frac{\overline{\beta}_{z,t}^{2}z_{j}^{2}}{\gamma_{z,t}\sigma_{z,t}}\right)}\\
&\times\int_{\mathbf{U}(2)\times\mathbf{U}(2)}d\mu(U_{1},U_{2})\exp\left\{\left(\frac{\alpha_{z,t}}{\sigma_{z,t}}+\frac{|\beta_{z,t}|^{2}}{\gamma_{z,t}\sigma_{z,t}}\right)\tr[U_{2}^{\ast}\overline{\mathbf{z}}U_{2}U_{1}^{\ast}\mathbf{z}U_{1}]\right\}.
\end{align*}
Now, recall that $\alpha_{z,t}+|\beta_{z,t}|^{2}\gamma_{z,t}^{-1}=\sigma_{z,t}$, so in the last line, the factor in front of the trace is $1$. So, we can combine terms and obtain the following, in which the last two identities are proven by the Harish-Chandra-Itzykson-Zuber formula as in \cite{MO}, definition of $\rho_{\mathrm{GinUE}}^{(2)}$, and $|\mathrm{Vol}(\mathbf{U}(2))|=(2\pi)^{3}$:
\begin{align*}
\Upsilon_{\mathrm{main}}&=\frac{\pi^{4}[1+O(N^{-\kappa})]}{N^{2}\gamma_{z,t}^{2}\upsilon_{z,t}^{2}}\det[(A-z)^{\ast}(A-z)+\eta_{z,t}^{2}]^{2}e^{-\frac{2N}{t}\eta_{z,t}^{2}}\\
&\times\prod_{j=1,2}|\det[V_{j}^{\ast}G_{z}^{(j-1)}(\eta_{z,t})V_{j}]|^{2}\psi_{j}e^{-z_{j}^{2}}\\
&\times|\mathrm{Vol}(\mathbf{U}(2))|\int_{\mathbf{U}(2)\times\mathbf{U}(2)}d\mu(U_{1},U_{2})\exp\left\{\tr[U_{2}^{\ast}\overline{\mathbf{z}}U_{2}U_{1}^{\ast}\mathbf{z}U_{1}]\right\}\\
&=\frac{\pi^{4}[1+O(N^{-\kappa})]}{N^{2}\gamma_{z,t}^{2}\upsilon_{z,t}^{2}}\det[(A-z)^{\ast}(A-z)+\eta_{z,t}^{2}]^{2}e^{-\frac{2N}{t}\eta_{z,t}^{2}}|\mathrm{Vol}(\mathbf{U}(2))|\frac{\det[e^{z_{i}\overline{z}_{j}}]}{|\Delta(\mathbf{z})|^{2}}\\
&\times\prod_{j=1,2}|\det[V_{j}^{\ast}G_{z}^{(j-1)}(\eta_{z,t})V_{j}]|^{2}\psi_{j}e^{-z_{j}^{2}}\\
&=\frac{8\pi^{7}}{N^{2}\gamma_{z,t}^{2}\upsilon_{z,t}^{2}}\det[(A-z)^{\ast}(A-z)+\eta_{z,t}^{2}]^{2}e^{-\frac{2N}{t}\eta_{z,t}^{2}}\frac{\rho_{\mathrm{GinUE}}^{(2)}(z_{1},z_{2})}{\Delta(\mathbf{z})^{2}}\\
&\times\prod_{j=1,2}|\det[V_{j}^{\ast}G_{z}^{(j-1)}(\eta_{z,t})V_{j}]|^{2}\psi_{j}.
\end{align*}
Now, by using our error estimates for $\Upsilon_{1},\Upsilon_{2},\Upsilon_{3},\Upsilon_{4}$ and trivial bounds $\gamma_{z,t}\lesssim t$ and $\upsilon_{z,t}\lesssim N^{D}$ for some $D=O(1)$, we have
\begin{align}
\Delta(\mathbf{z})^{2}\int_{M_{2\times2}(\C)}e^{-\frac{N}{t}\tr XX^{\ast}}\det[iM(X)]dX&=\Delta(\mathbf{z})^{2}\Upsilon_{\mathrm{main}}(1+O(N^{-\kappa}))
\end{align}
locally uniformly in $z_{1},z_{2}$, at which point we multiply by $\frac{4N{|z-\overline{z}|^{2}}}{t^{4}\pi^{4}\sigma_{z,t}^{3}}$ to finish the proof.
\end{proof}

\subsection{Fourier transform computation for \texorpdfstring{$\tilde{K}_{j}(z_{j})$}{K\_j(z\_j)}}
We require the following analog of Lemma 2.3 in \cite{MO} for real matrices and $V^{2}(\R^{N})$.
\begin{lemma}\label{lemma:fourier}
Suppose $f\in L^{1}(M_{2\times N}(\R))$is continuous in a neighborhood of $V^{2}(\R^{N})=\mathbf{O}(N)/\mathbf{O}(N-2)$. Let $M_{2}^{sa}(\R)$ be the space of symmetric $2\times2$ matrices, and define
\begin{align}
\hat{f}(P):=\sqrt{\frac{1}{2\pi^{7}}}\int_{M_{2\times N}(\R)}e^{-i\tr[PXX^{\ast}]}f(X)dX, \quad P\in M_{2}^{sa}(\R).
\end{align}
Then we have 
\begin{align}
\int_{V^{2}(\R^{N})}f(\mathbf{v})d\mathbf{v}=\int_{M_{2}^{sa}(\R)}e^{i\tr P}\hat{f}(P)dP.
\end{align}
We recall that $d\mathbf{v}$ is integration with respect to the rotationally invariant volume form on $V^{2}(\R^{N})$.
\end{lemma}
\begin{proof}
We first define
\begin{align*}
I_{\epsilon}(f):=\frac{1}{2\pi^{2}\epsilon^{2}}\int_{M_{2\times N}(\R)}e^{-\frac{1}{2\epsilon}\tr[(XX^{\ast}-I_{2})^{2}]}f(X)dX.
\end{align*}
We use change of variables $X=OP^{\frac12}$, where $O\in V^{2}(\R^{N})$  is a matrix of size $2\times N$ and $P$ is a positive definite matrix of size $N\times N$. This change of variables has Jacobian $dX=\frac14\det(P)^{\frac{N-3}{2}}dPd\mu_{N,2}(O)$, where $\mu_{N,2}$ is the rotationally invariant volume form on $V^{2}(\R^{N})$; see Proposition 4 in \cite{DGGJ}. Thus,
\begin{align*}
I_{\epsilon}(f)&=\frac{1}{2\pi^{2}\epsilon^{2}}\int_{P\geq0}e^{-\frac{1}{2\epsilon}\tr[(P-I_{2})^{2}]}g(P)dP,\\
g(P)&:=\frac14\det(P)^{\frac{N-3}{2}}\int_{V^{2}(\R^{N})}f(OP^{\frac12})d\mu_{N,2}(O).
\end{align*}
We take the $dP$ integral to be over $2\times 2$ positive definite matrices since all but two eigenvalues of $P$ are $0$. After doing so, $(2\pi^{2}\epsilon^{2})^{-1}\exp\{-\frac{1}{2\epsilon}\tr[(P-I_{2})^{2}]\}$ is an approximation to the delta function at $P=I_{2}$. Indeed, first, shift $\tilde{P}:=P-I_{2}$; the $\pi^{-2}\epsilon^{-2}$ comes from scaling the matrix entries of $\tilde{P}$ to be Gaussians which integrate to $1$. Recall $g$ is continuous at $P=I_{2}$. Taking $\epsilon\to0$, we get
\begin{align*}
I_{\epsilon}(f)\to_{\epsilon\to0}\frac14\det(P)^{\frac{N-3}{2}}\int_{V^{2}(\R^{N})}f(O)d\mu_{N,2}(O),
\end{align*}
which we can formally rewrite by replacing $O$ by $\mathbf{v}$ and $d\mu_{N,2}(O)$ by $d\mathbf{v}$. We now compute $I_{\epsilon}(f)$ differently; by a Hubbard-Stratonovich transform, we have 
\begin{align*}
I_{\epsilon}(f)&=\frac{1}{2\pi^{2}\epsilon^{2}}\int_{M_{2\times N}(\R)}\frac{1}{(2\pi)^{\frac32}}\int_{M_{2}^{sa}(\R)}e^{-\frac12\tr P^{2}+\frac{i}{\sqrt{\epsilon}}\tr[P(I_{2}-XX^{\ast})]}f(X)dXdP\\
&=\frac{1}{2^{\frac52}\pi^{\frac72}}\int_{M_{2}^{sa}(\R)}e^{-\frac{\epsilon}{2}\tr P^{2}+i\tr P}\int_{M_{2\times N}(\R)}e^{-i\tr[PXX^{\ast}]}f(X)dXdP,
\end{align*}
at which point we conclude by definition of $\hat{f}(P)$.
\end{proof}
With this in hand, we can now start to compute $\tilde{K}_{j}(z_{j})$. First, we introduce some notation:
\begin{align*}
A_{a_{j},b_{j},\theta_{j}}^{(j-1)}&:=I_{2}\otimes A^{(j-1)}-\Lambda_{a_{j},b_{j},\theta_{j}}\otimes I_{N-2j+2},\\
C_{a_{j},b_{j},\theta_{j}}^{(j-1)}&:=(A_{a_{j},b_{j},\theta_{j}}^{(j-1)})^{\ast}A_{a_{j},b_{j},\theta_{j}}^{(j-1)},\\
\tilde{H}_{a_{j},b_{j},\theta_{j}}^{(j-1)}(\eta)&:=[(A_{a_{j},b_{j},\theta_{j}}^{(j-1)})^{\ast}A_{a_{j},b_{j},\theta_{j}}^{(j-1)}+\eta^{2}]^{-1},\\
\tilde{H}_{a_{j},b_{j},\theta_{j}}^{(j-1)}&:=\tilde{H}_{a_{j},b_{j},\theta_{j}}^{(j-1)}(\eta_{z,t}).
\end{align*}
%
\begin{lemma}\label{lemma:tildekjcomputation}
We have 
\begin{align*}
&\tilde{K}_{j}(z_{j})=\frac{N^{2j+1}t^{-2j-1}e^{\frac{N}{t}\eta_{z,t}^{2}}}{2^{2j+\frac12}\pi^{2j+\frac32}}\det[(A_{a_{j},b_{j},\theta_{j}}^{(j-1)})^{\ast}A_{a_{j},b_{j},\theta_{j}}^{(j-1)}+\eta_{z,t}^{2}]^{-\frac12}\\
&\times\int_{M_{2}^{sa}(\R)}e^{i\frac{N}{2t}\tr P}\det[I+i[\tilde{H}_{a_{j},b_{j},\theta_{j}}^{(j-1)}]^{\frac12}(P\otimes I_{N-2j+2})[\tilde{H}_{a_{j},b_{j},\theta_{j}}^{(j-1)}]^{\frac12}]^{-\frac12}dP.
\end{align*}
\end{lemma}
\begin{proof}
By definition, we have 
\begin{align*}
\tilde{K}_{j}(z_{j})&=\left(\frac{N}{2\pi t}\right)^{N}\int_{V^{2}(\R^{N-2j+2})}\exp\left\{-\frac{N}{2t}\mathbf{v}^{\ast}C_{a_{j},b_{j},\theta_{j}}^{(j-1)}\mathbf{v}\right\}d\mathbf{v}\\
&=\left(\frac{N}{2\pi t}\right)^{N}e^{\frac{N\eta_{z,t}^{2}}{t}}\int_{V^{2}(\R^{N-2j+2})}\exp\left\{-\frac{N}{2t}\mathbf{v}^{\ast}[C_{a_{j},b_{j},\theta_{j}}^{(j-1)}+\eta_{z,t}^{2}]\mathbf{v}\right\}d\mathbf{v}.
\end{align*}
By Lemma \ref{lemma:fourier} and a change-of-variables $P\mapsto\frac{N}{2t}P$ in the $dP$ integral therein, we know
\begin{align*}
\tilde{K}_{j}(z_{j})&=\left(\frac{N}{2\pi t}\right)^{N}\left(\frac{N}{2t}\right)^{3}\sqrt{\frac{1}{2\pi^{7}}}e^{\frac{N}{t}\eta_{z,t}^{2}}\\
&\times\int_{M_{2}^{sa}(\R)}e^{i\frac{N}{2t}\tr P}\int_{\R^{2N-4j+4}}e^{-\frac{N}{2t}\mathbf{x}^{\ast}[C_{a_{j},b_{j},\theta_{j}}^{(j-1)}+\eta_{z,t}^{2}+iP\otimes I_{N-2j+2}]\mathbf{x}}d\mathbf{x}dP.
\end{align*}
We now perform the $d\mathbf{x}$ integral explicitly since it is a Gaussian. The previous display thus equals
\begin{align*}
&\left(\frac{2\pi t}{N}\right)^{N-2j+2}\left(\frac{N}{2\pi t}\right)^{N}\left(\frac{N}{2t}\right)^{3}\sqrt{\frac{1}{2\pi^{7}}}e^{\frac{N}{t}\eta_{z,t}^{2}}\\
&\times\int_{M_{2}^{sa}(\R)}e^{i\frac{N}{2t}\tr P}\det[(A_{a_{j},b_{j},\theta_{j}}^{(j-1)})^{\ast}A_{a_{j},b_{j},\theta_{j}}^{(j-1)}+\eta_{z,t}^{2}+iP\otimes I_{N-2j+2}]^{-\frac12}dP\\
&=2^{-2j-\frac32}\pi^{-2j-\frac32}N^{2j+1}t^{-2j-1}e^{\frac{N}{t}\eta_{z,t}^{2}}\\
&\times\int_{M_{2}^{sa}(\R)}e^{i\frac{N}{2t}\tr P}\det[(A_{a_{j},b_{j},\theta_{j}}^{(j-1)})^{\ast}A_{a_{j},b_{j},\theta_{j}}^{(j-1)}+\eta_{z,t}^{2}+iP\otimes I_{N-2j+2}]^{-\frac12}dP.
\end{align*}
By factoring out $\tilde{H}_{a_{j},b_{j},\theta_{j}}^{(j-1)}$ from the last determinant, the proof follows.
\end{proof}
Combining our computations thus far for $\tilde{F}(\mathbf{z};A^{(2)})$ and $\tilde{K}_{j}(z_{j})$ yields the following in which $\kappa>0$ is independent of $N$:
\begin{align}
  \frac{\rho_{t}(z,\mathbf{z};A)}{\rho_{\mathrm{GinUE}}^{(2)}(z_{1},z_{2})} &=\frac{32\pi^{3}e^{-\frac{2N}{t}\eta_{z,t}^{2}}[1+O(N^{-\kappa})]}{Nt^{4}\gamma_{z,t}^{2}\sigma_{z,t}^{3}\upsilon_{z,t}^{2}} \int_{[0,\frac{\pi}{2}]^{2}}d\theta_{1}d\theta_{2}\frac{256b_{1}^{2}b_{2}^{2}|\cos2\theta_{1}||\cos2\theta_{2}|}{|\sin^{2}2\theta_{1}|\sin^{2}2\theta_{2}|}\label{eq:prelimcomputerho}\\
  \times\prod_{j=1,2}&\int_{V^{2}(\R^{N-2j+2})}d\mu_{j}(\mathbf{v}_{j})|\det[V_{j}^{\ast}G_{z}^{(j-1)}(\eta_{z,t})V_{j}]|^{2}\nonumber\\
 \times\prod_{j=1,2}&\psi_{j}\det[(A-z)^{\ast}(A-z)+\eta_{z,t}^{2}]\det[(A_{a_{j},b_{j},\theta_{j}}^{(j-1)})^{\ast}A_{a_{j},b_{j},\theta_{j}}^{(j-1)}+\eta_{z,t}^{2}]^{-\frac12}\nonumber\\
\times \prod_{j=1,2}&\frac{N^{2j+1}t^{-2j-1}e^{\frac{N}{t}\eta_{z,t}^{2}}}{2^{2j+\frac32}\pi^{2j+\frac32}}\int_{M_{2}^{sa}(\R)}e^{i\frac{N}{2t}\tr P}\det[I+i[\tilde{H}_{a_{j},b_{j},\theta_{j}}^{(j-1)}]^{\frac12}(P\otimes I)[\tilde{H}_{a_{j},b_{j},\theta_{j}}^{(j-1)}]^{\frac12}]^{-\frac12}dP\nonumber\\
&=\frac{N^{7}b^{4}[1+O(N^{-\kappa})]}{16\pi^{4}t^{12}\gamma_{z,t}^{2}\sigma_{z,t}^{3}\upsilon_{z,t}^{2}}\int_{[0,\frac{\pi}{2}]^{2}}d\theta_{1}d\theta_{2}\frac{256|\cos2\theta_{1}||\cos2\theta_{2}|}{|\sin^{2}2\theta_{1}|\sin^{2}2\theta_{2}|}\nonumber\\
\times\prod_{j=1,2}&\int_{V^{2}(\R^{N-2j+2})}d\mu_{j}(\mathbf{v}_{j})|\det[V_{j}^{\ast}G_{z}^{(j-1)}(\eta_{z,t})V_{j}]|^{2}\nonumber\\
\times\prod_{j=1,2}&\psi_{j}\det[(A-z)^{\ast}(A-z)+\eta_{z,t}^{2}]\det[(A_{a_{j},b_{j},\theta_{j}}^{(j-1)})^{\ast}A_{a_{j},b_{j},\theta_{j}}^{(j-1)}+\eta_{z,t}^{2}]^{-\frac12}\nonumber\\
\times\prod_{j=1,2}&\int_{M_{2}^{sa}(\R)}e^{i\frac{N}{2t}\tr P}\det[I+i[\tilde{H}_{a_{j},b_{j},\theta_{j}}^{(j-1)}]^{\frac12}(P\otimes I)[\tilde{H}_{a_{j},b_{j},\theta_{j}}^{(j-1)}]^{\frac12}]^{-\frac12}dP,\nonumber
\end{align}
where the second identity follows by combining factors of $N,t,\pi,2$, viewing $V_{j}$ as a function of $\mathbf{v}_{j}$ in the last line above, and noting that $b_{1},b_{2}=b+O(N^{-1/2})$ (recall $b>0$ is positive). 
%
%
%
\section{Estimating the integral over \texorpdfstring{$\theta_{1},\theta_{2}$}{theta1, theta2}}
Recall $A_{a_{j},b_{j},\theta_{j}}^{(j-1)}$ from before Lemma \ref{lemma:tildekjcomputation}. Let $\Theta$ denote the $d\theta_{1}d\theta_{2}$ integration of interest:
\begin{align*}
&\Theta:=\int_{[0,\frac{\pi}{2}]^{2}}d\theta_{1}d\theta_{2}\frac{256|\cos2\theta_{1}||\cos2\theta_{2}|}{|\sin^{2}2\theta_{1}|\sin^{2}2\theta_{2}|}\\
&\times\prod_{j=1,2}\psi_{j}\det[(A-z)^{\ast}(A-z)+\eta_{z,t}^{2}]\det[(A_{a_{j},b_{j},\theta_{j}}^{(j-1)})^{\ast}A_{a_{j},b_{j},\theta_{j}}^{(j-1)}+\eta_{z,t}^{2}]^{-\frac12}\\
&\times\prod_{j=1,2}\int_{M_{2}^{sa}(\R)}e^{i\frac{N}{2t}\tr P}\det[I+i[\tilde{H}_{a_{j},b_{j},\theta_{j}}^{(j-1)}]^{\frac12}(P\otimes I)[\tilde{H}_{a_{j},b_{j},\theta_{j}}^{(j-1)}]^{\frac12}]^{-\frac12}dP.
\end{align*}
Technically, $\Theta$ depends on $\mathbf{v}_{1},\mathbf{v}_{2}$ (the integration variables in $V^{2}(\R^{N-2j+2})$ for $j=1,2$), though this dependence will not play a role in this section. Now, for $\tau>0$ small and fixed (depending only on $\epsilon_{0}$), define 
\begin{align*}
\mathcal{I}_{0}&=\left\{\theta: |\theta-\pi/4|\le N^{-1/2+\tau}\right\} \\
\mathcal{I}_{1}&=\left[0  ,\frac{\pi}{2}\right]\setminus {\cal I}_0.
\end{align*}
We write $$\Theta=\sum_{j,k=0,1}\Theta_{jk},$$ where $\Theta_{jk}$ is the same as $\Theta$, but the integration in $d\theta_{1}d\theta_{2}$ is over $\theta_{1}\in\mathcal{I}_{j}$ and $\theta_{2}\in\mathcal{I}_{k}$. Our goal in this section is the following. 
\begin{prop}\label{prop:cutoff}
There exists $\kappa>0$ independent of $N$ such that if $(j,k)\neq(0,0)$, then $\Theta_{jk}=O(e^{-N^{\kappa}})$ {locally uniformly in $a_{j},b_{j}$.}
\end{prop}
As an immediate consequence of this result, we have the following estimate.
\begin{corollary}\label{corollary:cutoff}
First, define 
\begin{align*}
&\rho_{\mathrm{main}}(z,\mathbf{z};A):=\frac{N^{7}b^{4}[1+O(N^{-\kappa})]}{16\pi^{4}t^{12}\gamma_{z,t}^{2}\sigma_{z,t}^{3}\upsilon_{z,t}^{2}}\rho_{\mathrm{GinUE}}^{(2)}(z_{1},z_{2})\int_{\mathcal{I}_{0}^{2}}d\theta_{1}d\theta_{2}\frac{256|\cos2\theta_{1}||\cos2\theta_{2}|}{|\sin^{2}2\theta_{1}|\sin^{2}2\theta_{2}|}\\
&\times\prod_{j=1,2}\int_{V^{2}(\R^{N-2j+2})}d\mu_{j}(\mathbf{v}_{j})|\det[V_{j}^{\ast}G_{z}^{(j-1)}(\eta_{z,t})V_{j}]|^{2}\\
&\times\prod_{j=1,2}\psi_{j}\det[(A-z)^{\ast}(A-z)+\eta_{z,t}^{2}]\det[(A_{a_{j},b_{j},\theta_{j}}^{(j-1)})^{\ast}A_{a_{j},b_{j},\theta_{j}}^{(j-1)}+\eta_{z,t}^{2}]^{-\frac12}\\
&\times\prod_{j=1,2}\int_{M_{2}^{sa}(\R)}e^{i\frac{N}{2t}\tr P}\det[I+i[\tilde{H}_{a_{j},b_{j},\theta_{j}}^{(j-1)}]^{\frac12}(P\otimes I)[\tilde{H}_{a_{j},b_{j},\theta_{j}}^{(j-1)}]^{\frac12}]^{-\frac12}dP.
\end{align*}
There exists $\kappa>0$ independent of $N$ such that 
\begin{align}
\rho_{t}(z,\mathbf{z};A)&=\rho_{\mathrm{main}}(z,\mathbf{z};A)+O(e^{-N^{\kappa}}).
\end{align}
\end{corollary}
\subsection{Preliminary bounds for integrating on \texorpdfstring{$\mathcal{I}_{1}$}{I1}}
The crux of this subsection is control integrating $\theta_{j}\in\mathcal{I}_{1}$; in this region, we do not have local laws for $A_{a_{j},b_{j},\theta_{j}}^{(j-1)}$ and its related resolvents since the operator norm of $A_{a_{j},b_{j},\theta_{j}}^{(j-1)}$ blows up as $\theta\to0,\frac{\pi}{2}$. The first step we take is the following bound on the product of determinants.
\begin{lemma}\label{lemma:cutoff1}
There exists a constant $C>0$ such that for any $j\geq1$, we have the following {locally uniformly in $a_{j},b_{j}$}:
\begin{align*}
&\psi_{j}\det[(A-z)^{\ast}(A-z)+\eta_{z,t}^{2}]\det[(A_{a_{j},b_{j},\theta_{j}}^{(j-1)})^{\ast}A_{a_{j},b_{j},\theta_{j}}^{(j-1)}+\eta_{z,t}^{2}]^{-\frac12}\\
&\lesssim N^{C\epsilon_{0}}\exp\left\{-Cb_{j}^{2}N\eta_{z,t}^{2}[\tan\theta_{j}-\tan^{-1}\theta_{j}]^{2}\right\}.
\end{align*}
\end{lemma}
\begin{proof}
As shown in the proof of Lemma \ref{lemma:festimate}, we know that 
\begin{align*}
\psi_{j}\det[(A-z)^{\ast}(A-z)+\eta_{z,t}^{2}]&=\det[(A^{(j-1)}-\lambda_{j})^{\ast}(A^{(j-1)}-\lambda_{j})+\eta_{z,t}^{2}]\\
&\times\prod_{\ell=1}^{j-1}\left|\det\left[V_{\ell}^{\ast}G_{\lambda_{j}}^{(\ell-1)}(\eta_{z,t})V_{\ell}\right]\right|^{-1}\left[1+O(N^{-\kappa})\right]
\end{align*}
for some $\kappa>0$ {locally uniformly in $a_{j},b_{j}$}. We now bound the inverse-determinants in the second line. We first claim that the following holds: 
\begin{align*}
\left\|\left[V_{\ell}^{\ast}G_{\lambda_{j}}^{(\ell-1)}(\eta_{z,t})V_{\ell}\right]^{-1}\right\|_{\mathrm{op}}&=\left\|V_{\ell}\left[V_{\ell}^{\ast}G_{\lambda_{j}}^{(\ell-1)}(\eta_{z,t})V_{\ell}\right]^{-1}V_{\ell}^{\ast}\right\|_{\mathrm{op}}\\
&\lesssim\left\|G_{\lambda_{j}}^{(\ell-1)}(\eta_{z,t})V_{\ell}\left[V_{\ell}^{\ast}G_{\lambda_{j}}^{(\ell-1)}(\eta_{z,t})V_{\ell}\right]^{-1}V_{\ell}^{\ast}G_{\lambda_{j}}^{(\ell-1)}(\eta_{z,t})\right\|_{\mathrm{op}}\\
&\lesssim\|[\mathrm{Im}\mathcal{H}_{\lambda_{j}}^{(\ell-1)}(\eta_{z,t})]^{-1}\|_{\mathrm{op}}\lesssim\eta_{z,t}^{-1}.
\end{align*}
The first line follows because $V_{\ell}$ is projection onto its image, and $[V_{\ell}^{\ast}G_{\lambda_{j}}^{(\ell-1)}(\eta_{z,t})V_{\ell}]^{-1}$is a map from said image to itself. The last line follows by Lemma 2.1 in \cite{MO}. Thus, since $\eta_{z,t}=N^{-\epsilon_{0}}$, it suffices to show that
\begin{align*}
&\det[(A^{(j-1)}-\lambda_{j})^{\ast}(A^{(j-1)}-\lambda_{j})+\eta_{z,t}^{2}]\det[(A_{a_{j},b_{j},\theta_{j}}^{(j-1)})^{\ast}A_{a_{j},b_{j},\theta_{j}}^{(j-1)}+\eta_{z,t}^{2}]^{-\frac12}\\
&\lesssim\exp\left\{-Cb_{j}^{2}N^{2}\eta_{z,t}^{2}[\tan\theta_{j}-\tan^{-1}\theta_{j}]^{2}\right\}.
\end{align*}
We claim the following holds:
\begin{align}
&\det[(A^{(j-1)}-\lambda_{j})^{\ast}(A^{(j-1)}-\lambda_{j})+\eta_{z,t}^{2}]\det[(A_{a_{j},b_{j},\theta_{j}}^{(j-1)})^{\ast}A_{a_{j},b_{j},\theta_{j}}^{(j-1)}+\eta_{z,t}^{2}]^{-\frac12}\label{eq:cutoff1main}\\
&=\left\{\det\left[I+b_{j}^{2}(\tan\theta_{j}-\tan^{-1}\theta_{j})^{2}\eta_{z,t}^{2}H_{\lambda_{j}}^{(j-1)}(\eta_{z,t})^{\frac12}\tilde{H}_{\overline{\lambda}_{j}}^{(j-1)}(\eta_{z,t})H_{\lambda_{j}}^{(j-1)}(\eta_{z,t})^{\frac12}\right]\right\}^{-\frac12}.\nonumber
\end{align}
Assuming this, the proof follows quickly. Indeed, we know by definition that
\begin{align*}
&H_{\lambda_{j}}^{(j-1)}(\eta_{z,t})^{\frac12}\tilde{H}_{\overline{\lambda}_{j}}^{(j-1)}(\eta_{z,t})H_{\lambda_{j}}^{(j-1)}(\eta_{z,t})^{\frac12}\\
&=[(A^{(j-1)}-\lambda_{j})(A^{(j-1)}-\lambda_{j})^{\ast}+\eta_{z,t}^{2}]^{-\frac12}\\
&\times[(A^{(j-1)}-\overline{\lambda}_{j})^{\ast}(A^{(j-1)}-\overline{\lambda}_{j})+\eta_{z,t}^{2}]^{-1}\\
&\times[(A^{(j-1)}-\lambda_{j})(A^{(j-1)}-\lambda_{j})^{\ast}+\eta_{z,t}^{2}]^{-\frac12}.
\end{align*}
The operator norm of each matrix we take inverse of on the RHS is $O(1)$; there is no $\theta$-dependence. So, we know that for some $C>0$ independent of $N,\theta$, we have 
\begin{align*}
H_{\lambda_{j}}^{(j-1)}(\eta_{z,t})^{\frac12}\tilde{H}_{\overline{\lambda}_{j}}^{(j-1)}(\eta_{z,t})H_{\lambda_{j}}^{(j-1)}(\eta_{z,t})^{\frac12}\geq C.
\end{align*}
This gives
\begin{align*}
&\left\{\det\left[I+b_{j}^{2}(\tan\theta_{j}-\tan^{-1}\theta_{j})^{2}\eta_{z,t}^{2}H_{\lambda_{j}}^{(j-1)}(\eta_{z,t})^{\frac12}\tilde{H}_{\overline{\lambda}_{j}}^{(j-1)}(\eta_{z,t})H_{\lambda_{j}}^{(j-1)}(\eta_{z,t})^{\frac12}\right]\right\}^{-\frac12}\\
&\leq\left\{\det\left[I+Cb_{j}^{2}(\tan\theta_{j}-\tan^{-1}\theta_{j})^{2}\eta_{z,t}^{2}\right]\right\}^{-\frac12}\\
&\leq(1+Cb_{j}^{2}(\tan\theta_{j}-\tan^{-1}\theta_{j})^{2}\eta_{z,t}^{2})^{-\frac{N}{2}}\lesssim e^{-NC'b_{j}^{2}\eta_{z,t}^{2}(\tan\theta_{j}-\tan^{-1}\theta_{j})^{2}}.
\end{align*}
At this point, the claim follows, so it suffices to prove \eqref{eq:cutoff1main}. To ease notation, let us focus on $j=1$. We let matrices $R$ and $L$ consist of right and left eigenvectors of $\Lambda_{a_{1},b_{1},\theta_{1}}$. More precisely,
\[
L = \frac{1}{\sqrt{2}}\begin{pmatrix}1 & -i\tan\theta_{1}\\ 1 & i\tan\theta_{1}\end{pmatrix}, \qquad R = \frac{1}{\sqrt{2}}\begin{pmatrix}1 & -i\tan^{-1}\theta_{1}\\ 1 & i\tan^{-1}\theta_{1}\end{pmatrix}.
\]
Then it is easy to see that
\[
\Lambda_{a_{1},b_{1},\theta_{1}} = R^\ast \begin{pmatrix} \lambda_{1} &0\\ 0 & \overline{\lambda}_{1}\end{pmatrix} L, \qquad \Lambda_{a_{1},b_{1},\frac{\pi}{2}-\theta_{1}}^\ast = R^\ast \begin{pmatrix} \lambda &0\\ 0 & \overline{\lambda}_{1}\end{pmatrix}^\ast L.
\]
Consider a permutation matrix
\[
J = 
\begin{pmatrix}
    I_{N} & 0 & 0 & 0\\
    0 & 0 & I_{N} & 0\\
    0 & I_{N} & 0 & 0\\
    0 & 0 & 0 & I_{N}
\end{pmatrix}.
\]
Now, we introduce the following $4N\times4N$ matrix:
\[
\Gamma_{\lambda_{1},\theta_{1}}(\eta) := 
\begin{pmatrix}
    i\eta & A_{a_{1},b_{1},\theta_{1}} \\
    A_{a_{1},b_{1},\frac{\pi}{2}-\theta_{1}} & i\eta
\end{pmatrix}^{-1}.
\]
It is straightforward to check that
\[
\Gamma_{\lambda_{1},\theta_{1}}(\eta) = \begin{pmatrix} R^\ast\otimes I_N & 0\\ 0 & R^\ast\otimes I_N\end{pmatrix} J \begin{pmatrix} G_{\lambda_{1}}(\eta) & 0\\ 0 & G_{\overline{\lambda}_{1}}(\eta)\end{pmatrix} J \begin{pmatrix} L\otimes I_N & 0\\ 0 & L\otimes I_N\end{pmatrix}.
\]
Thus
\begin{align*}
\det\Gamma_{\lambda_{1},\theta_{1}}(\eta) = \det G_{\lambda_{1}}(\eta)\det G_{\overline{\lambda}_{1}}(\eta) &= \left\{\det\left[(A-\lambda_{1})(A-\lambda_{1})^\ast+\eta^2\right]\right\}^{-2} \\
&= \left\{\det\begin{pmatrix} i\eta & A-\lambda_{1}\\ A^T - \overline{\lambda}_{1} & i\eta \end{pmatrix}\right\}^{-2}.
\end{align*}
We plug this in to get
\begin{align*}
&\det[(A-\lambda_{1})^{\ast}(A-\lambda_{1})+\eta_{z,t}^{2}]\det[A_{a_{1},b_{1},\theta_{1}}^{\ast}A_{a_{1},b_{1},\theta_{1}}+\eta_{z,t}^{2}]^{-\frac12}\\
&=\det\begin{pmatrix} i\eta_{z,t} & A-\lambda_{1} \\ A^{T} - \overline{\lambda}_{1} & i\eta_{z,t}\end{pmatrix}\left\{\det\begin{pmatrix}i\eta_{z,t}&A_{a_{1},b_{1},\theta_{1}}\\A_{a_{1},b_{1},\theta_{1}}^{\ast}&i\eta_{z,t}\end{pmatrix}\right\}^{-\frac12} \\
&= \left\{\det\begin{pmatrix}i\eta_{z,t}&A_{a_{1},b_{1},\theta_{1}}\\A_{a,b,\frac{\pi}{2}-\theta_{1}}^{\ast}&i\eta_{z,t}\end{pmatrix}\right\}^{\frac12}\left\{\det\begin{pmatrix}i\eta_{z,t}&A_{a_{1},b_{1},\theta_{1}}\\A_{a_{1},b_{1},\theta_{1}}^{\ast}&i\eta_{z,t}\end{pmatrix}\right\}^{-\frac12} \\
&= \left\{\det\left[I_{4N} + \Gamma_{\lambda_{1},\theta_{1}}(\eta_{z,t})\begin{pmatrix}0 & 0\\\left(\Lambda_{a_{1},b_{1},\frac{\pi}{2}-\theta_{1}} - \Lambda_{a_{1},b_{1},\theta_{1}}\right)^\ast\otimes I_N & 0\end{pmatrix}\right]\right\}^{-\frac12} \\
&= \left\{\det\left[I_{4N} - \begin{pmatrix}G_{\lambda_{1}}(\eta_{z,t}) & 0\\0 & G_{\overline{\lambda}_{1}}(\eta_{z,t})\end{pmatrix}b_{1}\left(\tan\theta_{1}-\tan^{-1}\theta_{1}\right)Q_{\theta_{1}}\otimes I_N\right]\right\}^{-\frac12},
\end{align*}
where
\begin{align*}
Q_{\theta_{1}} &= J\begin{pmatrix}L & 0\\0 & L\end{pmatrix}\begin{pmatrix}0 & 0 & 0 & 0\\0 & 0 & 0 & 0\\0 & 1 & 0 & 0\\1 & 0 & 0 & 0\end{pmatrix}\begin{pmatrix}R^\ast & 0\\0 & R^\ast\end{pmatrix}J \\
&= \frac{i}{2}\begin{pmatrix}-(\tan\theta_{1}-\tan^{-1}\theta_{1})& -(\tan\theta_{1}+\tan^{-1}\theta_{1})\\\tan\theta_{1}+\tan^{-1}\theta_{1} & \tan\theta_{1}-\tan^{-1}\theta_{1} \end{pmatrix}\otimes E_{(2)},
\end{align*}
and, for any integer $k\geq1$, we set
\[
E_{(k)} = \begin{pmatrix}
    0 & 0\\
    I_{k} & 0
\end{pmatrix}.
\]
Thus, we can compute the LHS of the identity in \eqref{eq:cutoff1main} as follows:
\begin{align*}
&\det[(A-\lambda_{1})^{\ast}(A-\lambda_{1})+\eta_{z,t}^{2}]\det[A_{a_{1},b_{1},\theta_{1}}^{\ast}A_{a_{1},b_{1},\theta_{1}}+\eta_{z,t}^{2}]^{-\frac12}\\
&=
\left\{\det
\begin{pmatrix}
    I_{2N} + \frac{ib_{1}}{2}\left(\tan\theta_{1}-\tan^{-1}\theta_{1}\right)^2 G_{\lambda_{1}}(\eta_{z,t})E_{(N)} & \frac{ib}{2}\left(\tan^2\theta_{1} - \tan^{-2}\theta_{1}\right)G_{\lambda_{1}}(\eta_{z,t})E_{(N)} \\
    -\frac{ib_{1}}{2}\left(\tan^2\theta_{1} - \tan^{-2}\theta_{1}\right)G_{\overline{\lambda}_{1}}(\eta_{z,t})E_{(N)} & I_N - \frac{ib_{1}}{2}\left(\tan\theta_{1}-\tan^{-1}\theta_{1}\right)^2 G_{\overline{\lambda}_{1}}(\eta_{z,t})E_{(N)}
\end{pmatrix}
\right\}^{-\frac12}.
\end{align*}
Since the $(1,1)$ and $(1,2)$ blocks of this matrix commute, this determinant is equal to
\begin{align*}
&\left\{\det\left[\left(I_N + \frac{ib_{1}}{2}\left(\tan\theta_{1}-\tan^{-1}\theta_{1}\right)^2 G_{\lambda_{1}}(\eta_{z,t})E_{(N)}\right)\left(I_N - \frac{ib_{1}}{2}\left(\tan\theta_{1}-\tan^{-1}\theta_{1}\right)^2 G_{\overline{\lambda}_{1}}(\eta_{z,t})E_{(N)}\right)\right.\right. \\
&- \left.\left.\frac{b_{1}^2}{4}\left(\tan^2\theta_{1} - \tan^{-2}\theta_{1}\right)^2G_{\lambda_{1}}(\eta_{z,t})E_{(N)}G_{\overline{\lambda}_{1}}(\eta_{z,t})E_{(N)}\right]\right\}^{-\frac12}\\
&=\left\{\det\left[I + \frac{ib_{1}}{2}\left(\tan\theta_{1}-\tan^{-1}\theta_{1}\right)^2\left(G_{\lambda_{1}}(\eta_{z,t}) - G_{\overline{\lambda}_{1}}(\eta_{z,t})\right)E_{(N)}\right.\right. \\
&+ \left.\left.b_{1}^2\left(2-\tan^2\theta_{1}-\tan^{-2}\theta_{1}\right)G_{\lambda_{1}}(\eta_{z,t})E_{(N)}G_{\overline{\lambda}_{1}}(\eta_{z,t})E_{(N)}\right]\right\}^{-\frac12}.
\end{align*}
Note that
\begin{align*}
G_{\lambda_{1}}(\eta_{z,t}) - G_{\overline{\lambda}_{1}}(\eta_{z,t}) &= G_{\lambda_{1}}(\eta_{z,t})\begin{pmatrix} 0 & \lambda_{1}-\overline{\lambda}_{1}\\ \overline{\lambda}_{1}-\lambda_{1} & 0\end{pmatrix}G_{\overline{\lambda}_{1}}(\eta_{z,t}) \\
&= 2b_{1}iG_{\lambda_{1}}(\eta_{z,t})(E_{(N)}^\ast-E_{(N)})G_{\overline{\lambda}_{1}}(\eta_{z,t}).
\end{align*}
Combining the previous three displays gives \eqref{eq:cutoff1main}.
\end{proof}
We now control the $dP$ integral.
\begin{lemma}\label{lemma:cutoff2}
Fix $j\geq1$ and $\theta_{j}\in[0,\frac{\pi}{4}]$. We have 
\begin{align*}
&\left|\int_{M_{2}^{sa}(\R)}e^{i\frac{N}{2t}\tr P}\det[I+i[\tilde{H}_{a_{j},b_{j},\theta_{j}}^{(j-1)}]^{\frac12}(P\otimes I)[\tilde{H}_{a_{j},b_{j},\theta_{j}}^{(j-1)}]^{\frac12}]^{-\frac12}dP\right|\\
&\lesssim N^{C}+C^{N}N^{-10(N-2j+6)}\theta_{j}^{-2N}+N^{-5N+10j+12}\theta_{j}^{-2N}.
\end{align*}
If $\theta_{j}\in[\frac{\pi}{4},\frac{\pi}{2}]$, the same bound holds upon replacing $\theta_{j}$ by $\frac{\pi}{2}-\theta_{j}$.
\end{lemma}
\begin{proof}
We write the proof in the case where $\theta_{j}\leq\frac{\pi}{4}$. For the case $\theta_{j}\geq\frac{\pi}{4}$, it suffices to use the same argument but replace $\theta_{j}$ by $\frac{\pi}{2}-\theta_{j}$.

For the sake of an upper bound, we can move the absolute value inside the integration and drop the complex exponential. For any $P\in M_{2}^{sa}(\R)$, write $P=UDU^{\ast}$, where $U\in\mathbf{U}(2)$ is unitary and $D$ is diagonal. Change-of-variables then shows
\begin{align*}
&\int_{M_{2}^{sa}(\R)}|\det[I+i[\tilde{H}_{a_{j},b_{j},\theta_{j}}^{(j-1)}]^{\frac12}(P\otimes I)[\tilde{H}_{a_{j},b_{j},\theta_{j}}^{(j-1)}]^{\frac12}]^{-\frac12}|dP\\
&\lesssim\int_{\mathbf{U}(2)}\int_{\R^{2}}|\det[I+i\mathbf{H}D\otimes I_{N-2j+2}\mathbf{H}]|^{-\frac12}|D_{11}-D_{22}|dD_{11}dD_{22}dU,
\end{align*}
where $$\mathbf{H}:=U^{\ast}\otimes I_{N-2j+2}\tilde{H}_{a_{j},b_{j},\theta_{j}}^{(j-1)}(\eta_{z,t})^{\frac12}U\otimes I_{N-2j+2}.$$ Because $\mathbf{U}(2)$ is compact, we only need to control the integral on $\R^{2}$ uniformly in $U$. We split the $\R^{2}$ integration as follows:
\begin{align*}
&\int_{\R^{2}}|\det[I+i\mathbf{H}D\otimes I_{N-2j+2}\mathbf{H}]|^{-\frac12}|D_{11}-D_{22}|dD_{11}dD_{22}\\
&=\int_{|D_{11}|,|D_{22}|\leq N^{10}}|\det[I+i\mathbf{H}D\otimes I_{N-2j+2}\mathbf{H}]|^{-\frac12}|D_{11}-D_{22}|dD_{11}dD_{22}\\
&+\int_{|D_{11}|,|D_{22}|\geq N^{10}}|\det[I+i\mathbf{H}D\otimes I_{N-2j+2}\mathbf{H}]|^{-\frac12}|D_{11}-D_{22}|dD_{11}dD_{22}\\
&+\int_{\substack{|D_{11}|\leq N^{10}\\|D_{22}|\geq N^{10}}}|\det[I+i\mathbf{H}D\otimes I_{N-2j+2}\mathbf{H}]|^{-\frac12}|D_{11}-D_{22}|dD_{11}dD_{22}\\
&+\int_{\substack{|D_{11}|\geq N^{10}\\|D_{22}|\leq N^{10}}}|\det[I+i\mathbf{H}D\otimes I_{N-2j+2}\mathbf{H}]|^{-\frac12}|D_{11}-D_{22}|dD_{11}dD_{22}\\
&=\mathrm{I}+\mathrm{II}+\mathrm{III}+\mathrm{IV}.
\end{align*}
Since $\mathbf{H}D\otimes I_{N-2j+2}\mathbf{H}$ is self-adjoint (as $\mathbf{H}$ is self-adjoint), we get that $|\det[I+i\mathbf{H}D\otimes I_{N-2j+2}\mathbf{H}]|\geq1$. This implies $|\mathrm{I}|\leq N^{C}$. Next, we note that 
\begin{align*}
|\det[I+i\mathbf{H}D\otimes I_{N-2j+2}\mathbf{H}]|^{-1/2}&\leq|\det[\mathbf{H}D\otimes I_{N-2j+2}\mathbf{H}]|^{-1/2}\\
&=|D_{11}|^{-\frac{N-2j+2}{2}}|D_{22}|^{-\frac{N-2j+2}{2}}|\det\mathbf{H}^{2}|^{-1/2}
\end{align*}
for the same reason. But the eigenvalues of $\mathbf{H}^{2}$ are those of $H_{a_{j},b_{j},\theta_{j}}^{(j-1)}(\eta_{z,t})$, which are uniformly $\gtrsim\theta_{j}^{2}$ by definition of $H_{a_{j},b_{j},\theta_{j}}^{(j-1)}(\eta_{z,t})$ and a bound on the operator norm of $A_{a_{j},b_{j},\theta_{j}}^{(j-1)}$ of $\lesssim\theta_{j}^{-2}$. This gives 
\begin{align*}
|\mathrm{II}|&\lesssim\int_{|D_{11}|,|D_{22}|\geq N^{10}}|D_{11}|^{-\frac{N-2j+2}{2}}|D_{22}|^{-\frac{N-2j+2}{2}}\theta_{j}^{-2N}|D_{11}-D_{22}|dD_{11}dD_{22}\\
&\lesssim C^{N}N^{-10(N-2j+6)}\theta_{j}^{-2N}.
\end{align*}
We are left to bound $\mathrm{III}$; the bound for $\mathrm{IV}$ follows by the same argument but swapping $D_{11}$ and $D_{22}$. For convenience, write 
\begin{align*}
\mathbf{H}_{1}:=I+i\mathbf{H}\begin{pmatrix}0&0\\0&D_{22}\end{pmatrix}\otimes I_{N-2j+2}\mathbf{H}.
\end{align*}
We have 
\begin{align*}
&|\det[I+i\mathbf{H}D\otimes I_{N-2j+2}\mathbf{H}]|\\
&=|\det[\mathbf{H}_{1}]|\times\det\left[I+i\mathbf{H}_{1}^{-\frac12}\mathbf{H}\begin{pmatrix}D_{11}&0\\0&0\end{pmatrix}I_{N-2j+2}\mathbf{H}\mathbf{H}_{1}^{-\frac12}\right]\\
&\geq\left|\det\left[I+i\mathbf{H}^{2}\begin{pmatrix}0&0\\0&D_{22}\end{pmatrix}\otimes I_{N-2j+2}\right]\right|.
\end{align*}
where the last inequality holds because the second factor in the second line has the form $|\det[I+iA]|$ with $A$ self-adjoint. Now, write $\mathbf{H}^{2}=\begin{pmatrix}\mathbf{L}_{11}&\mathbf{L}_{12}\\\mathbf{L}_{12}&\mathbf{L}_{22}\end{pmatrix}$, and compute
\begin{align*}
\mathbf{H}^{2}\begin{pmatrix}0&0\\0&D_{22}\end{pmatrix}\otimes I_{N-2j+2}&=D_{22}\begin{pmatrix}0&\mathbf{L}_{12}\\0&\mathbf{L}_{22}\end{pmatrix}.
\end{align*}
This gives us 
\begin{align*}
\left|\det\left[I+i\mathbf{H}^{2}\begin{pmatrix}0&0\\0&D_{22}\end{pmatrix}\otimes I_{N-2j+2}\right]\right|=\left|\det\begin{pmatrix}1&iD_{22}\mathbf{L}_{12}\\0&1+iD_{22}\mathbf{L}_{22}\end{pmatrix}\right|\geq |\det[D_{22}\mathbf{L}_{22}]|.
\end{align*}
Since $\mathbf{L}_{22}$ is a diagonal block of $\mathbf{H}^{2}$, and since $\mathbf{H}$ is self-adjoint, we know that $\mathbf{L}_{22}$ is self-adjoint. We also recall $\mathbf{H}^{2}\gtrsim\theta_{j}^{2}$ from earlier in this proof. This implies $\mathbf{L}_{22}\gtrsim\theta_{j}^{2}$ by restricting to vectors in the block corresponding to $\mathbf{L}_{22}$. Thus, the previous display is $\geq C^{-N}D_{22}^{N-2j+2}\theta_{j}^{-2N}$. Ultimately, 
\begin{align*}
|\mathrm{III}|&\lesssim\int_{\substack{|D_{11}\leq N^{10}\\|D_{22}|\geq N^{10}}}|\det[D_{22}\mathbf{L}_{22}]|^{-\frac12}|D_{11}-D_{22}|dD_{11}dD_{22}\\
&\lesssim N^{10}\int_{|D_{22}|\geq N^{10}}C^{N}D_{22}^{-\frac{N-2j+2}{2}}\theta_{j}^{-2N}dD_{22}\\
&\lesssim C^{N}N^{-5N+10j+12}\theta_{j}^{-2N}.
\end{align*}
As mentioned earlier, the same bound holds for $\mathrm{IV}$. This completes the proof.
\end{proof}
By the previous lemmas, we can now bound the integral $d\theta_{j}$ when $\theta_{j}\in\mathcal{I}_{1}$.
\begin{lemma}\label{lemma:i2cutoff}
Fix any $j\geq1$. There exists $\kappa>0$ such that {locally uniformly in $a_{j},b_{j}$}, we have
\begin{align*}
&\int_{\mathcal{I}_{1}}\frac{|\cos2\theta_{j}|}{\sin^{2}2\theta_{j}}\psi_{j}\det[(A-z)^{\ast}(A-z)+\eta_{z,t}^{2}]\det[(A_{a_{j},b_{j},\theta_{j}}^{(j-1)})^{\ast}A_{a_{j},b_{j},\theta_{j}}^{(j-1)}+\eta_{z,t}^{2}]^{-\frac12}\\
&\times\int_{M_{2}^{sa}(\R)}e^{i\frac{N}{2t}\tr P}\det[I+i[\tilde{H}_{a_{j},b_{j},\theta_{j}}^{(j-1)}]^{\frac12}(P\otimes I)[\tilde{H}_{a_{j},b_{j},\theta_{j}}^{(j-1)}]^{\frac12}]^{-\frac12}dPd\theta_{j}\\
&\lesssim e^{-N^{\kappa}}.
\end{align*}
\end{lemma}
\begin{proof}
We break up $\mathcal{I}_{1}=\mathcal{I}_{11}\cup\mathcal{I}_{12}$, where $\mathcal{I}_{11}=[0,\frac{\pi}{4}-N^{-1/2+\tau}]$ and $\mathcal{I}_{12}=[\frac{\pi}{4}+N^{-1/2+\tau},\frac{\pi}{2}]$. We will prove the estimate
\begin{align*}
&\int_{\mathcal{I}_{11}}\frac{|\cos2\theta_{j}|}{\sin^{2}2\theta_{j}}\psi_{j}\det[(A-z)^{\ast}(A-z)+\eta_{z,t}^{2}]\det[(A_{a_{j},b_{j},\theta_{j}}^{(j-1)})^{\ast}A_{a_{j},b_{j},\theta_{j}}^{(j-1)}+\eta_{z,t}^{2}]^{-\frac12}\\
&\times\int_{M_{2}^{sa}(\R)}e^{i\frac{N}{2t}\tr P}\det[I+i[\tilde{H}_{a_{j},b_{j},\theta_{j}}^{(j-1)}]^{\frac12}(P\otimes I)[\tilde{H}_{a_{j},b_{j},\theta_{j}}^{(j-1)}]^{\frac12}]^{-\frac12}dPd\theta_{j}\\
&\lesssim e^{-N^{\kappa}}.
\end{align*}
The estimate for $\mathcal{I}_{12}$ instead of $\mathcal{I}_{11}$ follows by the same argument after changing variables $\theta_{j}\mapsto\frac{\pi}{2}-\theta_{j}$. 

By Lemmas \ref{lemma:cutoff1} and \ref{lemma:cutoff2}, it suffices to show 
\begin{align}
&\int_{\mathcal{I}_{11}}\frac{|\cos2\theta_{j}|}{\sin^{2}2\theta_{j}}N^{C\epsilon_{0}}\exp\left\{-Cb_{j}^{2}N^{2}\eta_{z,t}^{2}[\tan\theta_{j}-\tan^{-1}\theta_{j}]^{2}\right\}\label{eq:i2cutoffI}\\
&\times\left[N^{C}+C^{N}N^{-10(N-2j+6)}\theta_{j}^{-2N}+N^{-5N+10j+12}\theta_{j}^{-2N}\right]d\theta_{j}\nonumber\\
&\lesssim e^{-N^{\kappa}}.\nonumber
\end{align}
Now, further decompose $\mathcal{I}_{11}=[0,c]\cup[c,\frac{\pi}{4}-N^{-1/2+\tau}]$, where $c>0$ is a small, fixed constant.

Assume first that $\theta_{j}\in[c,\frac{\pi}{4}-N^{-1/2+\tau}]$. We know $|\theta_{j}-\frac{\pi}{4}|\geq N^{-1/2+\tau}\eta_{z,t}^{-1}$ for $\tau>0$ small and fixed. Thus, the exponential in the first line in \eqref{eq:i2cutoffI} is $O(\exp[-CN^{2\tau}])$, and $|\cos2\theta_{j}|^{-1}\lesssim N^{1/2}$. So 
\begin{align*}
&\int_{c}^{\frac{\pi}{4}-N^{-1/2+\tau}}\frac{|\cos2\theta_{j}|}{\sin^{2}2\theta_{j}}N^{C\epsilon_{0}}\exp\left\{-Cb_{j}^{2}N^{2}\eta_{z,t}^{2}[\tan\theta_{j}-\tan^{-1}\theta_{j}]^{2}\right\}\\
&\times\left[N^{C}+C^{N}N^{-10(N-2j+6)}\theta_{j}^{-2N}+N^{-5N+10j+12}\theta_{j}^{-2N}\right]d\theta_{j}\\
&\lesssim D^{N}e^{-CN^{2\tau}}\lesssim e^{-N^{\kappa}},
\end{align*}
where $D=O(1)$. We now handle the integral on $[0,c]$. In this case, if $c>0$ is small enough, we have the lower bound $|\tan\theta_{j}-\tan^{-1}\theta_{j}|\gtrsim\theta_{j}^{-1}$. We can also bound $|\cos2\theta_{j}|=O(1)$. Thus, we have the following estimate for $D=O(1)$ and $C>0$:
\begin{align*}
&\int_{0}^{c}\frac{|\cos2\theta_{j}|}{\sin^{2}2\theta_{j}}N^{C\epsilon_{0}}\exp\left\{-Cb_{j}^{2}N^{2}\eta_{z,t}^{2}[\tan\theta_{j}-\tan^{-1}\theta_{j}]^{2}\right\}\\
&\times\left[N^{C}+C^{N}N^{-10(N-2j+6)}\theta_{j}^{-2N}+N^{-5N+10j+12}\theta_{j}^{-2N}\right]d\theta_{j}\\
&\lesssim\int_{0}^{c}\frac{1}{\sin^{2}2\theta_{j}}D^{N}\theta_{j}^{-2N}\exp[-Cb_{j}^{2}N^{2}\eta_{z,t}^{2}\theta_{j}^{-2}]d\theta_{j}.
\end{align*}
But $N^{2}\eta_{z,t}^{2}\gg N$ since $\eta_{z,t}\gtrsim t=N^{-\epsilon_{0}}$. So, the exponential decays faster than $\theta_{j}^{-2N}$ or $\sin^{-2}2\theta_{j}$ blow up as $\theta_{j}\to0$. In particular, the last line is $\lesssim\exp[-N^{\kappa}]$ for some $\kappa>0$ by elementary calculus. This completes the proof.
\end{proof}
\subsection{Proof of Proposition \ref{prop:cutoff}}
Lemmas \ref{lemma:cutoff1} and \ref{lemma:cutoff2} and a straightforward bound $|\det[V_{j}^{\ast}G_{z}^{(j-1)}(\eta_{z,t})V_{j}]|\lesssim\|G_{z}^{(j-1)}(\eta_{z,t})\|_{\mathrm{op}}^{4}\lesssim\eta_{z,t}^{-4}\lesssim N^{4\epsilon_{0}}$ shows that $\Theta_{0}\lesssim N^{D}$ for some $D=O(1)$. Thus, $\Theta_{10},\Theta_{11}=O(\exp[-N^{\kappa}])$ by Lemma \ref{lemma:i2cutoff}. \qed
%
%
%
\section{Estimates for \texorpdfstring{$\rho_{\mathrm{main}}(z,\mathbf{z};A)$}{rho\_main}}
Let us recall
\begin{align*}
\rho_{\mathrm{main}}(z,\mathbf{z};A)&:=\frac{N^{7}b^{4}[1+O(N^{-\kappa})]}{16\pi^{4}t^{12}\gamma_{z,t}^{2}\sigma_{z,t}^{3}\upsilon_{z,t}^{2}}\rho_{\mathrm{GinUE}}^{(2)}(z_{1},z_{2})\int_{\mathcal{I}_{0}^{2}}d\theta_{1}d\theta_{2}\frac{256|\cos2\theta_{1}||\cos2\theta_{2}|}{|\sin^{2}2\theta_{1}|\sin^{2}2\theta_{2}|}\\
&\times\prod_{j=1,2}\int_{V^{2}(\R^{N-2j+2})}d\mu_{j}(\mathbf{v}_{j})|\det[V_{j}^{\ast}G_{z}^{(j-1)}(\eta_{z,t})V_{j}]|^{2}\\
&\times\prod_{j=1,2}\psi_{j}\det[(A-z)^{\ast}(A-z)+\eta_{z,t}^{2}]\det[(A_{a_{j},b_{j},\theta_{j}}^{(j-1)})^{\ast}A_{a_{j},b_{j},\theta_{j}}^{(j-1)}+\eta_{z,t}^{2}]^{-\frac12}\\
&\times\prod_{j=1,2}\int_{M_{2}^{sa}(\R)}e^{i\frac{N}{2t}\tr P}\det[I+i[\tilde{H}_{a_{j},b_{j},\theta_{j}}^{(j-1)}]^{\frac12}(P\otimes I)[\tilde{H}_{a_{j},b_{j},\theta_{j}}^{(j-1)}]^{\frac12}]^{-\frac12}dP.
\end{align*}
The goal of this section is to express every term above in terms of traces of resolvents.
\subsection{Ratio of determinants}
As in the proof of Lemma \ref{lemma:cutoff1}, we start with the estimate
\begin{align*}
\psi_{j}\det[(A-z)^{\ast}(A-z)+\eta_{z,t}^{2}]&=\det[(A^{(j-1)}-\lambda_{j})^{\ast}(A^{(j-1)}-\lambda_{j})+\eta_{z,t}^{2}]\\
&\times\prod_{\ell=1}^{j-1}\left|\det\left[V_{\ell}^{\ast}G_{\lambda_{j}}^{(\ell-1)}(\eta_{z,t})V_{\ell}\right]\right|^{-1}\left[1+O(N^{-\kappa})\right]
\end{align*}
for some $\kappa>0$ {locally uniformly in $a_{j},b_{j}$}. The second line will be addressed by Lemma \ref{lemma:concentration} below, so we deal with the first determinant on the RHS. Unlike the proof of Lemma \ref{lemma:cutoff1}, we want to compute \eqref{eq:cutoff1main} more precisely in terms of traces of resolvents; the point is that we now have the a priori estimate $|\theta-\frac{\pi}{4}|\lesssim N^{-1/2+\tau}\eta_{z,t}^{-1}$ since we restrict to $\theta_{j}\in\mathcal{I}_{0}$. So, take the RHS of \eqref{eq:cutoff1main} and expand in terms of the trace. In particular, we have 
\begin{align*}
&\det[(A^{(j-1)}-\lambda_{j})(A^{(j-1)}-\lambda_{j})^{\ast}+\eta_{z,t}^{2}]\det[A_{a_{j},b_{j},\theta_{j}}^{(j-1)}(A_{a_{j},b_{j},\theta_{j}}^{(j-1)})^{\ast}+\eta_{z,t}^{2}]^{-\frac12}\\
&=\exp\left\{-\sum_{k=1}^{\infty}\frac{1}{k}b_{j}^{2k}(\tan\theta_{j}-\tan^{-1}\theta_{j})^{2k}\eta_{z,t}^{2k}\tr\left[H_{\lambda_{j}}^{(j-1)}(\eta_{z,t})\tilde{H}_{\overline{\lambda}_{j}}(\eta_{z,t})\right]^{k}\right\}.
\end{align*}
Since $\theta_{j}\in\mathcal{I}_{0}$, we know that $|\tan\theta_{j}-\tan^{-1}\theta_{j}|\lesssim N^{-1/2+\tau}\eta_{z,t}^{-1}$, where $\tau>0$ is small. Thus, if we trivially bound the trace of the $k$-th power by $O(N\eta_{z,t}^{-4k})$, since $\eta_{z,t}=N^{-\epsilon_{0}}$ with $\epsilon_{0}>0$ small, the contribution of the sum from $k=2$ and on is $O(N^{-\kappa})$. In particular, we have 
\begin{align*}
&\exp\left\{-\sum_{k=1}^{\infty}\frac{1}{k}b_{j}^{2k}(\tan\theta_{j}-\tan^{-1}\theta_{j})^{2k}\eta_{z,t}^{2k}\tr\left[H_{\lambda_{j}}^{(j-1)}(\eta_{z,t})\tilde{H}_{\overline{\lambda}_{j}}(\eta_{z,t})\right]^{k}\right\}\\
&=\exp\left\{-b_{j}^{2}(\tan\theta_{j}-\tan^{-1}\theta_{j})^{2}\eta_{z,t}^{2}\tr H_{\lambda_{j}}^{(j-1)}(\eta_{z,t})\tilde{H}_{\overline{\lambda}_{j}}^{(j-1)}(\eta_{z,t})\right\}\left[1+O(N^{-\kappa})\right].
\end{align*}
By Cauchy interlacing and trivial resolvent bounds, we  can remove the superscript:
\begin{align}
\tr H_{\lambda_{j}}^{(j-1)}(\eta_{z,t})\tilde{H}_{\overline{\lambda}_{j}}^{(j-1)}(\eta_{z,t})&=\tr H_{\lambda_{j}}^{(j-1)}(\eta_{z,t})^{\frac12}\tilde{H}_{\overline{\lambda}_{j}}^{(j-1)}(\eta_{z,t})H_{\lambda_{j}}^{(j-1)}(\eta_{z,t})^{\frac12}\label{eq:interlacingdemo}\\
&=\tr H_{\lambda_{j}}(\eta_{z,t})^{\frac12}\tilde{H}_{\overline{\lambda}_{j}}(\eta_{z,t})H_{\lambda_{j}}(\eta_{z,t})^{\frac12}+O(N^{-1+2\tau}\eta_{z,t}^{-D})\nonumber\\
&=\tr H_{\lambda_{j}}(\eta_{z,t})\tilde{H}_{\overline{\lambda}_{j}}(\eta_{z,t})+O(N^{-1+2\tau}\eta_{z,t}^{-D}).\nonumber
\end{align}
The error term on the RHS is obtained by our bound on $\theta_{j}-\frac{\pi}{4}$ and a trivial operator norm bound on resolvents of $\eta_{z,t}^{-D}$. But $\eta_{z,t}=N^{-\epsilon_{0}}$, so if we choose $\epsilon_{0}$ small enough, this cost is $\lesssim N^{-\kappa}$. Putting this altogether, we have 
\begin{align*}
&\psi_{j}\det[(A-z)^{\ast}(A-z)+\eta_{z,t}^{2}]\det[(A_{a_{j},b_{j},\theta_{j}}^{(j-1)})^{\ast}A_{a_{j},b_{j},\theta_{j}}+\eta_{z,t}^{2}]^{-\frac12}\\
&\approx\exp\left\{-b_{j}^{2}(\tan\theta_{j}-\tan^{-1}\theta_{j})^{2}\eta_{z,t}^{2}\tr H_{\lambda_{j}}(\eta_{z,t})\tilde{H}_{\overline{\lambda}_{j}}(\eta_{z,t})\right\}\prod_{\ell=1}^{j-1}\left|\det\left[V_{\ell}^{\ast}G_{\lambda_{j}}^{(\ell-1)}(\eta_{z,t})V_{\ell}\right]\right|^{-1},
\end{align*}
where $\approx$ means equal to modulo a factor of $1+O(N^{-\kappa})$. 
\subsection{Estimating \texorpdfstring{$\det[V_{j}^{\ast}G_{z}^{(j-1)}(\eta_{z,t})V_{j}]$}{det V\_j G\_z V\_j}}
Write $G_{z}^{(j-1)}=G_{z}^{(j-1)}(\eta_{z,t})$ and $H_{z}^{(j-1)}=H_{z}^{(j-1)}(\eta_{z,t})$. For simplicity, we focus on the case $j=1$. We comment on the more general case at the end of this subsection. For $\mathbf{v}_{1}\in V^{2}(\R^{N})$, we write $\mathbf{v}_{1}=(\mathbf{v}_{11},\mathbf{v}_{12})$. By the block representation for $G_{z}$, we have
\begin{align}
V_{1}^{\ast}G_{z}V_{1}=\begin{pmatrix}i\eta\mathbf{v}_{11}^{\ast} H_{z}\mathbf{v}_{11}&i\eta\mathbf{v}_{11}^{\ast} H_{z}\mathbf{v}_{12}&\mathbf{v}_{11}^{\ast} H_{z}(A-z)\mathbf{v}_{11}&\mathbf{v}_{11}^{\ast} H_{z}(A-z)\mathbf{v}_{12}\\i\eta\mathbf{v}_{12}^{\ast} H_{z}\mathbf{v}_{11}&i\eta\mathbf{v}_{12}^{\ast} H_{z}\mathbf{v}_{12}&\mathbf{v}_{12}^{\ast} H_{z}(A-z)\mathbf{v}_{11}&\mathbf{v}_{12}^{\ast} H_{z}(A-z)\mathbf{v}_{21}\\\mathbf{v}_{11}^{\ast}(A-z)^{\ast} H_{z}\mathbf{v}_{11}&\mathbf{v}_{11}^{\ast}(A-z)^{\ast} H_{z}\mathbf{v}_{12}&i\eta\mathbf{v}_{11}^{\ast}\tilde{ H}_{z}\mathbf{v}_{11}&i\eta\mathbf{v}_{11}^{\ast}\tilde{ H}_{z}\mathbf{v}_{12}\\\mathbf{v}_{12}^{\ast}(A-z)^{\ast} H_{z}\mathbf{v}_{11}&\mathbf{v}_{12}^{\ast}(A-z)^{\ast} H_{z}\mathbf{v}_{12}&i\eta\mathbf{v}_{12}^{\ast}\tilde{ H}_{z}\mathbf{v}_{11}&i\eta\mathbf{v}_{12}^{\ast}\tilde{ H}_{z}\mathbf{v}_{12}\end{pmatrix}.\nonumber
\end{align}
Before we estimate its determinant, we must introduce notation. For any $a,b,\theta$, define $H_{a,b,\theta}(\eta)=[(I_{2}\otimes A-\Lambda_{a,b,\theta}\otimes I_{N})(I_{2}\otimes A-\Lambda_{a,b,\theta}\otimes I_{N})^{\ast}+\eta^{2}]^{-1}$ and $H_{a,b,\theta}=H_{a,b,\theta}(\eta_{z,t})$.
\begin{lemma}\label{lemma:concentration}
For any $p=O(1)$, there exists $\kappa>0$ such that 
\begin{align*}
\int_{V^{2}(\R^{N})}|\det[V_{1}^{\ast}G_{z}(\eta_{z,t})V_{1}]|^{p}d\mu_{1}(\mathbf{v}_{1})&=|\det\mathcal{G}|^{p}\left[1+O(N^{-\kappa})\right],
\end{align*}
where $\mathcal{G}$ is the following $4\times4$ matrix:
\begin{align*}
\mathcal{G}&:=\begin{pmatrix}\mathcal{G}_{11}&\mathcal{G}_{12}\\\mathcal{G}_{21}&\mathcal{G}_{22}\end{pmatrix}\\
\mathcal{G}_{11}&:=\begin{pmatrix}i\eta_{z,t}\frac{t}{N}\tr \tilde{H}_{a_{1},b_{1},\theta_{1}}( H_{z}\otimes E_{11})&i\eta_{z,t}\frac{t}{N}\tr \tilde{H}_{a_{1},b_{1},\theta_{1}}( H_{z}\otimes E_{12})\\i\eta_{z,t}\frac{t}{N}\tr \tilde{H}_{a_{1},b_{1},\theta_{1}}( H_{z}\otimes E_{21})&i\eta_{z,t}\frac{t}{N}\tr \tilde{H}_{a_{1},b_{1},\theta_{1}}( H_{z}\otimes E_{22})\end{pmatrix}\\
\mathcal{G}_{12}&:=\begin{pmatrix}\frac{t}{N}\tr \tilde{H}_{a_{1},b_{1},\theta_{1}}[ H_{z}(A-z)\otimes E_{11}]&\frac{t}{N}\tr \tilde{H}_{a_{1},b_{1},\theta_{1}}[ H_{z}(A-z)\otimes E_{12}]\\\frac{t}{N}\tr \tilde{H}_{a_{1},b_{1},\theta_{1}}[ H_{z}(A-z)\otimes E_{21}]&\frac{t}{N}\tr \tilde{H}_{a_{1},b_{1},\theta_{1}}[ H_{z}(A-z)\otimes E_{22}]\end{pmatrix}\\
\mathcal{G}_{21}&:=\begin{pmatrix}\frac{t}{N}\tr \tilde{H}_{a_{1},b_{1},\theta_{1}}[(A-z)^{\ast} H_{z}\otimes E_{11}]&\frac{t}{N}\tr \tilde{H}_{a_{1},b_{1},\theta_{1}}[(A-z)^{\ast} H_{z}\otimes E_{12}]\\\frac{t}{N}\tr \tilde{H}_{a_{1},b_{1},\theta_{1}}[(A-z)^{\ast} H_{z}\otimes E_{21}]&\frac{t}{N}\tr \tilde{H}_{a_{1},b_{1},\theta_{1}}[(A-z)^{\ast} H_{z}\otimes E_{22}]\end{pmatrix}\\
\mathcal{G}_{22}&:=\begin{pmatrix}i\eta_{z,t}\frac{t}{N}\tr \tilde{H}_{a_{1},b_{1},\theta_{1}}(\tilde{H}_{z}\otimes E_{11})&i\eta_{z,t}\frac{t}{N}\tr \tilde{H}_{a_{1},b_{1},\theta_{1}}(\tilde{H}_{z}\otimes E_{12})\\i\eta_{z,t}\frac{t}{N}\tr \tilde{H}_{a_{1},b_{1},\theta_{1}}(\tilde{H}_{z}\otimes E_{21})&i\eta_{z,t}\frac{t}{N}\tr \tilde{H}_{a_{1},b_{1},\theta_{1}}(\tilde{H}_{z}\otimes E_{22})\end{pmatrix},\\
E_{11}&:=\begin{pmatrix}1&0\\0&0\end{pmatrix},\quad E_{22}:=\begin{pmatrix}0&0\\0&1\end{pmatrix}\quad
E_{12}:=\begin{pmatrix}0&1\\0&0\end{pmatrix}\quad
E_{21}:=\begin{pmatrix}0&0\\1&0\end{pmatrix}.
\end{align*}
\end{lemma}
\begin{proof}
Let $A_{a,b,\theta}=I_{2}\otimes A-\Lambda_{a,b,\theta}\otimes I_{N}$, For any Hermitian matrix $F$, we start by introducing 
\begin{align*}
m_{F}(r)&:=\frac{e^{-\frac{rt}{2N}\tr \tilde{H}_{a_{1},b_{1},\theta_{1}}F}}{K_{a_{1},b_{1},\theta_{1}}}\int_{V^{2}(\R^{N})}\exp\left\{-\frac{N}{t}\mathbf{v}^{\ast}\left(A_{a_{1},b_{1},\theta_{1}}^{\ast}A_{a_{1},b_{1},\theta_{1}}-\frac{rt}{N}F\right)\mathbf{v}\right\}d\mathbf{v}\\
&=\frac{e^{-\frac{rt}{2N}\tr \tilde{H}_{a_{1},b_{1},\theta_{1}}F}e^{\frac{2N}{t}\eta_{z,t}^{2}}}{K_{a_{1},b_{1},\theta_{1}}}\int_{V^{2}(\R^{N})}\exp\left\{-\frac{N}{t}\mathbf{v}^{\ast}\left(A_{a_{1},b_{1},\theta_{1}}^{\ast}A_{a_{1},b_{1},\theta_{1}}+\eta_{z,t}^{2}-\frac{rt}{N}F\right)\mathbf{v}\right\}d\mathbf{v},
\end{align*}
where $K_{a_{1},b_{1},\theta_{1}}$ is a normalizing constant chosen so that $m_{F}(0)=1$. By Lemma \ref{lemma:fourier}, as in the proof of Lemma \ref{lemma:tildekjcomputation}, we can compute
\begin{align*}
&\int_{V^{2}(\R^{N})}\exp\left\{-\frac{N}{t}\mathbf{v}^{\ast}\left(A_{a_{1},b_{1},\theta_{1}}^{\ast}A_{a_{1},b_{1},\theta_{1}}+\eta_{z,t}^{2}-\frac{rt}{N}F\right)\mathbf{v}\right\}d\mathbf{v}\\
&=C_{N,t}\int_{M^{sa}_{2}(\R)}e^{i\frac{N}{t}\tr P}\left\{\det\left(A_{a_{1},b_{1},\theta_{1}}^{\ast}A_{a_{1},b_{1},\theta_{1}}+\eta_{z,t}^{2}-\frac{rt}{N}F+iP\otimes I_{N}\right)\right\}^{-\frac12}dP\\
&=C_{N,t}\left\{\det\left(A_{a_{1},b_{1},\theta_{1}}^{\ast}A_{a_{1},b_{1},\theta_{1}}+\eta_{z,t}^{2}-\frac{rt}{N}F\right)\right\}^{-\frac12}\\
&\times\int_{M^{sa}_{2}(\R)}e^{i\frac{N}{t}\tr P}\left\{\det\left(1+\sqrt{ \tilde{H}_{a_{1},b_{1},\theta_{1}}^{(rF)}}(iP\otimes I_{N})\sqrt{ \tilde{H}_{a_{1},b_{1},\theta_{1}}^{(rF)}}\right)\right\}^{-\frac12}dP,
\end{align*}
where $ \tilde{H}_{a_{1},b_{1},\theta_{1}}^{(rF)}:=(A_{a_{1},b_{1},\theta_{1}}^{\ast}A_{a_{1},b_{1},\theta_{1}}+\eta_{z,t}^{2}-\frac{rt}{N}F)^{-1}$. Above, and throughout this proof, $C_{N,t}$ is a constant that comes from our application of Lemma \ref{lemma:fourier}; its exact value is not important. Next,
\begin{align*}
&\left\{\det\left(A_{a_{1},b_{1},\theta_{1}}^{\ast}A_{a_{1},b_{1},\theta_{1}}+\eta_{z,t}^{2}-\frac{rt}{N}F\right)\right\}^{-\frac12}\\
&=\left\{\det\left(A_{a_{1},b_{1},\theta_{1}}^{\ast}A_{a_{1},b_{1},\theta_{1}}+\eta_{z,t}^{2}\right)\right\}^{-\frac12}\left\{\det\left(1- \tilde{H}_{a_{1},b_{1},\theta_{1}}^{\frac12}\frac{rt}{N}F \tilde{H}_{a_{1},b_{1},\theta_{1}}^{\frac12}\right)\right\}^{-\frac12}\\
&=\left\{\det\left(A_{a_{1},b_{1},\theta_{1}}^{\ast}A_{a_{1},b_{1},\theta_{1}}+\eta_{z,t}^{2}\right)\right\}^{-\frac12}\exp\left\{-\frac12\tr\log\left(1- \tilde{H}_{a_{1},b_{1},\theta_{1}}^{\frac12}\frac{rt}{N}F \tilde{H}_{a_{1},b_{1},\theta_{1}}^{\frac12}\right)\right\},
\end{align*}
where $\tilde{H}_{a_{1},b_{1},\theta_{1}}:=\tilde{H}^{(0)}_{a_{1},b_{1},\theta_{1}}$. Next, we compute
\begin{align*}
&\int_{M^{sa}_{2}(\R)}e^{i\frac{N}{t}\tr P}\left\{\det\left(1+\sqrt{ \tilde{H}_{a_{1},b_{1},\theta_{1}}^{(rF)}}(iP\otimes I_{N})\sqrt{ \tilde{H}_{a_{1},b_{1},\theta_{1}}^{(rF)}}\right)\right\}^{-\frac12}dP\\
&=\int_{M^{sa}_{2}(\R)}e^{i\frac{N}{t}\tr P}\exp\left\{-\frac12\tr\log\left(1+\sqrt{ \tilde{H}_{a_{1},b_{1},\theta_{1}}^{(rF)}}(iP\otimes I_{N})\sqrt{ \tilde{H}_{a_{1},b_{1},\theta_{1}}^{(rF)}}\right)\right\}dP.
\end{align*}
Putting all the previous displays together, we get
\begin{align}
m(r)&=\frac{C_{N,t}e^{2\frac{N}{t}\eta_{z,t}^{2}}}{K_{a_{1},b_{1},\theta_{1}}}\left\{\det\left(A_{a_{1},b_{1},\theta_{1}}^{\ast}A_{a_{1},b_{1},\theta_{1}}+\eta_{z,t}^{2}\right)\right\}^{-\frac12}\label{eq:mrformula}\\
&\times\exp\left\{-\frac{rt}{2N}\tr \tilde{H}_{a_{1},b_{1},\theta_{1}}F-\frac12\tr\log\left(1- \tilde{H}_{a_{1},b_{1},\theta_{1}}^{\frac12}\frac{rt}{N}F \tilde{H}_{a_{1},b_{1},\theta_{1}}^{\frac12}\right)\right\}\nonumber\\
&\times\int_{M^{sa}_{2}(\R)}e^{i\frac{N}{t}\tr P}\exp\left\{-\frac12\tr\log\left(1+\sqrt{ \tilde{H}_{a_{1},b_{1},\theta_{1}}^{(rF)}}(iP\otimes I_{N})\sqrt{ \tilde{H}_{a_{1},b_{1},\theta_{1}}^{(rF)}}\right)\right\}dP.\nonumber
\end{align}
This identity can only hold if $1- \tilde{H}_{a_{1},b_{1},\theta_{1}}^{\frac12}\frac{rt}{N}F \tilde{H}_{a_{1},b_{1},\theta_{1}}^{\frac12}>0$; we will always have this condition for any $F$ we take. To control the second line, we have the following as in Lemma 6.2 of \cite{MO}:
\begin{align*}
&\exp\left\{-\frac{rt}{2N}\tr \tilde{H}_{a_{1},b_{1},\theta_{1}}F-\frac12\tr\log\left(1- \tilde{H}_{a_{1},b_{1},\theta_{1}}^{\frac12}\frac{rt}{N}F \tilde{H}_{a_{1},b_{1},\theta_{1}}^{\frac12}\right)\right\}\\
&\lesssim\exp\left\{\frac{1}{\|1- \tilde{H}_{a_{1},b_{1},\theta_{1}}^{\frac12}\frac{rt}{N}F \tilde{H}_{a_{1},b_{1},\theta_{1}}^{\frac12}\|_{\mathrm{op}}}\tr\left( \tilde{H}_{a_{1},b_{1},\theta_{1}}^{\frac12}\frac{rt}{N}F \tilde{H}_{a_{1},b_{1},\theta_{1}}^{\frac12}\right)^{2}\right\}.
\end{align*}
Let us now control the $dP$ integral in \eqref{eq:mrformula}. Next, by Taylor expansion as in Section 6 of \cite{MO},
\begin{align*}
&\exp\left\{-\frac12\tr\log\left(1+\sqrt{ \tilde{H}_{a_{1},b_{1},\theta_{1}}^{(rF)}}(iP\otimes I_{N})\sqrt{ \tilde{H}_{a_{1},b_{1},\theta_{1}}^{(rF)}}\right)\right\}\\
&=\exp\left\{-\frac{i}{2}\tr\sqrt{ \tilde{H}_{a_{1},b_{1},\theta_{1}}^{(rF)}}(P\otimes I_{N})\sqrt{ \tilde{H}_{a_{1},b_{1},\theta_{1}}^{(rF)}}-\frac14\tr\left(\sqrt{ \tilde{H}_{a_{1},b_{1},\theta_{1}}^{(rF)}}(P\otimes I_{N})\sqrt{ \tilde{H}_{a_{1},b_{1},\theta_{1}}^{(rF)}}\right)^{2}+\mathcal{E}\right\},
\end{align*}
where $\mathcal{E}$ satisfies both $|e^{\mathcal{E}}|=O(1)$ and $|\mathcal{E}|\lesssim N\eta_{z,t}^{-6}\|P\|_{\mathrm{op}}^{3}$. Note that
\begin{align*}
 \tilde{H}_{a_{1},b_{1},\theta_{1}}^{(rF)}&=\sqrt{ \tilde{H}_{a_{1},b_{1},\theta_{1}}}\left(1-\sqrt{ \tilde{H}_{a_{1},b_{1},\theta_{1}}}\frac{rt}{N}F\sqrt{ \tilde{H}_{a_{1},b_{1},\theta_{1}}}\right)^{-1}\sqrt{ \tilde{H}_{a_{1},b_{1},\theta_{1}}}.
\end{align*}
Now, assume that $\| \tilde{H}_{a_{1},b_{1},\theta_{1}}^{\frac12}\frac{rt}{N}F \tilde{H}_{a_{1},b_{1},\theta_{1}}^{\frac12}\|_{\mathrm{op}}\lesssim\frac{N^{\epsilon}}{N\eta_{z,t}^{3}}$ uniformly in $N$. We will also always have this constraint, and it clearly implies $1- \tilde{H}_{a_{1},b_{1},\theta_{1}}^{\frac12}\frac{rt}{N}F \tilde{H}_{a_{1},b_{1},\theta_{1}}^{\frac12}>0$ if $\eta_{z,t}\gg N^{-\frac13+\epsilon}$. Next, we compute
\begin{align}
&\tr\sqrt{ \tilde{H}_{a_{1},b_{1},\theta_{1}}^{(rF)}}(P\otimes I_{N})\sqrt{ \tilde{H}_{a_{1},b_{1},\theta_{1}}^{(rF)}}=\tr \tilde{H}_{a_{1},b_{1},\theta_{1}}^{(rF)}(P\otimes I_{N})\label{eq:tildeptilde}\\
&=\tr\sqrt{ \tilde{H}_{a_{1},b_{1},\theta_{1}}}\left(1-\sqrt{ \tilde{H}_{a_{1},b_{1},\theta_{1}}}\frac{rt}{N}F\sqrt{ \tilde{H}_{a_{1},b_{1},\theta_{1}}}\right)^{-1}\sqrt{ \tilde{H}_{a_{1},b_{1},\theta_{1}}}(P\otimes I_{N})\nonumber\\
&=\tr \tilde{H}_{a_{1},b_{1},\theta_{1}}(P\otimes I_{N})+O\left(\frac{N^{\epsilon}}{\eta_{z,t}^{5}}\|P\|_{\mathrm{op}}\right),\nonumber
\end{align} 
where the last line also uses $\|\tilde{H}_{a_{1},b_{1},\theta_{1}}\|_{\mathrm{op}}\leq\eta_{z,t}^{-2}$. By the same token, for some constant $C>0$, we have 
\begin{align}
&\left(1-\frac{CN^{\epsilon}}{\eta_{z,t}^{10}}\|P\|_{\mathrm{op}}^{2}\right)\tr\left(\sqrt{ \tilde{H}_{a_{1},b_{1},\theta_{1}}}(P\otimes I_{N})\sqrt{ \tilde{H}_{a_{1},b_{1},\theta_{1}}}\right)^{2}\label{eq:tildeptildesquared}\\
&\leq\tr\left(\sqrt{ \tilde{H}_{a_{1},b_{1},\theta_{1}}^{(rF)}}(P\otimes I_{N})\sqrt{ \tilde{H}_{a_{1},b_{1},\theta_{1}}^{(rF)}}\right)^{2}\nonumber\\
&\leq\left(1-\frac{CN^{\epsilon}}{\eta_{z,t}^{10}}\|P\|_{\mathrm{op}}^{2}\right)\tr\left(\sqrt{ \tilde{H}_{a_{1},b_{1},\theta_{1}}}(P\otimes I_{N})\sqrt{ \tilde{H}_{a_{1},b_{1},\theta_{1}}}\right)^{2}.\nonumber
\end{align}
This has two consequences. First, by the lower bound above and the lower bound $\tilde{H}_{a_{1},b_{1},\theta_{1}}\gtrsim1$, 
\begin{align}
\tr\left(\sqrt{ \tilde{H}_{a_{1},b_{1},\theta_{1}}^{(rF)}}(P\otimes I_{N})\sqrt{ \tilde{H}_{a_{1},b_{1},\theta_{1}}^{(rF)}}\right)^{2}\gtrsim N|P_{11}|^{2}+N|P_{22}|^{2}+N|P_{12}|^{2}.\label{eq:tildeptildesquaredlower}
\end{align}
Thus, we can restrict the $P$ integration to $|P_{11}|+|P_{12}|+|P_{22}|\lesssim N^{-1/2}\log N$ at the cost of something exponentially small in $N$. In particular, we have $\mathcal{E}=O(N^{-1/2+\kappa})$ for some $\kappa>0$ small. This bound on entries of $P$ combined with \eqref{eq:tildeptilde} also implies that
\begin{align*}
&\tr\sqrt{ \tilde{H}_{a_{1},b_{1},\theta_{1}}^{(rF)}}(P\otimes I_{N})\sqrt{ \tilde{H}_{a_{1},b_{1},\theta_{1}}^{(rF)}}=\tr \tilde{H}_{a_{1},b_{1},\theta_{1}}^{(rF)}(P\otimes I_{N})\\
&=\tr \tilde{H}_{a_{1},b_{1},\theta_{1}}(P\otimes I_{N})+O\left(\frac{N^{2\kappa}}{\sqrt{N}\eta_{z,t}^{5}}\right).
\end{align*}
Moreover, in the region $|P_{11}|+|P_{12}|+|P_{22}|\lesssim N^{-1/2}\log N$, we have
\begin{align*}
\tr\left(\sqrt{ \tilde{H}_{a_{1},b_{1},\theta_{1}}}(P\otimes I_{N})\sqrt{ \tilde{H}_{a_{1},b_{1},\theta_{1}}}\right)^{2}\lesssim \eta_{z,t}^{-C}\log^{2}N.
\end{align*}
By combining this with \eqref{eq:tildeptildesquared}, in this region, we have the following for $\kappa>0$ small:
\begin{align*}
\tr\left(\sqrt{ \tilde{H}_{a_{1},b_{1},\theta_{1}}^{(rF)}}(P\otimes I_{N})\sqrt{ \tilde{H}_{a_{1},b_{1},\theta_{1}}^{(rF)}}\right)^{2}=\tr\left(\sqrt{ \tilde{H}_{a_{1},b_{1},\theta_{1}}}(P\otimes I_{N})\sqrt{ \tilde{H}_{a_{1},b_{1},\theta_{1}}}\right)^{2}+O\left(\frac{N^{\kappa}}{N\eta_{z,t}^{10}}\right).
\end{align*}
Ultimately, for possibly different but still small $\kappa>0$, we get
\begin{align*}
&\exp\left\{-\frac{i}{2}\tr\sqrt{ \tilde{H}_{a_{1},b_{1},\theta_{1}}^{(rF)}}(P\otimes I_{N})\sqrt{ \tilde{H}_{a_{1},b_{1},\theta_{1}}^{(rF)}}-\frac14\tr\left(\sqrt{ \tilde{H}_{a_{1},b_{1},\theta_{1}}^{(rF)}}(P\otimes I_{N})\sqrt{ \tilde{H}_{a_{1},b_{1},\theta_{1}}^{(rF)}}\right)^{2}+\mathcal{E}\right\}\nonumber\\
&=\exp\left\{-\frac{i}{2}\tr\sqrt{ \tilde{H}_{a_{1},b_{1},\theta_{1}}}(P\otimes I_{N})\sqrt{ \tilde{H}_{a_{1},b_{1},\theta_{1}}}-\frac14\tr\left(\sqrt{ \tilde{H}_{a_{1},b_{1},\theta_{1}}}(P\otimes I_{N})\sqrt{ \tilde{H}_{a_{1},b_{1},\theta_{1}}}\right)^{2}\right\}\\
&\times\exp\left\{iO\left(\frac{N^{\kappa}}{\sqrt{N}\eta_{z,t}^{5}}\right)+O\left(\frac{N^{\kappa}}{N\eta_{z,t}^{10}}\right)\right\}.
\end{align*}
Note that the second line above, which contains the main term on the RHS of this identity, is independent of $r$. Thus, by the previous display, we have $m_{F}(r)\lesssim m_{F}(0)$ as long as $F$ is Hermitian and $\| \tilde{H}_{a_{1},b_{1},\theta_{1}}^{\frac12}\frac{rt}{N}F \tilde{H}_{a_{1},b_{1},\theta_{1}}^{\frac12}\|_{\mathrm{op}}\lesssim\frac{N^{\epsilon}}{N\eta_{z,t}^{3}}\ll1$. Recalling that $m_{F}(0)=1$, under this condition, we have $m_{F}(r)=O(1)$. Concentration of matrix entries of $V_{1}^{\ast}G_{z}V_{1}$ now follows as Lemma 6.2 in \cite{MO}. In particular, we will apply this inequality for $r=N^{-\kappa}$ with $\kappa>0$ small and for the following choices of $F$:
\begin{align*}
F&=\eta_{z,t}\re\left[E\otimes H_{z}\right],\\
F&=\eta_{z,t}\im\left[E\otimes H_{z}\right],\\
F&=\eta_{z,t}\re\left[E\otimes\tilde{ H}_{z}\right],\\
F&=\eta_{z,t}\im\left[E\otimes\tilde{ H}_{z}\right],\\
F&=\re\left[E\otimes H_{z}(A-z)\right],\\
F&=\im\left[E\otimes H_{z}(A-z)\right],\\
\end{align*}
where $E$ is any of the four matrices
\begin{align}
\begin{pmatrix}1&0\\0&0\end{pmatrix},\begin{pmatrix}0&1\\0&0\end{pmatrix},\begin{pmatrix}0&0\\1&0\end{pmatrix},\begin{pmatrix}0&0\\0&1\end{pmatrix}.\nonumber 
\end{align}
and $\re X = \frac{1}{2}(X^\ast+X)$, $\im X = \frac{i}{2}(X^\ast-X)$. (We note that our choice of $r$ is larger than the choices made in the proof of Lemma 6.2 in \cite{MO}, giving weaker concentration estimates, but this is fine.) Let us illustrate one example. Upon treating $\mu_{j}$ as a probability measure over $\mathbf{v}$, by the Markov inequality, we have 
\begin{align*}
\mu_{j}\left(\mathbf{v}^{\ast}F\mathbf{v}-\frac{t}{N}\tr \tilde{H}_{a_{1},b_{1},\theta_{1}}F\geq r\right)\leq e^{-Kr^{2}}m_{KF}(r).
\end{align*}
Choose $F=\eta_{z,t}\re\left[E\otimes H_{z}\right]$ and $r=N^{-\kappa}$ and $K=N^{3\kappa}$ for $\kappa>0$ small. Then we have the estimate $\|\tilde{H}_{a_{1},b_{1},\theta_{1}}^{\frac12}\frac{rt}{N}KF\tilde{H}_{a_{1},b_{1},\theta_{1}}^{\frac12}\|_{\mathrm{op}}\lesssim N^{-1+2\kappa}\eta_{z,t}^{-3}\lesssim1$, so $m_{KF}(r)=O(1)$. This shows that the event on the LHS is exponentially small in $\mu_{j}$-measure. By the same token, we know the same is true if we replace $\geq r$ on the LHS by $\leq r$, and we can replace $\re[E\otimes H_{z}]$ by $\im[E \otimes H_{z}]$. For example, this implies $v_{1}^{\ast}H_{z}v_{1}=\frac{\eta_{z,t}t}{N}\tr \tilde{H}_{a_{1},b_{1},\theta_{1}}H_{z}+O(N^{-\kappa})$ outside an event of $\mu_{j}$-probability that is exponentially small. Doing the same procedure for each entry of $V_{1}^{\ast}G_{z}V_{1}$ shows that each said entry matches the corresponding entry of $\mathcal{G}$ up to $O(N^{-\kappa})$ outside an event of $\mu_{j}$-probability that is exponentially small. Since the determinant is smooth in its entries, and since the entries of $V_{1}^{\ast}G_{z}V_{1}$ are $O(\eta_{z,t}^{-1})=O(N^{-\epsilon_{0}})$ deterministically with $\epsilon_{0}>0$ small, we deduce
\begin{align*}
\int_{V^{2}(\R^{N})}|\det[V_{1}^{\ast}G_{z}(\eta_{z,t})V_{1}]|^{p}d\mu_{1}(\mathbf{v}_{1})&=|\det\mathcal{G}|^{p}+O(N^{-\kappa}).
\end{align*}
We now claim that $|\det\mathcal{G}|^{p}\gtrsim \eta_{z,t}^{-D}$ for some $D=O(1)$, since we can take $\eta_{z,t}=N^{-\epsilon_{0}}$ for $\epsilon_{0}>0$ small enough so that $\eta_{z,t}^{-D}\gg N^{-\kappa}$. This implies that the RHS of the previous display is $|\det\mathcal{G}|^{p}[1+O(N^{-\kappa})]$ for possibly different $\kappa>0$. But said claim is equivalent to $|\det[V_{1}^{\ast}G_{z}(\eta_{z,t})V_{1}]|\gtrsim \eta_{z,t}^{-D}$ because of the previous display, and this bound is shown in the proof of Lemma \ref{lemma:cutoff1}.
\end{proof}
By the same argument, we deduce the following analog of Lemma \ref{lemma:concentration} but for arbitrary $j\geq1$. 
\begin{lemma}\label{lemma:concentrationj}
For any $p\geq1$ and $j\geq1$, there exists $\kappa>0$ such that 
\begin{align*}
\int_{V^{2}(\R^{N-2j+2)}}|\det[V_{j}^{\ast}G_{z}^{(j-1)}(\eta_{z,t})V_{j}]|^{p}d\mu_{j}(\mathbf{v}_{j})=|\det\mathcal{G}_{j}|^{p}\left[1+O(N^{-\kappa})\right],
\end{align*}
where $\mathcal{G}_{j}$ is the following $4\times4$ matrix:
\begin{align*}
\mathcal{G}_{j,a_{j},b_{j}}&:=\begin{pmatrix}\mathcal{G}_{j,a_{j},b_{j},11}&\mathcal{G}_{j,a_{j},b_{j},12}\\\mathcal{G}_{j,a_{j},b_{j},21}&\mathcal{G}_{j,a_{j},b_{j},22}\end{pmatrix}\\
\mathcal{G}_{j,a_{j},b_{j},11}&:=\begin{pmatrix}i\eta_{z,t}\frac{t}{N}\tr \tilde{H}_{a_{j},b_{j},\theta_{j}}^{(j-1)}( H_{z}^{(j-1)}\otimes E_{11})&i\eta_{z,t}\frac{t}{N}\tr \tilde{H}_{a_{j},b_{j},\theta_{j}}^{(j-1)}( H_{z}^{(j-1)}\otimes E_{12})\\i\eta_{z,t}\frac{t}{N}\tr \tilde{H}_{a_{j},b_{j},\theta_{j}}^{(j-1)}( H_{z}^{(j-1)}\otimes E_{21})&i\eta_{z,t}\frac{t}{N}\tr \tilde{H}_{a_{j},b_{j},\theta_{j}}^{(j-1)}( H_{z}^{(j-1)}\otimes E_{22})\end{pmatrix}\\
\mathcal{G}_{j,a_{j},b_{j},12}&:=\begin{pmatrix}\frac{t}{N}\tr \tilde{H}_{a_{j},b_{j},\theta_{j}}^{(j-1)}[ H_{z}^{(j-1)}(A^{(j-1)}-z)\otimes E_{11}]&\frac{t}{N}\tr \tilde{H}_{a_{j},b_{j},\theta_{j}}^{(j-1)}[ H_{z}^{(j-1)}(A^{(j-1)}-z)\otimes E_{12}]\\\frac{t}{N}\tr \tilde{H}_{a_{j},b_{j},\theta_{j}}^{(j-1)}[ H_{z}^{(j-1)}(A^{(j-1)}-z)\otimes E_{21}]&\frac{t}{N}\tr \tilde{H}_{a_{j},b_{j},\theta_{j}}^{(j-1)}[ H_{z}^{(j-1)}(A^{(j-1)}-z)\otimes E_{22}]\end{pmatrix}\\
\mathcal{G}_{j,a_{j},b_{j},21}&:=\begin{pmatrix}\frac{t}{N}\tr \tilde{H}_{a_{j},b_{j},\theta_{j}}^{(j-1)}[(A^{(j-1)}-z)^{\ast} H_{z}^{(j-1)}\otimes E_{11}]&\frac{t}{N}\tr \tilde{H}_{a_{j},b_{j},\theta_{j}}^{(j-1)}[(A^{(j-1)}-z)^{\ast} H_{z}^{(j-1)}\otimes E_{12}]\\\frac{t}{N}\tr \tilde{H}_{a_{j},b_{j},\theta_{j}}^{(j-1)}[(A^{(j-1)}-z)^{\ast} H_{z}^{(j-1)}\otimes E_{21}]&\frac{t}{N}\tr \tilde{H}_{a_{j},b_{j},\theta_{j}}^{(j-1)}[(A^{(j-1)}-z)^{\ast} H_{z}^{(j-1)}\otimes E_{22}]\end{pmatrix}\\
\mathcal{G}_{j,a_{j},b_{j},22}&:=\begin{pmatrix}i\eta_{z,t}\frac{t}{N}\tr \tilde{H}_{a_{j},b_{j},\theta_{j}}^{(j-1)}(\tilde{H}_{z}^{(j-1)}\otimes E_{11})&i\eta_{z,t}\frac{t}{N}\tr \tilde{H}_{a_{j},b_{j},\theta_{j}}^{(j-1)}(\tilde{H}_{z}^{(j-1)}\otimes E_{12})\\i\eta_{z,t}\frac{t}{N}\tr \tilde{H}_{a_{j},b_{j},\theta_{j}}^{(j-1)}(\tilde{H}_{z}^{(j-1)}\otimes E_{21})&i\eta_{z,t}\frac{t}{N}\tr \tilde{H}_{a_{j},b_{j},\theta_{j}}^{(j-1)}(\tilde{H}_{z}^{(j-1)}\otimes E_{22})\end{pmatrix},\\
E_{11}&:=\begin{pmatrix}1&0\\0&0\end{pmatrix},\quad E_{22}:=\begin{pmatrix}0&0\\0&1\end{pmatrix}\quad
E_{12}:=\begin{pmatrix}0&1\\0&0\end{pmatrix}\quad
E_{21}:=\begin{pmatrix}0&0\\1&0\end{pmatrix}.
\end{align*}
\end{lemma}
The final step in this computation is to show that in $\mathcal{G}_{j,a_{j},b_{j}}$, we can replace $a_{j},b_{j}$ by $a,b$ (i.e. replace $\lambda_{j}$ by $z$) and remove the $(j-1)$-superscripts. For the exact formula for $\mathcal{G}_{j,a,b}$ appearing in Lemma \ref{lemma:concentrationjzjtoz} below, see \eqref{eq:gjabmatrix}.
\begin{lemma}\label{lemma:concentrationjzjtoz}
Define $\mathcal{G}_{j,a,b}$ in the same way as in Lemma \ref{lemma:concentrationj} but replacing $(a_{j},b_{j})$ by $(a,b)$ and $\theta_{j}$ by $\pi/4$ and removing the $(j-1)$ superscript. Then we have $|\det\mathcal{G}_{j,a_{j},b_{j}}|^{p}=|\det\mathcal{G}_{j,a,b}|^{p}[1+O(N^{-\kappa})]$ for some $\kappa>0$.
\end{lemma}
\begin{proof}
Resolvent perturbation implies that $\|\tilde{H}_{a_{j},b_{j},\theta_{j}}^{(j-1)}-\tilde{H}_{a,b,\frac{\pi}{4}}^{(j-1)}\|_{\mathrm{op}}\lesssim\eta_{z,t}^{-C}[|a_{j}-a|+|b_{j}-b|+|\theta_{j}-\frac{\pi}{4}]$. But $a_{j}-a,b_{j}-b=O(N^{-1/2})$ and $|\theta_{j}-\frac{\pi}{4}|=O(N^{-1/2+\kappa}\eta_{z,t}^{-1})$. Moreover, if we use Cauchy interlacing as in \eqref{eq:interlacingdemo}, then for any $k,\ell\in\{1,2\}$, we have 
\begin{align*}
i\eta_{z,t}\frac{t}{N}\left|\tr \tilde{H}_{a_{j},b_{j},\frac{\pi}{4}}^{(j-1)}(\tilde{H}_{z}^{(j-1)}\otimes E_{11})-\tr \tilde{H}_{a_{j},b_{j},\frac{\pi}{4}}(\tilde{H}_{z}\otimes E_{11})\right|\lesssim \tfrac{t}{N}\|\tilde{H}_{a_{j},b_{j},\theta_{j}}\|_{\mathrm{op}}\|H_{z}\|_{\mathrm{op}}\lesssim N^{-1}\eta_{z,t}^{-C}.
\end{align*}
A similar estimate for removing $(j-1)$-superscripts in entries of $\mathcal{G}_{j,a_{j},b_{j}}$ also holds. Thus, the entries of $\mathcal{G}_{j,a_{j},b_{j}}-\mathcal{G}_{j,a,b}$ are $O(N^{-1/2+\kappa+C\epsilon_{0}})$. Since the determinant is smooth in its entries, we deduce that $|\det\mathcal{G}_{j,a_{j},b_{j}}|^{p}=|\det\mathcal{G}_{j,a,b}|^{p}+O(N^{-1/2+4\epsilon_{0}})$. Now, again use the lower bound $|\det\mathcal{G}_{j,a_{j},b_{j}}|^{p}\gtrsim \eta_{z,t}^{-D}$ for some $D=O(1)$ to get $|\det\mathcal{G}_{j,a_{j},b_{j}}|^{p}\gtrsim \eta_{z,t}^{-D}$. This implies that $|\det\mathcal{G}_{j,a,b}|^{p}+O(N^{-1/2+4\epsilon_{0}})=|\det\mathcal{G}_{j,a,b}|^{p}[1+O(N^{-1/2+4\epsilon_{0}})]$ and completes the proof.
\end{proof}
\subsection{The \texorpdfstring{$dP$}{dP} integration over \texorpdfstring{$M_{2}^{sa}(\R)$}{M\_2\^sa(R)}}
We start by a similar computation to the $dP$ integration in the proof of Lemma \ref{lemma:concentration}. By $\det[1+A]=\exp\tr\log[1+A]$ and Taylor expansion of $\log$, we have
\begin{align}
&\det[I+i[\tilde{H}_{a_{j},b_{j},\theta_{j}}^{(j-1)}]^{\frac12}(P\otimes I)[\tilde{H}_{a_{j},b_{j},\theta_{j}}^{(j-1)}]^{\frac12}]^{-\frac12}\label{eq:gaussapprox1}\\
&=\exp\left\{-\frac{i}{2}\tr[\tilde{H}_{a_{j},b_{j},\theta_{j}}^{(j-1)}]^{\frac12}(P\otimes I)[\tilde{H}_{a_{j},b_{j},\theta_{j}}^{(j-1)}]^{\frac12}\right\}\nonumber\\
&\times\exp\left\{-\frac12\tr\left([\tilde{H}_{a_{j},b_{j},\theta_{j}}^{(j-1)}]^{\frac12}(P\otimes I)[\tilde{H}_{a_{j},b_{j},\theta_{j}}^{(j-1)}](P\otimes I)[\tilde{H}_{a_{j},b_{j},\theta_{j}}^{(j-1)}]^{\frac12}\right)\right\}\nonumber\\
&\times\exp\left\{O(N\eta_{z,t}^{-6}\|P\|_{\mathrm{op}}^{3})\right\}.\nonumber
\end{align}
We also have the following bound (by an essentially identical Taylor expansion) which is an inequality version of the previous display:
\begin{align}
&\det[I+i[\tilde{H}_{a_{j},b_{j},\theta_{j}}^{(j-1)}]^{\frac12}(P\otimes I)[\tilde{H}_{a_{j},b_{j},\theta_{j}}^{(j-1)}]^{\frac12}]^{-\frac12}\label{eq:gaussapprox2}\\
&\lesssim\exp\left\{-\frac12\tr\left([\tilde{H}_{a_{j},b_{j},\theta_{j}}^{(j-1)}]^{\frac12}(P\otimes I)[\tilde{H}_{a_{j},b_{j},\theta_{j}}^{(j-1)}](P\otimes I)[\tilde{H}_{a_{j},b_{j},\theta_{j}}^{(j-1)}]^{\frac12}\right)\right\}.\nonumber
\end{align}
If we write $P=\begin{pmatrix}P_{11}&P_{12}\\P_{12}&P_{22}\end{pmatrix}$, then an elementary computation shows
\begin{align*}
&\exp\left\{-\frac12\tr\left([\tilde{H}_{a_{j},b_{j},\theta_{j}}^{(j-1)}]^{\frac12}(P\otimes I)[\tilde{H}_{a_{j},b_{j},\theta_{j}}^{(j-1)}](P\otimes I)[\tilde{H}_{a_{j},b_{j},\theta_{j}}^{(j-1)}]^{\frac12}\right)\right\}\\
&=\exp\left\{-\begin{pmatrix}P_{11}\\P_{12}\\P_{22}\end{pmatrix}^{\ast}\mathbf{Q}_{a_{j},b_{j},\theta_{j}}^{(j-1)}\begin{pmatrix}P_{11}\\P_{12}\\P_{22}\end{pmatrix}\right\},
\end{align*}
in which we use the notation
\begin{align*}
\tilde{H}_{a_{j},b_{j},\theta_{j}}^{(j-1)}&=\begin{pmatrix}\tilde{H}_{a_{j},b_{j},\theta_{j},11}^{(j-1)}&\tilde{H}_{a_{j},b_{j},\theta_{j},12}^{(j-1)}\\\tilde{H}_{a_{j},b_{j},\theta_{j},12}^{(j-1)}&\tilde{H}_{a_{j},b_{j},\theta_{j},22}^{(j-1)}\end{pmatrix}\\
\mathbf{Q}_{a_{j},b_{j},\theta_{j}}^{(j-1)}&:=\begin{pmatrix}[\mathbf{Q}_{a_{j},b_{j},\theta_{j}}^{(j-1)}]_{1}&[\mathbf{Q}_{a_{j},b_{j},\theta_{j}}^{(j-1)}]_{2}&[\mathbf{Q}_{a_{j},b_{j},\theta_{j}}^{(j-1)}]_{3}\end{pmatrix},\\
[\mathbf{Q}_{a_{j},b_{j},\theta_{j}}^{(j-1)}]_{1}&=\begin{pmatrix}\tr[\tilde{H}_{a_{j},b_{j},\theta_{j},11}^{(j-1)}]^{2}\\2\tr \tilde{H}_{a_{j},b_{j},\theta_{j},11}^{(j-1)} \tilde{H}_{a_{j},b_{j},\theta_{j},12}^{(j-1)}\\\tr \tilde{H}_{a_{j},b_{j},\theta_{j},12}^{(j-1)} [\tilde{H}_{a_{j},b_{j},\theta_{j},12}^{(j-1)}]^T\end{pmatrix}\\
[\mathbf{Q}_{a_{j},b_{j},\theta_{j}}^{(j-1)}]_{2}&=\begin{pmatrix}2\tr \tilde{H}_{a_{j},b_{j},\theta_{j},11}^{(j-1)} \tilde{H}_{a_{j},b_{j},\theta_{j},12}^{(j-1)}\\2\tr \tilde{H}_{a_{j},b_{j},\theta_{j},11}^{(j-1)} \tilde{H}_{a_{j},b_{j},\theta_{j},22}^{(j-1)}+2\tr [\tilde{H}_{a_{j},b_{j},\theta_{j},12}^{(j-1)}]^{2}\\2\tr \tilde{H}_{a_{j},b_{j},\theta_{j},22}^{(j-1)} \tilde{H}_{a_{j},b_{j},\theta_{j},12}^{(j-1)}\end{pmatrix}\\
[\mathbf{Q}_{a_{j},b_{j},\theta_{j}}^{(j-1)}]_{3}&=\begin{pmatrix}\tr \tilde{H}_{a_{j},b_{j},\theta_{j},12}^{(j-1)} [\tilde{H}_{a_{j},b_{j},\theta_{j},12}^{(j-1)}]^{T}\\2\tr \tilde{H}_{a_{j},b_{j},\theta_{j},22}^{(j-1)} \tilde{H}_{a_{j},b_{j},\theta_{j},12}^{(j-1)}\\\tr [\tilde{H}_{a_{j},b_{j},\theta_{j},22}^{(j-1)}]^{2}\end{pmatrix}
\end{align*}
We clarify that $\mathbf{Q}_{a_{j},b_{j},\theta_{j}}^{(j-1)}$ is a $3\times3$ matrix. To continue with this computation, we need the following auxiliary bound. (It is a refinement of \eqref{eq:tildeptildesquaredlower}.)
\begin{lemma}\label{lemma:qlowerbound}
If $\theta\in\mathcal{I}_{0}$, then $\mathbf{Q}_{a_{j},b_{j},\theta_{j}}^{(j-1)}\gtrsim N\eta_{z,t}^{-2}$.
\end{lemma}
\begin{proof}
This is shown for $\theta_{j}=\pi/4$ in Appendix \ref{subsection:Qestimatepi/4}. For general $\theta_{j}\in\mathcal{I}_{0}$, we will perform resolvent perturbation in the entries with respect to $\theta_{j}$. We show one example; we claim that 
\begin{align*}
\tr[\tilde{H}_{a_{j},b_{j},\theta_{j},11}^{(j-1)}]^{2}=\tr[\tilde{H}_{a_{j},b_{j},\frac{\pi}{4},11}^{(j-1)}]^{2}+O(N^{\frac12+\delta})
\end{align*}
for $\delta>0$ small. For this, note $\|\tilde{H}_{a_{j},b_{j},\theta_{j},11}^{(j-1)}\|_{\mathrm{op}}\lesssim\eta_{z,t}^{-2}$. Thus, resolvent perturbation shows 
\begin{align*}
\|[\tilde{H}_{a_{j},b_{j},\theta_{j},11}^{(j-1)}]^{2}-[\tilde{H}_{a_{j},b_{j},\frac{\pi}{4},11}^{(j-1)}]^{2}\|_{\mathrm{op}}\lesssim \eta_{z,t}^{-D}\left\|(\Lambda_{a_{j},b_{j},\theta_{j}}-\Lambda_{a_{j},b_{j},\frac{\pi}{4}})\otimes I_{N-2j}\right\|_{\mathrm{op}}
\end{align*}
for some $D=O(1)$. Since $\theta\in\mathcal{I}_{0}$, we know $\|\Lambda_{a_{j},b_{j},\theta_{j}}-\Lambda_{a_{j},b_{j},\frac{\pi}{4}}\|_{\mathrm{op}}\lesssim|\theta-\frac{\pi}{4}|\lesssim N^{-\frac12+\tau}\eta_{z,t}^{-1}$ with $\tau>0$ small. So the RHS of the previous display is $\lesssim N^{-1/2+\tau}\eta_{z,t}^{-D}$ for different $D=O(1)$ and $\tau>0$ small. We deduce
\begin{align*}
|\tr[\tilde{H}_{a_{j},b_{j},\theta_{j},11}^{(j-1)}]^{2}-\tr[\tilde{H}_{a_{j},b_{j},\frac{\pi}{4},11}^{(j-1)}]^{2}|&\lesssim N\|[\tilde{H}_{a_{j},b_{j},\theta_{j},11}^{(j-1)}]^{2}-[\tilde{H}_{a_{j},b_{j},\frac{\pi}{4},11}^{(j-1)}]^{2}\|_{\mathrm{op}}\lesssim N^{\frac12+\tau}\eta_{z,t}^{-D}
\end{align*}
at which point the claim follows if $\epsilon_{0}>0$ in $\eta_{z,t}=N^{-\epsilon_{0}}$ and $\tau>0$ are both small enough. A similar estimate holds for all entries of $\mathbf{Q}_{a_{j},b_{j},\theta_{j}}^{(j-1)}$. This implies 
\begin{align*}
\|\mathbf{Q}_{a_{j},b_{j},\theta_{j}}^{(j-1)}-\mathbf{Q}_{a_{j},b_{j},\frac{\pi}{4}}^{(j-1)}\|_{\mathrm{op}}\lesssim N^{\frac12+\delta}
\end{align*}
with $\delta>0$ small. Since the lemma is true for $\theta_{j}=\pi/4$ as noted at the beginning of this proof, and since $N\eta_{z,t}^{-2}\gg N^{1/2+\delta}$ for $\delta>0$ small, the lemma must be true for all $\theta_{j}\in\mathcal{I}_{0}$.
\end{proof}
By Lemma \ref{lemma:qlowerbound} and \eqref{eq:gaussapprox2}, we deduce the following for any $\upsilon>0$
\begin{align*}
&\left|\int_{\|P\|_{\mathrm{op}}\geq N^{-\frac12+\upsilon}\eta_{z,t}}e^{i\frac{N}{2t}\tr P}\det[I+i[\tilde{H}_{a_{j},b_{j},\theta_{j}}^{(j-1)}]^{\frac12}(P\otimes I)[\tilde{H}_{a_{j},b_{j},\theta_{j}}^{(j-1)}]^{\frac12}]^{-\frac12}dP\right|\lesssim \exp[-CN^{2\upsilon}].
\end{align*}
In particular, we can restrict to the region $\|P\|_{\mathrm{op}}\leq N^{-1/2+\upsilon}\eta_{z,t}$. In this region, we can use the identity \eqref{eq:gaussapprox1} and control the last line therein (the cubic term) by $N^{-1/2+3\tau}$ for $\tau>0$ small. In particular, we have 
\begin{align*}
&\int_{M_{2}^{sa}(\R)}e^{i\frac{N}{2t}\tr P}\det[I+i[\tilde{H}_{a_{j},b_{j},\theta_{j}}^{(j-1)}]^{\frac12}(P\otimes I)[\tilde{H}_{a_{j},b_{j},\theta_{j}}^{(j-1)}]^{\frac12}]^{-\frac12}dP\\
&=[1+O(N^{-\kappa})]\int_{\|P\|_{\mathrm{op}}\leq N^{-1/2+\upsilon}\eta_{z,t}}e^{i\frac{N}{2t}P}\exp\left\{-\frac{i}{2}\tr[\tilde{H}_{a_{j},b_{j},\theta_{j}}^{(j-1)}]^{\frac12}(P\otimes I)[\tilde{H}_{a_{j},b_{j},\theta_{j}}^{(j-1)}]^{\frac12}\right\}\\
&\quad\quad\quad\times\exp\left\{-\frac12\tr\left([\tilde{H}_{a_{j},b_{j},\theta_{j}}^{(j-1)}]^{\frac12}(P\otimes I)[\tilde{H}_{a_{j},b_{j},\theta_{j}}^{(j-1)}](P\otimes I)[\tilde{H}_{a_{j},b_{j},\theta_{j}}^{(j-1)}]^{\frac12}\right)\right\}d P+O(\exp[-CN^{2\upsilon}]).
\end{align*}
A similar argument lets us extend the integration back to all $M_{2}^{sa}(\R)$. We ultimately have
\begin{align*}
&\int_{M_{2}^{sa}(\R)}e^{i\frac{N}{2t}\tr P}\det[I+i[\tilde{H}_{a_{j},b_{j},\theta_{j}}^{(j-1)}]^{\frac12}(P\otimes I)[\tilde{H}_{a_{j},b_{j},\theta_{j}}^{(j-1)}]^{\frac12}]^{-\frac12}dP\\
&=[1+O(N^{-\kappa})]\int_{M_{2}^{sa}(\R)}e^{i\frac{N}{2t}\tr P}\exp\left\{-\frac{i}{2}\tr[\tilde{H}_{a_{j},b_{j},\theta_{j}}^{(j-1)}]^{\frac12}(P\otimes I)[\tilde{H}_{a_{j},b_{j},\theta_{j}}^{(j-1)}]^{\frac12}\right\}\\
&\quad\quad\quad\exp\left\{-\begin{pmatrix}P_{11}\\P_{12}\\P_{22}\end{pmatrix}^{\ast}\mathbf{Q}_{a_{j},b_{j},\theta_{j}}^{(j-1)}\begin{pmatrix}P_{11}\\P_{12}\\P_{22}\end{pmatrix}\right\}dP+O(\exp[-CN^{2\upsilon}]).
\end{align*}
We can compute the $dP$ integration above since it is just a Gaussian Fourier transform; this gives
\begin{align}
&\int_{M_{2}^{sa}(\R)}e^{i\frac{N}{2t}\tr P}\det[I+i[\tilde{H}_{a_{j},b_{j},\theta_{j}}^{(j-1)}]^{\frac12}(P\otimes I)[\tilde{H}_{a_{j},b_{j},\theta_{j}}^{(j-1)}]^{\frac12}]^{-\frac12}dP\label{eq:gaussapprox}\\
&=\mathbf{C}_{1}\left[1+O(e^{-CN^{2\upsilon}})\right]|\det\mathbf{Q}_{a_{j},b_{j},\theta_{j}}^{(j-1)}|^{-\frac12}\nonumber\\
&\times\exp\left\{-\mathbf{C}_{2}\begin{pmatrix}\frac{N}{t}-\tr \tilde{H}_{a_{j},b_{j},\theta_{j},11}^{(j-1)}\\-2\tr \tilde{H}_{a_{j},b_{j},\theta_{j},12}^{(j-1)}\\\frac{N}{t}-\tr \tilde{H}_{a_{j},b_{j},\theta_{j},22}^{(j-1)}\end{pmatrix}^{\ast}[\mathbf{Q}_{a_{j},b_{j},\theta_{j}}^{(j-1)}]^{-1}\begin{pmatrix}\frac{N}{t}-\tr \tilde{H}_{a_{j},b_{j},\theta_{j},11}^{(j-1)}\\-2\tr \tilde{H}_{a_{j},b_{j},\theta_{j},12}^{(j-1)}\\\frac{N}{t}-\tr \tilde{H}_{a_{j},b_{j},\theta_{j},22}^{(j-1)}\end{pmatrix}\right\}\nonumber\\
&+O(\exp[-CN^{2\upsilon}]).\nonumber
\end{align}
Above, $\mathbf{C}_{1}$ and $\mathbf{C}_{2}$ are constants coming from the Gaussian integration. We now remove superscripts $(j-1)$. By Cauchy interlacing, we know $\tr \tilde{H}_{a_{j},b_{j},\theta_{j},k\ell}^{(j-1)}=\tr \tilde{H}_{a_{j},b_{j},\theta_{j},k\ell}+O(\eta_{z,t}^{-2})$, where the big-O comes from a trivial bound on the operator norm of $\tilde{H}_{a_{j},b_{j},\theta_{j}}$.

Now, note $\mathbf{Q}_{a_{j},b_{j},\theta_{j}}^{(j-1)}\gtrsim N\eta_{z,t}^{-2}$. We now claim that 
\begin{align}
\left\|\begin{pmatrix}\frac{N}{t}-\tr \tilde{H}_{a_{j},b_{j},\theta_{j},11}^{(j-1)}\\-2\tr \tilde{H}_{a_{j},b_{j},\theta_{j},12}^{(j-1)}\\\frac{N}{t}-\tr \tilde{H}_{a_{j},b_{j},\theta_{j},22}^{(j-1)}\end{pmatrix}\right\|\lesssim N^{\frac12+\kappa}\eta_{z,t}^{-D},\label{eq:Vestimate}
\end{align}
where $\kappa>0$ is small and $D=O(1)$. To see this, we first remark that if $a_{j}=a$ and $b_{j}=b$ and $\theta_{j}=\frac{\pi}{4}$ and $j-1=0$, then the vector on the LHS is the zero vector; see Appendix \ref{section:resolventsatpi/4}. We now use Cauchy interlacing and $a_{j}-a=O(N^{-1/2})$ and $b_{j}-b=O(N^{-1/2})$ and $\theta_{j}-\frac{\pi}{4}=O(N^{-1/2+\kappa}\eta_{z,t}^{-1})$ for $\theta_{j}\in\mathcal{I}_{0}$ and standard Green's function perturbations to conclude. 

We now combine the previous display with $\tr \tilde{H}_{a_{j},b_{j},\theta_{j},k\ell}^{(j-1)}=\tr \tilde{H}_{a_{j},b_{j},\theta_{j},k\ell}+O(\eta_{z,t}^{-2})$ and Lemma \ref{lemma:qlowerbound} to get
\begin{align*}
&\exp\left\{-\mathbf{C}_{2}\begin{pmatrix}\frac{N}{t}-\tr \tilde{H}_{a_{j},b_{j},\theta_{j},11}^{(j-1)}\\-2\tr \tilde{H}_{a_{j},b_{j},\theta_{j},12}^{(j-1)}\\\frac{N}{t}-\tr \tilde{H}_{a_{j},b_{j},\theta_{j},22}^{(j-1)}\end{pmatrix}^{\ast}[\mathbf{Q}_{a_{j},b_{j},\theta_{j}}^{(j-1)}]^{-1}\begin{pmatrix}\frac{N}{t}-\tr \tilde{H}_{a_{j},b_{j},\theta_{j},11}^{(j-1)}\\-2\tr \tilde{H}_{a_{j},b_{j},\theta_{j},12}^{(j-1)}\\\frac{N}{t}-\frac12\tr \tilde{H}_{a_{j},b_{j},\theta_{j},22}^{(j-1)}\end{pmatrix}\right\}\\
&=\exp\left\{-\mathbf{C}_{2}\begin{pmatrix}\frac{N}{t}-\tr \tilde{H}_{a_{j},b_{j},\theta_{j},11}\\-2\tr \tilde{H}_{a_{j},b_{j},\theta_{j},12}\\\frac{N}{t}-\tr \tilde{H}_{a_{j},b_{j},\theta_{j},22}\end{pmatrix}^{\ast}[\mathbf{Q}_{a_{j},b_{j},\theta_{j}}^{(j-1)}]^{-1}\begin{pmatrix}\frac{N}{t}-\tr \tilde{H}_{a_{j},b_{j},\theta_{j},11}\\-2\tr \tilde{H}_{a_{j},b_{j},\theta_{j},12}\\\frac{N}{t}-\tr \tilde{H}_{a_{j},b_{j},\theta_{j},22}\end{pmatrix}\right\}\\
&\times\exp\left\{-\mathbf{C}_{2}\begin{pmatrix}\frac{N}{t}-\tr \tilde{H}_{a_{j},b_{j},\theta_{j},11}^{(j-1)}\\-2\tr \tilde{H}_{a_{j},b_{j},\theta_{j},12}^{(j-1)}\\\frac{N}{t}-\tr \tilde{H}_{a_{j},b_{j},\theta_{j},22}^{(j-1)}\end{pmatrix}^{\ast}[\mathbf{Q}_{a_{j},b_{j},\theta_{j}}^{(j-1)}]^{-1}\left[\begin{pmatrix}\frac{N}{t}-\tr \tilde{H}_{a_{j},b_{j},\theta_{j},11}^{(j-1)}\\-2\tr \tilde{H}_{a_{j},b_{j},\theta_{j},12}^{(j-1)}\\\frac{N}{t}-\frac12\tr \tilde{H}_{a_{j},b_{j},\theta_{j},22}^{(j-1)}\end{pmatrix}-\begin{pmatrix}\frac{N}{t}-\tr \tilde{H}_{a_{j},b_{j},\theta_{j},11}\\-2\tr \tilde{H}_{a_{j},b_{j},\theta_{j},12}\\\frac{N}{t}-\tr \tilde{H}_{a_{j},b_{j},\theta_{j},22}\end{pmatrix}\right]\right\}\\
&\times\exp\left\{-\mathbf{C}_{2}\begin{pmatrix}\frac{N}{t}-\tr \tilde{H}_{a_{j},b_{j},\theta_{j},11}\\-2\tr \tilde{H}_{a_{j},b_{j},\theta_{j},12}\\\frac{N}{t}-\tr \tilde{H}_{a_{j},b_{j},\theta_{j},22}\end{pmatrix}^{\ast}[\mathbf{Q}_{a_{j},b_{j},\theta_{j}}^{(j-1)}]^{-1}\left[\begin{pmatrix}\frac{N}{t}-\tr \tilde{H}_{a_{j},b_{j},\theta_{j},11}^{(j-1)}\\-2\tr \tilde{H}_{a_{j},b_{j},\theta_{j},12}^{(j-1)}\\\frac{N}{t}-\frac12\tr \tilde{H}_{a_{j},b_{j},\theta_{j},22}^{(j-1)}\end{pmatrix}-\begin{pmatrix}\frac{N}{t}-\tr \tilde{H}_{a_{j},b_{j},\theta_{j},11}\\-2\tr \tilde{H}_{a_{j},b_{j},\theta_{j},12}\\\frac{N}{t}-\tr \tilde{H}_{a_{j},b_{j},\theta_{j},22}\end{pmatrix}\right]\right\}\\
&=\exp\left\{-\mathbf{C}_{2}\begin{pmatrix}\frac{N}{t}-\tr \tilde{H}_{a_{j},b_{j},\theta_{j},11}\\-2\tr \tilde{H}_{a_{j},b_{j},\theta_{j},12}\\\frac{N}{t}-\tr \tilde{H}_{a_{j},b_{j},\theta_{j},22}\end{pmatrix}^{\ast}[\mathbf{Q}_{a_{j},b_{j},\theta_{j}}^{(j-1)}]^{-1}\begin{pmatrix}\frac{N}{t}-\tr \tilde{H}_{a_{j},b_{j},\theta_{j},11}\\-2\tr \tilde{H}_{a_{j},b_{j},\theta_{j},12}\\\frac{N}{t}-\tr \tilde{H}_{a_{j},b_{j},\theta_{j},22}\end{pmatrix}\right\}[1+O(N^{-\kappa})].
\end{align*}
for some $\kappa>0$. Let us now move to removing $(j-1)$ from the $\mathbf{Q}$ matrix. To this end, use Cauchy interlacing as in \eqref{eq:interlacingdemo} and $\|\tilde{H}_{a_{j},b_{j},\theta_{j}}^{(j-1)}\|_{\mathrm{op}}\lesssim\eta_{z,t}^{-2}$ to get $\|\mathbf{Q}_{a_{j},b_{j},\theta_{j}}^{(j-1)}-\mathbf{Q}_{a_{j},b_{j},\theta_{j}}\|_{\mathrm{op}}\lesssim \eta_{z,t}^{-D}$ for some $D>0$. Resolvent perturbation and Lemma \ref{lemma:qlowerbound} then give $\|[\mathbf{Q}_{a_{j},b_{j},\theta_{j}}^{(j-1)}]^{-1}-\mathbf{Q}_{a_{j},b_{j},\theta_{j}}^{-1}\|_{\mathrm{op}}\lesssim \eta_{z,t}^{-D}N^{-2}\lesssim N^{-2+D\epsilon_{0}}$. Therefore, we have 
\begin{align*}
&\exp\left\{-\mathbf{C}_{2}\begin{pmatrix}\frac{N}{t}-\tr \tilde{H}_{a_{j},b_{j},\theta_{j},11}\\-2\tr \tilde{H}_{a_{j},b_{j},\theta_{j},12}\\\frac{N}{t}-\tr \tilde{H}_{a_{j},b_{j},\theta_{j},22}\end{pmatrix}^{\ast}[\mathbf{Q}_{a_{j},b_{j},\theta_{j}}^{(j-1)}]^{-1}\begin{pmatrix}\frac{N}{t}-\tr \tilde{H}_{a_{j},b_{j},\theta_{j},11}\\-2\tr \tilde{H}_{a_{j},b_{j},\theta_{j},12}\\\frac{N}{t}-\tr \tilde{H}_{a_{j},b_{j},\theta_{j},22}\end{pmatrix}\right\}\\
&=\exp\left\{-\mathbf{C}_{2}\begin{pmatrix}\frac{N}{t}-\tr \tilde{H}_{a_{j},b_{j},\theta_{j},11}\\-2\tr \tilde{H}_{a_{j},b_{j},\theta_{j},12}\\\frac{N}{t}-\tr \tilde{H}_{a_{j},b_{j},\theta_{j},22}\end{pmatrix}^{\ast}[\mathbf{Q}_{a_{j},b_{j},\theta_{j}}]^{-1}\begin{pmatrix}\frac{N}{t}-\tr \tilde{H}_{a_{j},b_{j},\theta_{j},11}\\-2\tr \tilde{H}_{a_{j},b_{j},\theta_{j},12}\\\frac{N}{t}-\tr \tilde{H}_{a_{j},b_{j},\theta_{j},22}\end{pmatrix}\right\}\left[1+O(N^{-\kappa})\right].
\end{align*}
A similar argument shows that $|\det\mathbf{Q}_{a_{j},b_{j},\theta_{j}}^{(j-1)}|^{-\frac12}=|\det\mathbf{Q}_{a_{j},b_{j},\theta_{j}}|^{-\frac12}[1+O(N^{-\kappa})]$. Ultimately, 
\begin{align*}
&\int_{M_{2}^{sa}(\R)}e^{i\frac{N}{2t}\tr P}\det[I+i[\tilde{H}_{a_{j},b_{j},\theta_{j}}^{(j-1)}]^{\frac12}(P\otimes I)[\tilde{H}_{a_{j},b_{j},\theta_{j}}^{(j-1)}]^{\frac12}]^{-\frac12}dP\\
&=O(\exp[-CN^{2\upsilon}])+\mathbf{C}_{1}|\det\mathbf{Q}_{a_{j},b_{j},\theta_{j}}|^{-\frac12}\\
&\times\exp\left\{-\mathbf{C}_{2}\begin{pmatrix}\frac{N}{t}-\tr \tilde{H}_{a_{j},b_{j},\theta_{j},11}\\-2\tr \tilde{H}_{a_{j},b_{j},\theta_{j},12}\\\frac{N}{t}-\tr \tilde{H}_{a_{j},b_{j},\theta_{j},22}\end{pmatrix}^{\ast}[\mathbf{Q}_{a_{j},b_{j},\theta_{j}}]^{-1}\begin{pmatrix}\frac{N}{t}-\tr \tilde{H}_{a_{j},b_{j},\theta_{j},11}\\-2\tr \tilde{H}_{a_{j},b_{j},\theta_{j},12}\\\frac{N}{t}-\tr \tilde{H}_{a_{j},b_{j},\theta_{j},22}\end{pmatrix}\right\}\left[1+O(N^{-\kappa})\right].
\end{align*}
%
\subsection{Putting it altogether}
By our computations of the ratio of determinants and the $dP$ integration, we have the following estimate locally uniformly in $a_{j},b_{j}$:
\begin{align*}
&\rho_{\mathrm{main}}(z,\mathbf{z};A)=\frac{16N^{7}b^{4}[1+O(N^{-\kappa})]}{\pi^{4}t^{12}\gamma_{z,t}^{2}\sigma_{z,t}^{3}\upsilon_{z,t}^{2}}\rho_{\mathrm{GinUE}}^{(2)}(z_{1},z_{2})\int_{\mathcal{I}_{0}^{2}}d\theta_{1}d\theta_{2}\frac{|\cos2\theta_{1}||\cos2\theta_{2}|}{\sin^{2}\theta_{1}\sin^{2}\theta_{2}}\\
&\times\prod_{j=1,2}\int_{V^{2}(\R^{N-2j+2})}d\mu_{j}(\mathbf{v}_{j})|\det[V_{j}^{\ast}G_{z}^{(j-1)}(\eta_{z,t})V_{j}]|^{j}\\
&\times\exp\left\{-b_{j}^{2}(\tan\theta_{j}-\tan^{-1}\theta_{j})^{2}\eta_{z,t}^{2}\tr H_{\lambda_{j}}(\eta_{z,t})\tilde{H}_{\overline{\lambda}_{j}}(\eta_{z,t})\right\}\\
&\times\mathbf{C}_{1}|\det\mathbf{Q}_{a_{j},b_{j},\theta_{j}}|^{-\frac12}\exp\left\{-\mathbf{C}_{2}\begin{pmatrix}\frac{N}{t}-\tr \tilde{H}_{a_{j},b_{j},\theta_{j},11}\\-2\tr \tilde{H}_{a_{j},b_{j},\theta_{j},12}\\\frac{N}{t}-\tr \tilde{H}_{a_{j},b_{j},\theta_{j},22}\end{pmatrix}^{\ast}\mathbf{Q}_{a_{j},b_{j},\theta_{j}}^{-1}\begin{pmatrix}\frac{N}{t}-\tr \tilde{H}_{a_{j},b_{j},\theta_{j},11}\\-2\tr \tilde{H}_{a_{j},b_{j},\theta_{j},12}\\\frac{N}{t}-\tr \tilde{H}_{a_{j},b_{j},\theta_{j},22}\end{pmatrix}\right\}\\
&+O(\exp[-CN^{\kappa}]).
\end{align*}
By \eqref{eq:prelimcomputerho}, Proposition \ref{prop:cutoff}, and our computation of the remaining $4\times4$ determinants (see Lemmas \ref{lemma:concentrationj} and \ref{lemma:concentrationjzjtoz}), we deduce
\begin{align}
&\rho_{t}(z,\mathbf{z};A)=\frac{16N^{7}b^{4}[1+O(N^{-\kappa})]}{\pi^{4}t^{12}\gamma_{z,t}^{2}\sigma_{z,t}^{3}\upsilon_{z,t}^{2}}\rho_{\mathrm{GinUE}}^{(2)}(z_{1},z_{2})\int_{\mathcal{I}_{0}^{2}}d\theta_{1}d\theta_{2}\frac{|\cos2\theta_{1}||\cos2\theta_{2}|}{\sin^{2}\theta_{1}\sin^{2}\theta_{2}}\label{eq:rhomain}\\
&\times\prod_{j=1,2}|\det\mathcal{G}_{j,a,b}|^{j}\times\exp\left\{-b_{j}^{2}(\tan\theta_{j}-\tan^{-1}\theta_{j})^{2}\eta_{z,t}^{2}\tr H_{\lambda_{j}}(\eta_{z,t})\tilde{H}_{\overline{\lambda}_{j}}(\eta_{z,t})\right\}\nonumber\\
&\times\mathbf{C}_{1}|\det\mathbf{Q}_{a_{j},b_{j},\theta_{j}}|^{-\frac12}\exp\left\{-\mathbf{C}_{2}\begin{pmatrix}\frac{N}{t}-\tr \tilde{H}_{a_{j},b_{j},\theta_{j},11}\\-2\tr \tilde{H}_{a_{j},b_{j},\theta_{j},12}\\\frac{N}{t}-\tr \tilde{H}_{a_{j},b_{j},\theta_{j},22}\end{pmatrix}^{\ast}\mathbf{Q}_{a_{j},b_{j},\theta_{j}}^{-1}\begin{pmatrix}\frac{N}{t}-\tr \tilde{H}_{a_{j},b_{j},\theta_{j},11}\\-2\tr \tilde{H}_{a_{j},b_{j},\theta_{j},12}\\\frac{N}{t}-\tr \tilde{H}_{a_{j},b_{j},\theta_{j},22}\end{pmatrix}\right\}\nonumber\\
&+O(\exp[-CN^{\kappa}]),\nonumber
\end{align}
where, to be totally explicit, $\mathcal{G}_{j,a,b}$ is the $4\times4$ matrix
\begin{align}
\mathcal{G}_{j,a,b}&:=\begin{pmatrix}\mathcal{G}_{j,a,b,11}&\mathcal{G}_{j,a,b,12}\\\mathcal{G}_{j,a,b,21}&\mathcal{G}_{j,a,b,22}\end{pmatrix}\label{eq:gjabmatrix}\\
\mathcal{G}_{j,a,b,11}&:=\begin{pmatrix}i\eta_{z,t}\frac{t}{N}\tr \tilde{H}_{a,b,\frac{\pi}{4}}( H_{z}\otimes E_{11})&i\eta_{z,t}\frac{t}{N}\tr \tilde{H}_{a,b,\frac{\pi}{4}}( H_{z}\otimes E_{12})\\i\eta_{z,t}\frac{t}{N}\tr \tilde{H}_{a,b,\frac{\pi}{4}}( H_{z}\otimes E_{21})&i\eta_{z,t}\frac{t}{N}\tr \tilde{H}_{a,b,\frac{\pi}{4}}( H_{z}\otimes E_{22})\end{pmatrix}\nonumber\\
\mathcal{G}_{j,a,b,12}&:=\begin{pmatrix}\frac{t}{N}\tr \tilde{H}_{a,b,\frac{\pi}{4}}[ H_{z}(A-z)\otimes E_{11}]&\frac{t}{N}\tr \tilde{H}_{a,b,\frac{\pi}{4}}[ H_{z}(A-z)\otimes E_{12}]\\\frac{t}{N}\tr \tilde{H}_{a,b,\frac{\pi}{4}}[ H_{z}(A-z)\otimes E_{21}]&\frac{t}{N}\tr \tilde{H}_{a,b,\frac{\pi}{4}}[ H_{z}(A-z)\otimes E_{22}]\end{pmatrix}\nonumber\\
\mathcal{G}_{j,a,b,21}&:=\begin{pmatrix}\frac{t}{N}\tr \tilde{H}_{a,b,\frac{\pi}{4}}[(A-z)^{\ast} H_{z}\otimes E_{11}]&\frac{t}{N}\tr \tilde{H}_{a,b,\frac{\pi}{4}}[(A-z)^{\ast} H_{z}\otimes E_{12}]\\\frac{t}{N}\tr \tilde{H}_{a,b,\frac{\pi}{4}}[(A-z)^{\ast} H_{z}\otimes E_{21}]&\frac{t}{N}\tr \tilde{H}_{a,b,\frac{\pi}{4}}[(A-z)^{\ast} H_{z}\otimes E_{22}]\end{pmatrix}\nonumber\\
\mathcal{G}_{j,a,b,22}&:=\begin{pmatrix}i\eta_{z,t}\frac{t}{N}\tr \tilde{H}_{a,b,\frac{\pi}{4}}(\tilde{H}_{z}\otimes E_{11})&i\eta_{z,t}\frac{t}{N}\tr \tilde{H}_{a,b,\frac{\pi}{4}}(\tilde{H}_{z}\otimes E_{12})\\i\eta_{z,t}\frac{t}{N}\tr \tilde{H}_{a,b,\frac{\pi}{4}}(\tilde{H}_{z}\otimes E_{21})&i\eta_{z,t}\frac{t}{N}\tr \tilde{H}_{a,b,\frac{\pi}{4}}(\tilde{H}_{z}\otimes E_{22})\end{pmatrix},\nonumber\\
E_{11}&:=\begin{pmatrix}1&0\\0&0\end{pmatrix},\quad E_{22}:=\begin{pmatrix}0&0\\0&1\end{pmatrix}\quad
E_{12}:=\begin{pmatrix}0&1\\0&0\end{pmatrix}\quad
E_{21}:=\begin{pmatrix}0&0\\1&0\end{pmatrix},\nonumber
\end{align}
and $\mathbf{Q}_{a_{j},b_{j},\theta_{j}}$ is given by
\begin{align*}
\mathbf{Q}_{a_{j},b_{j},\theta_{j}}&:=\begin{pmatrix}[\mathbf{Q}_{a_{j},b_{j},\theta_{j}}]_{1}&[\mathbf{Q}_{a_{j},b_{j},\theta_{j}}]_{2}&[\mathbf{Q}_{a_{j},b_{j},\theta_{j}}]_{3}\end{pmatrix},\\
[\mathbf{Q}_{a_{j},b_{j},\theta_{j}}]_{1}&=\begin{pmatrix}\tr \tilde{H}_{a_{j},b_{j},\theta_{j},11}^{2}\\2\tr \tilde{H}_{a_{j},b_{j},\theta_{j},11} \tilde{H}_{a_{j},b_{j},\theta_{j},12}\\\tr \tilde{H}_{a_{j},b_{j},\theta_{j},12}\tilde{H}_{a_{j},b_{j},\theta_{j},12}^T\end{pmatrix}\\
[\mathbf{Q}_{a_{j},b_{j},\theta_{j}}]_{2}&=\begin{pmatrix}2\tr \tilde{H}_{a_{j},b_{j},\theta_{j},11} \tilde{H}_{a_{j},b_{j},\theta_{j},12}\\2\tr \tilde{H}_{a_{j},b_{j},\theta_{j},11} \tilde{H}_{a_{j},b_{j},\theta_{j},22}+2\tr \tilde{H}_{a_{j},b_{j},\theta_{j},12}^{2}\\2\tr \tilde{H}_{a_{j},b_{j},\theta_{j},22} \tilde{H}_{a_{j},b_{j},\theta_{j},12}\end{pmatrix}\\
[\mathbf{Q}_{a_{j},b_{j},\theta_{j}}]_{3}&=\begin{pmatrix}\tr \tilde{H}_{a_{j},b_{j},\theta_{j},12}\tilde{H}_{a_{j},b_{j},\theta_{j},12}^{T}\\2\tr \tilde{H}_{a_{j},b_{j},\theta_{j},22} \tilde{H}_{a_{j},b_{j},\theta_{j},12}\\\tr \tilde{H}_{a_{j},b_{j},\theta_{j},22}^{2}\end{pmatrix}
\end{align*}
%
%
%
%
\section{Replacing resolvents in \eqref{eq:rhomain} by universal local law approximations}
The RHS of \eqref{eq:rhomain} is $\rho_{\mathrm{GinUE}}^{(2)}(z_{1},z_{2})$ times a quantity that depends only on traces of resolvents. Thus, we can replace each such trace by its local approximation to obtain a quantity of the form $\Phi_{z,t}(z_{1},z_{2})$ as in Theorem \ref{theorem:mainwitht}.
\subsection{Replacing the constants, e.g. \texorpdfstring{$\eta_{z,t}$}{eta\_zt}}
We start with the following result, which approximates constants appearing in the introduction with universal quantities.
\begin{lemma}\label{lemma:replaceconstants}
There exists a constant $\kappa>0$, which is independent of $\epsilon_{0}>0$, as well as constants $\eta_{\mathrm{univ},t},\alpha_{\mathrm{univ},z,t},\beta_{\mathrm{univ},z,t},\gamma_{\mathrm{univ},z,t},\sigma_{\mathrm{univ},z,t}$, which do not depend on the distribution of the entries of $A$, such that with high probability. we have
\begin{align*}
\eta_{z,t}&=\eta_{\mathrm{univ},z,t}[1+O(N^{-\kappa})],\\
\alpha_{z,t}&=\alpha_{\mathrm{univ},z,t}[1+O(N^{-\kappa})],\\
\beta_{z,t}&=\beta_{\mathrm{univ},z,t}[1+O(N^{-\kappa})],\\
\gamma_{z,t}&=\gamma_{\mathrm{univ},z,t}[1+O(N^{-\kappa})],\\
\sigma_{z,t}&=\sigma_{\mathrm{univ},z,t}[1+O(N^{-\kappa})],\\
\upsilon_{z,t}&=\upsilon_{\mathrm{univ},z,t}[1+O(N^{-\kappa})].
\end{align*}
\end{lemma}
\begin{proof}
We will give details for the $\eta_{z,t}$ bound. By definition, $t\langle H_{z}(\eta_{z,t})\rangle=-it\eta_{z,t}^{-1}\langle G_{z}(\eta_{z,t})E_{11}\rangle=1$. Let $\eta_{\mathrm{univ},z,t}$ be such that $-it\eta_{\mathrm{univ},z,t}^{-1}\langle M_{a,b,\frac{\pi}{4}}(\eta_{\mathrm{univ},z,t})E_{11}\rangle=1$, where $M_{a,b,\frac{\pi}{4}}(\eta)$ is the local law approximation to $G_{z}(\eta)$ from Lemma \ref{1GE}. Now, suppose $\eta_{\mathrm{univ},z,t}\asymp t$ (i.e. it is bounded above and below by a constant times $t$). We will show this shortly. In this case, by Lemma \ref{1GE}, we have the estimate $\langle G_{z}(\eta_{\mathrm{univ},z,t})E_{11}-M_{a,b,\frac{\pi}{4}}(\eta_{\mathrm{univ},z,t})E_{11}\rangle\prec N^{-1}\eta_{\mathrm{univ},z,t}^{-2}$, where $\prec$ means bounded above up to a factor of $N^{\tau}$ for any fixed $\tau>0$ with high probability; see Definition \ref{defn:prec}. Thus, we have 
\begin{align*}
-it\eta_{\mathrm{univ},z,t}^{-1}\langle G_{z}(\eta_{\mathrm{univ},z,t})E_{11}\rangle=1+O_{\prec}(N^{-1}\eta_{\mathrm{univ},z,t}^{-2}),
\end{align*}
where, similarly, $O_{\prec}(N^{-1}\eta_{\mathrm{univ},z,t]}^{-2})$ means something whose absolute value is $\prec N^{-1}\eta_{\mathrm{univ},z,t}^{-2}$.

The LHS of the above display is $t\langle H_{z}(\eta_{\mathrm{univ},z,t})\rangle$, so $t(\langle H_{z}(\eta_{\mathrm{univ},z,t})\rangle-\langle H_{z}(\eta_{z,t})\rangle)\prec N^{-1}\eta_{\mathrm{univ},z,t}^{-2}$. Standard resolvent perturbation and the operator bound $\|H_{z}(\eta_{\mathrm{univ},z,t})\|_{\mathrm{op}}\gtrsim1$ (which holds with high probability since $\|A\|_{\mathrm{op}}\lesssim1$ with high probability) imply that
\begin{align*}
|t\langle H_{z}(\eta_{\mathrm{univ},z,t})\rangle-t\langle H_{z}(\eta_{z,t})\rangle|&=\left|\frac{t(\eta_{z,t}^{2}-\eta_{\mathrm{univ},z,t}^{2})}{N}\tr H_{z}(\eta_{\mathrm{univ},z,t})H_{z}(\eta_{z,t})\right|\gtrsim t^{2}|\eta_{z,t}-\eta_{\mathrm{univ},z,t}|.
\end{align*}
We deduce that $|\eta_{z,t}-\eta_{\mathrm{univ},z,t}|\lesssim N^{-1}t^{-4}\lesssim N^{-\kappa}t$ for some $\kappa>0$, since $t=N^{-\epsilon_{0}}$ with $\epsilon_{0}>0$ small. This proves the desired $\eta_{z,t}$ estimate, provided we can show that $\eta_{\mathrm{univ},z,t}\asymp t$. This follows by $\eta_{z,t}\asymp t$ and $|\eta_{z,t}-\eta_{\mathrm{univ},z,t}|\lesssim N^{-1}t^{-4}\lesssim N^{-\kappa}t$. (To avoid a circular argument, one can use a standard continuity argument. In particular, in this proof so far, replace $M_{a,b,\frac{\pi}{4}}$ by $sM_{a,b,\frac{\pi}{4}}+(1-s)G_{z}$, and use continuity in $s\in[0,1]$ to show that $\eta_{\mathrm{univ},z,t}\asymp t$.) The other estimates follow similarly, though $\alpha_{\mathrm{univ},z,t}$, for example, is defined instead using the two-term local law in Lemma \ref{2GE}.
\end{proof}
\subsection{Replacing \texorpdfstring{$|\det\mathcal{G}_{j,a,b}|$}{|det G\_jab|}}
The matrix $\mathcal{G}_{j,a,b}$ has size $4\times4$, and its entries are given by normalized traces of products of two resolvents. Thus, Lemma \ref{2GE} will allow us to replace its determinant by that of a matrix which does not depend on the distribution of the entries of $A$.
\begin{lemma}\label{lemma:replacedetG}
There exists a $4\times4$ matrix $\mathcal{G}_{\mathrm{univ}}$ such that $|\det\mathcal{G}_{j,a,b}|=|\det\mathcal{G}_{\mathrm{univ}}|[1+O(N^{-\kappa})]$, where $\kappa>0$ and $\mathcal{G}_{\mathrm{univ}}$ has entries independent of the distribution of the entries of $A$.
\end{lemma}
\begin{proof}
It suffices to show that for any matrix entry indices $k,\ell$, we have $(\mathcal{G}_{j,a,b})_{k\ell}=(\mathcal{G}_{\mathrm{univ}})_{k\ell}[1+O(N^{-\kappa})]$, where $(\mathcal{G}_{\mathrm{univ}})_{k\ell}$ denotes a quantity which does not depend on distribution of the entries of $A$. Indeed, the determinant is a polynomial in the entries, and the entries of $\mathcal{G}_{j,a,b}$ are easily checked to be $\lesssim\eta_{z,t}^{-4}\lesssim N^{4\epsilon_{0}}$, and we can choose $\epsilon_{0}>0$ sufficiently small. We will choose $k,\ell=1$; the other choices of $k,\ell$ follow by similar arguments, since each matrix that is tensored with $E_{11},E_{12},E_{21},E_{22}$ in $\mathcal{G}_{j,a,b}$ is a block in $G_{z}(\eta_{z,t})$. In particular, we want to show that 
\begin{align*}
i\eta_{z,t}\frac{t}{N}\tr \tilde{H}_{a,b,\frac{\pi}{4}}(\eta_{z,t})(H_{z}(\eta_{z,t})\otimes E_{11})
\end{align*}
is equal to a universal quantity times an error $1+O(N^{-\kappa})$. First, we replace $\eta_{z,t}$ by $\eta_{\mathrm{univ},z,t}$ from Lemma \ref{lemma:replaceconstants}. Since $\tilde{H}_{a,b,\frac{\pi}{4}}$ and $H_{z}$ have operator norms $\lesssim\eta_{z,t}^{-2}$, we have 
\begin{align*}
&i\eta_{z,t}\frac{t}{N}\tr \tilde{H}_{a,b,\frac{\pi}{4}}(\eta_{z,t})(H_{z}(\eta_{z,t})\otimes E_{11})\\
&=i\eta_{\mathrm{univ},z,t}\frac{t}{N}\tr \tilde{H}_{a,b,\frac{\pi}{4}}(\eta_{\mathrm{univ},z,t})(H_{z}(\eta_{\mathrm{univ},z,t})\otimes E_{11})+O(N^{-\kappa}\eta_{z,t}^{-D})
\end{align*}
for some $D>0$. Since $\tilde{H}_{a,b,\frac{\pi}{4}}\geq C$ and $H_{z}\geq C$ with high probability (it is the inverse of a covariance matrix of a shift of $A$), and because $\eta_{\mathrm{univ},z,t}t\asymp N^{-2\epsilon_{0}}$ (see the proof of Lemma \ref{lemma:replaceconstants}), the RHS of the previous identity is equal to
\begin{align*}
i\eta_{\mathrm{univ},z,t}\frac{t}{N}\tr \tilde{H}_{a,b,\frac{\pi}{4}}(\eta_{\mathrm{univ},z,t})(H_{z}(\eta_{\mathrm{univ},z,t})\otimes E_{11})[1++O(N^{-\kappa})]
\end{align*}
for possibly different $\kappa>0$. Now, apply Lemma \ref{2GE} for $\theta=\frac{\pi}{4}$; this shows that the first term on the RHS of the previous display is universal (i.e. independent of the distribution of the entries of $A$) plus an error of $O(N^{-\kappa})$. We now absorb this additive error of $O(N^{-\kappa})$ as a multiplicative error of $1+O(N^{-\kappa})$ (for possibly different $\kappa>0$) using the same argument that gave the previous display. This completes the proof.
\end{proof}
\subsection{Replacing \texorpdfstring{$\tr H_{\lambda_{j}}(\eta_{z,t})\tilde{H}_{\overline{\lambda}_{j}}(\eta_{z,t})$}{tr H\_lambdaj tildeH\_lambdaj}}
We move to the exponential factor in the second line in \eqref{eq:rhomain}. The only term we must deal with here is $\eta_{z,t}^{2}$ times the trace of two resolvents. For this, we use Lemmas \ref{lemma:replaceconstants} and \ref{2GE}.
\begin{lemma}\label{lemma:replacetrHH}
There exists a constant $\mathfrak{T}=\mathfrak{T}_{\lambda_{j}}$ that is independent of the distribution of the entries of $A$ and such that the following holds for $\kappa>0$ and for all $\theta_{j}\in\mathcal{I}_{0}$:
\begin{align*}
&\exp\left\{-b_{j}^{2}(\tan\theta_{j}-\tan^{-1}\theta_{j})^{2}\eta_{z,t}^{2}\tr H_{\lambda_{j}}(\eta_{z,t})\tilde{H}_{\overline{\lambda}_{j}}(\eta_{z,t})\right\}\\
&=\exp\left\{-b_{j}^{2}(\tan\theta_{j}-\tan^{-1}\theta_{j})^{2}\eta_{\mathrm{univ},z,t}^{2}\mathfrak{T}\right\}\left[1+O(N^{-\kappa})\right].
\end{align*}
\end{lemma}
\begin{proof}
We first have the decomposition
\begin{align}
&b_{j}^{2}(\tan\theta_{j}-\tan^{-1}\theta_{j})^{2}\eta_{z,t}^{2}\tr H_{\lambda_{j}}(\eta_{z,t})\tilde{H}_{\overline{\lambda}_{j}}(\eta_{z,t})\label{eq:replacetrHH1}\\
&=b_{j}^{2}(\tan\theta_{j}-\tan^{-1}\theta_{j})^{2}\eta_{\mathrm{univ},z,t}^{2}\tr H_{\lambda_{j}}(\eta_{\mathrm{univ},z,t})\tilde{H}_{\overline{\lambda}_{j}}(\eta_{\mathrm{univ},z,t})\nonumber\\
&+b_{j}^{2}(\tan\theta_{j}-\tan^{-1}\theta_{j})^{2}[\eta_{z,t}^{2}-\eta_{\mathrm{univ},z,t}^{2}]\tr H_{\lambda_{j}}(\eta_{z,t})\tilde{H}_{\overline{\lambda}_{j}}(\eta_{z,t})\nonumber\\
&+b_{j}^{2}(\tan\theta_{j}-\tan^{-1}\theta_{j})^{2}\eta_{\mathrm{univ},z,t}^{2}\left[\tr H_{\lambda_{j}}(\eta_{z,t})\tilde{H}_{\overline{\lambda}_{j}}(\eta_{z,t})-\tr H_{\lambda_{j}}(\eta_{\mathrm{univ},z,t})\tilde{H}_{\overline{\lambda}_{j}}(\eta_{\mathrm{univ},z,t})\right].\nonumber
\end{align}
We bound the last two lines in \eqref{eq:replacetrHH1}. We start with the third line. By Lemma \ref{lemma:replaceconstants}, we know $\eta_{z,t}^{2}-\eta_{\mathrm{univ},z,t}^{2}\lesssim N^{-\kappa}$ for some $\kappa>0$ with high probability. Moreover, $\|H_{\lambda_{j}}(\eta_{z,t})\|_{\mathrm{op}},\|\tilde{H}_{\overline{\lambda}_{j}}(\eta_{z,t})\|_{\mathrm{op}}\lesssim\eta_{z,t}^{-2}$. Finally, because $\theta_{j}\in\mathcal{I}_{0}$, we know $|\tan\theta_{j}-\tan^{-1}\theta_{j}|^{2}\lesssim|\theta-\frac{\pi}{4}|^{2}\lesssim N^{-1+2\tau}\eta_{z,t}^{-2}$ for small $\tau>0$ depending on at most $\epsilon_{0}$. Thus, the third line satisfies
\begin{align*}
b_{j}^{2}(\tan\theta_{j}-\tan^{-1}\theta_{j})^{2}[\eta_{z,t}^{2}-\eta_{\mathrm{univ},z,t}^{2}]\tr H_{\lambda_{j}}(\eta_{z,t})\tilde{H}_{\overline{\lambda}_{j}}(\eta_{z,t})&\lesssim N^{-\kappa}N^{-1+2\tau}\eta_{z,t}^{-2}N\eta_{z,t}^{-4}\\
&\lesssim N^{-\kappa+2\tau}\eta_{z,t}^{-6},
\end{align*}
which is $\lesssim N^{-\kappa/2}$ if we choose $\tau,\epsilon_{0}$ small enough. For the last line in \eqref{eq:replacetrHH1}, we have 
\begin{align*}
&\left\|H_{\lambda_{j}}(\eta_{z,t})\tilde{H}_{\overline{\lambda}_{j}}(\eta_{z,t})-H_{\lambda_{j}}(\eta_{\mathrm{univ},z,t})\tilde{H}_{\overline{\lambda}_{j}}(\eta_{\mathrm{univ},z,t})\right\|_{\mathrm{op}}\\
&\lesssim \eta_{z,t}^{-2}\|H_{\lambda_{j}}(\eta_{z,t})-H_{\lambda_{j}}(\eta_{\mathrm{univ},z,t})\|_{\mathrm{op}}+\eta_{z,t}^{-2}\|\tilde{H}_{\overline{\lambda}_{j}}(\eta_{z,t})-\tilde{H}_{\overline{\lambda}_{j}}(\eta_{\mathrm{univ},z,t})\|_{\mathrm{op}},
\end{align*}
and by resolvent perturbation, the last line is $\lesssim\eta_{z,t}^{-D}|\eta_{z,t}-\eta_{\mathrm{univ},z,t}|\lesssim N^{-\kappa}\eta_{z,t}^{-D}$ for some $\kappa>0$ and $D=O(1)$. (The last bound follows by Lemma \ref{lemma:replaceconstants}.) Using this and the bound $|\tan\theta_{j}-\tan^{-1}\theta_{j}|^{2}\lesssim|\theta-\frac{\pi}{4}|^{2}\lesssim N^{-1+2\tau}\eta_{z,t}^{-2}$ for small $\tau>0$ and $\theta_{j}\in\mathcal{I}_{0}$, we deduce
\begin{align*}
&b_{j}^{2}(\tan\theta_{j}-\tan^{-1}\theta_{j})^{2}\eta_{\mathrm{univ},z,t}^{2}\left[\tr H_{\lambda_{j}}(\eta_{z,t})\tilde{H}_{\overline{\lambda}_{j}}(\eta_{z,t})-\tr H_{\lambda_{j}}(\eta_{\mathrm{univ},z,t})\tilde{H}_{\overline{\lambda}_{j}}(\eta_{\mathrm{univ},z,t})\right]\\
&\lesssim N^{-1+2\tau}NN^{-\kappa}\eta_{z,t}^{-D}\lesssim N^{-\kappa+2\tau}\eta_{z,t}^{-D},
\end{align*}
which is $\lesssim N^{-\kappa/2}$ if $\tau,\epsilon_{0}>0$ are small enough. Thus, we have
\begin{align}
&b_{j}^{2}(\tan\theta_{j}-\tan^{-1}\theta_{j})^{2}\eta_{z,t}^{2}\tr H_{\lambda_{j}}(\eta_{z,t})\tilde{H}_{\overline{\lambda}_{j}}(\eta_{z,t})\label{eq:replacetrHH2}\\
&=b_{j}^{2}(\tan\theta_{j}-\tan^{-1}\theta_{j})^{2}\eta_{\mathrm{univ},z,t}^{2}\tr H_{\lambda_{j}}(\eta_{\mathrm{univ},z,t})\tilde{H}_{\overline{\lambda}_{j}}(\eta_{\mathrm{univ},z,t})+O(N^{-\kappa}),\nonumber
\end{align}
where $\kappa>0$ is possibly different than earlier in this proof. Now, we can apply Lemma \ref{2GE} to the trace in the second line in \eqref{eq:replacetrHH2}. This shows that for some $\mathfrak{T}$ as in the statement of this lemma, we have the following, where the last line uses $|\tan\theta_{j}-\tan^{-1}\theta_{j}|^{2}\lesssim|\theta-\frac{\pi}{4}|^{2}\lesssim N^{-1+2\tau}\eta_{z,t}^{-2}$ and holds if $\epsilon_{0}>0$ is small enough:
\begin{align*}
&b_{j}^{2}(\tan\theta_{j}-\tan^{-1}\theta_{j})^{2}\eta_{\mathrm{univ},z,t}^{2}\tr H_{\lambda_{j}}(\eta_{\mathrm{univ},z,t})\tilde{H}_{\overline{\lambda}_{j}}(\eta_{\mathrm{univ},z,t})\\
&=b_{j}^{2}(\tan\theta_{j}-\tan^{-1}\theta_{j})^{2}\eta_{\mathrm{univ},z,t}^{2}\mathfrak{T}+O(|\tan\theta_{j}-\tan^{-1}\theta_{j}|^{2}\eta_{\mathrm{univ},z,t}^{-D}).
\end{align*}
If we exponentiate the previous two displays, the desired estimate follows.
\end{proof}
\subsection{Replacing \texorpdfstring{$\mathbf{Q}_{a_{j},b_{j},\theta_{j}}$}{Q\_aj bj thetaj}}
We now deal with the third line in \eqref{eq:rhomain}. The first step is to replace both copies of $\mathbf{Q}_{a_{j},b_{j},\theta_{j}}$ therein by universal quantities. Since the entries of $\mathbf{Q}_{a_{j},b_{j},\theta_{j}}$ are given by traces of products of blocks of two resolvents, we can again use Lemma \ref{2GE}.
\begin{lemma}\label{lemma:replaceQ}
There exists a $3\times3$ matrix $\mathbf{Q}_{\mathrm{univ},\theta_{j}}$ such that $\mathbf{Q}_{\mathrm{univ},\theta_{j}}$ is independent of the distribution of the entries of $A$, and there exists $\kappa>0$ such that the following estimates hold. First, we have $|\det\mathbf{Q}_{a_{j},b_{j},\theta_{j}}|^{-\frac12}=|\det\mathbf{Q}_{\mathrm{univ},\theta_{j}}|^{-\frac12}\left[1+O(N^{-\kappa})\right]$. We also have 
\begin{align*}
&\exp\left\{-\mathbf{C}_{2}\begin{pmatrix}\frac{N}{t}-\tr \tilde{H}_{a_{j},b_{j},\theta_{j},11}\\-2\tr \tilde{H}_{a_{j},b_{j},\theta_{j},12}\\\frac{N}{t}-\tr \tilde{H}_{a_{j},b_{j},\theta_{j},22}\end{pmatrix}^{\ast}\mathbf{Q}_{a_{j},b_{j},\theta_{j}}^{-1}\begin{pmatrix}\frac{N}{t}-\tr \tilde{H}_{a_{j},b_{j},\theta_{j},11}\\-2\tr \tilde{H}_{a_{j},b_{j},\theta_{j},12}\\\frac{N}{t}-\tr \tilde{H}_{a_{j},b_{j},\theta_{j},22}\end{pmatrix}\right\}\\
&=\exp\left\{-\mathbf{C}_{2}\begin{pmatrix}\frac{N}{t}-\tr \tilde{H}_{a_{j},b_{j},\theta_{j},11}\\-2\tr \tilde{H}_{a_{j},b_{j},\theta_{j},12}\\\frac{N}{t}-\tr \tilde{H}_{a_{j},b_{j},\theta_{j},22}\end{pmatrix}^{\ast}\mathbf{Q}_{\mathrm{univ},\theta_{j}}^{-1}\begin{pmatrix}\frac{N}{t}-\tr \tilde{H}_{a_{j},b_{j},\theta_{j},11}\\-2\tr \tilde{H}_{a_{j},b_{j},\theta_{j},12}\\\frac{N}{t}-\tr \tilde{H}_{a_{j},b_{j},\theta_{j},22}\end{pmatrix}\right\}\left[1+O(N^{-\kappa})\right].
\end{align*}
\end{lemma}
\begin{proof}
We claim there exists a $3\times3$ matrix $\mathbf{Q}_{\mathrm{univ},\theta_{j}}$ independent of the distribution of the entries of $A$ such that $\|\mathbf{Q}_{a_{j},b_{j},\theta_{j}}-\mathbf{Q}_{\mathrm{univ},\theta_{j}}\|_{\mathrm{op}}\lesssim N^{1-\kappa}\eta_{z,t}^{-D}$ with high probability for some $D=O(1)$. To prove this claim, we first show $\|\mathbf{Q}_{a_{j},b_{j},\theta_{j}}-\mathbf{Q}_{a_{j},b_{j},\theta_{j}}(\eta_{\mathrm{univ},z,t})\|_{\mathrm{op}}\lesssim N^{1-\kappa}\eta_{z,t}^{-D}$, where $\mathbf{Q}_{a_{j},b_{j},\theta_{j}}(\eta_{\mathrm{univ},z,t})$ is just $\mathbf{Q}_{a_{j},b_{j},\theta_{j}}$ but evaluating all resolvents at $\eta_{\mathrm{univ},z,t}$ instead of $\eta_{z,t}$. To see this, we first note that $\|\mathbf{Q}_{a_{j},b_{j},\theta_{j}}-\mathbf{Q}_{a_{j},b_{j},\theta_{j}}(\eta_{\mathrm{univ},z,t})\|_{\mathrm{op}}\lesssim\|\mathbf{Q}_{a_{j},b_{j},\theta_{j}}-\mathbf{Q}_{a_{j},b_{j},\theta_{j}}(\eta_{\mathrm{univ},z,t})\|_{\max}$ since matrices in question are $3\times3$. We now estimate the entries of $\mathbf{Q}_{a_{j},b_{j},\theta_{j}}-\mathbf{Q}_{a_{j},b_{j},\theta_{j}}(\eta_{\mathrm{univ},z,t})$.  We illustrate one example; other entries are treated similarly. We are claiming that $\tr \tilde{H}_{a_{j},b_{j},\theta_{j},11}(\eta_{z,t})^{2}=\tr \tilde{H}_{a_{j},b_{j},\theta_{j},11}(\eta_{\mathrm{univ},z,t})^{2}+O(N^{1-\kappa}\eta_{z,t}^{-D})$. This follows by Lemma \ref{lemma:replaceconstants}, the bound $\|\tilde{H}_{a_{j},b_{j},\theta_{j},11}\|_{\mathrm{op}}\lesssim\eta_{z,t}^{-2}$, and resolvent identities. The matrix $\mathbf{Q}_{\mathrm{univ},\theta_{j}}$ is then constructed by applying Lemma \ref{2GE} to each entry in $\mathbf{Q}_{a_{j},b_{j},\theta_{j}}(\eta_{\mathrm{univ},z,t})$.

We now prove the two proposed estimates. For the first determinant estimate, we note that 
\begin{align*}
|\det\mathbf{Q}_{a_{j},b_{j},\theta_{j}}|^{-\frac12}&=|\det\mathbf{Q}_{\mathrm{univ},\theta_{j}}|^{-\frac12}\left[\frac{|\det\mathbf{Q}_{\mathrm{univ},\theta_{j}}|}{|\det\mathbf{Q}_{a_{j},b_{j},\theta_{j}}|}\right]^{\frac12}\\
&=|\det\mathbf{Q}_{\mathrm{univ},\theta_{j}}|^{-\frac12}\left[1+\frac{|\det\mathbf{Q}_{\mathrm{univ},\theta_{j}}|-|\det\mathbf{Q}_{a_{j},b_{j},\theta_{j}}|}{|\det\mathbf{Q}_{a_{j},b_{j},\theta_{j}}|}\right]^{\frac12}.
\end{align*}
Since the entries in $\mathbf{Q}_{a_{j},b_{j},\theta_{j}}$ are $\lesssim N\eta_{z,t}^{-D}$ for some $D=O(1)$, the same must be true for $\mathbf{Q}_{\mathrm{univ},\theta_{j}}$ by the operator norm estimate $\|\mathbf{Q}_{a_{j},b_{j},\theta_{j}}-\mathbf{Q}_{\mathrm{univ},\theta_{j}}\|_{\mathrm{op}}\lesssim N^{1-\kappa}\eta_{z,t}^{-D}$. This same estimate combined with an elementary expansion of the determinant (as a cubic in the entries of these $3\times3$ matrices) then shows that $|\det\mathbf{Q}_{\mathrm{univ},\theta_{j}}|-|\det\mathbf{Q}_{a_{j},b_{j},\theta_{j}}|=O(N^{3-\kappa}\eta_{z,t}^{-D})$ for possibly different $D=O(1)$. On the other hand, Lemma \ref{lemma:qlowerbound} shows that $|\det\mathbf{Q}_{a_{j},b_{j},\theta_{j}}|\gtrsim N^{3}\eta_{z,t}^{-6}$. We deduce from this paragraph, the bound $\eta_{z,t}\asymp t=N^{-\epsilon_{0}}$ with $\epsilon_{0}>0$ sufficiently small, and the previous display that 
\begin{align*}
|\det\mathbf{Q}_{a_{j},b_{j},\theta_{j}}|^{-\frac12}&=|\det\mathbf{Q}_{\mathrm{univ},\theta_{j}}|^{-\frac12}\left[1+O(N^{-\kappa})\right].
\end{align*}
This is the first desired estimate. We move to the exponential estimate. We must estimate the difference of inverses $\mathbf{Q}_{a_{j},b_{j},\theta_{j}}^{-1}-\mathbf{Q}_{\mathrm{univ},\theta_{j}}^{-1}$. By resolvent identities, we have 
\begin{align*}
\mathbf{Q}_{a_{j},b_{j},\theta_{j}}^{-1}-\mathbf{Q}_{\mathrm{univ},\theta_{j}}^{-1}=\mathbf{Q}_{a_{j},b_{j},\theta_{j}}^{-1}(\mathbf{Q}_{\mathrm{univ},\theta_{j}}-\mathbf{Q}_{a_{j},b_{j},\theta_{j}})\mathbf{Q}_{\mathrm{univ},\theta_{j}}^{-1}.
\end{align*}
Lemma \ref{lemma:qlowerbound} shows that $\mathbf{Q}_{a_{j},b_{j},\mathrm{univ}}\gtrsim N\eta_{z,t}^{-2}$. The operator bound $\|\mathbf{Q}_{a_{j},b_{j},\theta_{j}}-\mathbf{Q}_{\mathrm{univ},\theta_{j}}\|_{\mathrm{op}}\lesssim N^{1-\kappa}\eta_{z,t}^{-D}$, if we choose $\epsilon_{0}>0$ sufficiently small, then implies the same for $\mathbf{Q}_{\mathrm{univ},\theta_{j}}$. Thus, if we choose $\epsilon_{0}>0$ small enough, we get
\begin{align*}
\|\mathbf{Q}_{a_{j},b_{j},\theta_{j}}^{-1}-\mathbf{Q}_{\mathrm{univ},\theta_{j}}^{-1}\|_{\mathrm{op}}\lesssim N^{-2}\eta_{z,t}^{4}N^{1-\kappa}\eta_{z,t}^{-D}\lesssim N^{-1-\frac{\kappa}{2}}.
\end{align*}
To derive the proposed exponential estimate, we use \eqref{eq:Vestimate} and Cauchy interlacing to get
\begin{align*}
\left\|\begin{pmatrix}\frac{N}{t}-\tr \tilde{H}_{a_{j},b_{j},\theta_{j},11}\\-2\tr \tilde{H}_{a_{j},b_{j},\theta_{j},12}\\\frac{N}{t}-\tr \tilde{H}_{a_{j},b_{j},\theta_{j},22}\end{pmatrix}\right\|\lesssim_{\tau} N^{\frac12+\tau}\eta_{z,t}^{-D},
\end{align*}
where $\tau>0$ is as small as we want. (To be clear, $\lesssim_{\tau}$ means bounded above up to a constant depending on $\tau$; in this case, the constant blows up as $\tau\to0$.) Combining the previous two displays yields the quadratic form estimate
\begin{align*}
&\begin{pmatrix}\frac{N}{t}-\tr \tilde{H}_{a_{j},b_{j},\theta_{j},11}\\-2\tr \tilde{H}_{a_{j},b_{j},\theta_{j},12}\\\frac{N}{t}-\tr \tilde{H}_{a_{j},b_{j},\theta_{j},22}\end{pmatrix}^{\ast}\mathbf{Q}_{a_{j},b_{j},\theta_{j}}^{-1}\begin{pmatrix}\frac{N}{t}-\tr \tilde{H}_{a_{j},b_{j},\theta_{j},11}\\-2\tr \tilde{H}_{a_{j},b_{j},\theta_{j},12}\\\frac{N}{t}-\tr \tilde{H}_{a_{j},b_{j},\theta_{j},22}\end{pmatrix}\\
&=\begin{pmatrix}\frac{N}{t}-\tr \tilde{H}_{a_{j},b_{j},\theta_{j},11}\\-2\tr \tilde{H}_{a_{j},b_{j},\theta_{j},12}\\\frac{N}{t}-\tr \tilde{H}_{a_{j},b_{j},\theta_{j},22}\end{pmatrix}^{\ast}\mathbf{Q}_{\mathrm{univ},\theta_{j}}^{-1}\begin{pmatrix}\frac{N}{t}-\tr \tilde{H}_{a_{j},b_{j},\theta_{j},11}\\-2\tr \tilde{H}_{a_{j},b_{j},\theta_{j},12}\\\frac{N}{t}-\tr \tilde{H}_{a_{j},b_{j},\theta_{j},22}\end{pmatrix}+O(N^{-\kappa+2\tau}\eta_{z,t}^{-D})
\end{align*}
for $\kappa>0$ and $\tau>0$ (which we can choose to be small) and $D=O(1)$. If we choose $\tau>0$ and $\epsilon_{0}>0$ sufficiently small (depending only on $\kappa$), the error term in the second line of the previous display becomes $O(N^{-\delta})$ for some $\delta>0$. Exponentiating the resulting estimate completes the proof.
\end{proof}
\subsection{Replacing the vector in the quadratic form}
The final step is to replace the traces in the vector in the third line \eqref{eq:rhomain} by universal quantities. To this end, we use Lemma \ref{1GE}.
\begin{lemma}\label{lemma:finalreplace}
There exists $\mathfrak{v}\in\R^{3}$ whose entries are independent of the distribution of the entries of $A$, and such that the following estimate holds for some $\kappa>0$:
\begin{align*}
&\exp\left\{-\mathbf{C}_{2}\begin{pmatrix}\frac{N}{t}-\tr \tilde{H}_{a_{j},b_{j},\theta_{j},11}\\-2\tr \tilde{H}_{a_{j},b_{j},\theta_{j},12}\\\frac{N}{t}-\tr \tilde{H}_{a_{j},b_{j},\theta_{j},22}\end{pmatrix}^{\ast}\mathbf{Q}_{\mathrm{univ},\theta_{j}}^{-1}\begin{pmatrix}\frac{N}{t}-\tr \tilde{H}_{a_{j},b_{j},\theta_{j},11}\\-2\tr \tilde{H}_{a_{j},b_{j},\theta_{j},12}\\\frac{N}{t}-\tr \tilde{H}_{a_{j},b_{j},\theta_{j},22}\end{pmatrix}\right\}\\
&=\exp\left\{-\mathbf{C}_{2}\mathfrak{v}^{\ast}\mathbf{Q}_{\mathrm{univ},\theta_{j}}^{-1}\mathfrak{v}\right\}\left[1+O(N^{-\kappa})\right]
\end{align*}
\end{lemma}
\begin{proof}
We first claim there exists $\mathfrak{v}\in\R^{3}$ whose entries are independent of the distribution of the entries of $A$ such that for some $\kappa>0$ and $D=O(1)$, we have the following with high probability:
\begin{align}
\left\|\begin{pmatrix}\frac{N}{t}-\tr \tilde{H}_{a_{j},b_{j},\theta_{j},11}\\-2\tr \tilde{H}_{a_{j},b_{j},\theta_{j},12}\\\frac{N}{t}-\tr \tilde{H}_{a_{j},b_{j},\theta_{j},22}\end{pmatrix}-\mathfrak{v}\right\|\lesssim N^{\frac12-\kappa}.\label{eq:finalreplace1}
\end{align}
We first note the following, where the last line follows by computations in Appendix \ref{section:resolventsatpi/4}:
\begin{align*}
\tr \tilde{H}_{a_{j},b_{j},\theta_{j},11}&=\tr \tilde{H}_{a_{j},b_{j},\frac{\pi}{4},11}+\tr[\tilde{H}_{a_{j},b_{j},\theta_{j},11}-\tilde{H}_{a_{j},b_{j},\frac{\pi}{4},11}]\\
&=\tr \tilde{H}_{a_{j},b_{j},\frac{\pi}{4},11}+\tr\tilde{H}_{a_{j},b_{j},\theta_{j},11}[A_{a_{j},b_{j},\frac{\pi}{4}}^{\ast}A_{a_{j},b_{j},\frac{\pi}{4}}-A_{a_{j},b_{j},\theta_{j}}^{\ast}A_{a_{j},b_{j},\theta_{j}}]\tilde{H}_{a_{j},b_{j},\frac{\pi}{4},11}\\
&=\frac{N}{t}+\tr\tilde{H}_{a_{j},b_{j},\theta_{j},11}[A_{a_{j},b_{j},\frac{\pi}{4}}^{\ast}A_{a_{j},b_{j},\frac{\pi}{4}}-A_{a_{j},b_{j},\theta_{j}}^{\ast}A_{a_{j},b_{j},\theta_{j}}]\tilde{H}_{a_{j},b_{j},\frac{\pi}{4},11}.
\end{align*}
Since $|\theta_{j}-\frac{\pi}{4}|\leq N^{-1/2+\tau}$, we know $\|A_{a_{j},b_{j},\frac{\pi}{4}}^{\ast}A_{a_{j},b_{j},\frac{\pi}{4}}-A_{a_{j},b_{j},\theta_{j}}^{\ast}A_{a_{j},b_{j},\theta_{j}}\|_{\mathrm{op}}\lesssim N^{-1/2+\tau}$, where $\tau>0$ is small. In particular, by resolvent perturbation, Lemma \ref{lemma:replaceconstants}, and trivial resolvent bounds, we have 
\begin{align*}
&\tr\tilde{H}_{a_{j},b_{j},\theta_{j},11}[A_{a_{j},b_{j},\frac{\pi}{4}}^{\ast}A_{a_{j},b_{j},\frac{\pi}{4}}-A_{a_{j},b_{j},\theta_{j}}^{\ast}A_{a_{j},b_{j},\theta_{j}}]\tilde{H}_{a_{j},b_{j},\frac{\pi}{4},11}\\
&=\tr\left\{\tilde{H}_{a_{j},b_{j},\theta_{j},11}(\eta_{\mathrm{univ},z,t})[A_{a_{j},b_{j},\frac{\pi}{4}}^{\ast}A_{a_{j},b_{j},\frac{\pi}{4}}-A_{a_{j},b_{j},\theta_{j}}^{\ast}A_{a_{j},b_{j},\theta_{j}}]\tilde{H}_{a_{j},b_{j},\frac{\pi}{4},11}(\eta_{\mathrm{univ},z,t})\right\}+O(N^{1/2-\kappa}),
\end{align*}
where $\kappa>0$ is fixed and independent of our choice of small $\epsilon_{0}>0$. By another resolvent identity, we have 
\begin{align*}
&\tr\left\{\tilde{H}_{a_{j},b_{j},\theta_{j},11}(\eta_{\mathrm{univ},z,t})[A_{a_{j},b_{j},\frac{\pi}{4}}^{\ast}A_{a_{j},b_{j},\frac{\pi}{4}}-A_{a_{j},b_{j},\theta_{j}}^{\ast}A_{a_{j},b_{j},\theta_{j}}]\tilde{H}_{a_{j},b_{j},\frac{\pi}{4},11}(\eta_{\mathrm{univ},z,t})\right\}\\
&=\tr\tilde{H}_{a_{j},b_{j},\theta_{j},11}(\eta_{\mathrm{univ},z,t})-\tr\tilde{H}_{a_{j},b_{j},\frac{\pi}{4},11}(\eta_{\mathrm{univ},z,t}).
\end{align*}
In particular, the previous three displays imply
\begin{align*}
\frac{N}{t}-\tr \tilde{H}_{a_{j},b_{j},\theta_{j},11}=\tr\tilde{H}_{a_{j},b_{j},\frac{\pi}{4},11}(\eta_{\mathrm{univ},z,t})-\tr\tilde{H}_{a_{j},b_{j},\theta_{j},11}(\eta_{\mathrm{univ},z,t}).
\end{align*}
The RHS of the previous display is equal to a quantity which does not depend on the distribution of the entries of $A$ plus $O(\eta_{z,t}^{-D})$ for some $D=O(1)$; this is by the local law in Lemma \ref{1GE}. The aforementioned ``universal" term is our choice of $\mathfrak{v}_{1}$, the first entry of $\mathfrak{v}$. The rest of $\mathfrak{v}$ is constructed using a similar argument. We now compute 
\begin{align}
&\begin{pmatrix}\frac{N}{t}-\tr \tilde{H}_{a_{j},b_{j},\theta_{j},11}\\-2\tr \tilde{H}_{a_{j},b_{j},\theta_{j},12}\\\frac{N}{t}-\tr \tilde{H}_{a_{j},b_{j},\theta_{j},22}\end{pmatrix}^{\ast}\mathbf{Q}_{\mathrm{univ},\theta_{j}}^{-1}\begin{pmatrix}\frac{N}{t}-\tr \tilde{H}_{a_{j},b_{j},\theta_{j},11}\\-2\tr \tilde{H}_{a_{j},b_{j},\theta_{j},12}\\\frac{N}{t}-\tr \tilde{H}_{a_{j},b_{j},\theta_{j},22}\end{pmatrix}-\mathfrak{v}^{\ast}\mathbf{Q}_{\mathrm{univ},\theta_{j}}^{-1}\mathfrak{v}\label{eq:finalreplace2}\\
&=\left[\begin{pmatrix}\frac{N}{t}-\tr \tilde{H}_{a_{j},b_{j},\theta_{j},11}\\-2\tr \tilde{H}_{a_{j},b_{j},\theta_{j},12}\\\frac{N}{t}-\tr \tilde{H}_{a_{j},b_{j},\theta_{j},22}\end{pmatrix}-\mathfrak{v}\right]^{\ast}\mathbf{Q}_{\mathrm{univ},j}^{-1}\begin{pmatrix}\frac{N}{t}-\tr \tilde{H}_{a_{j},b_{j},\theta_{j},11}\\-2\tr \tilde{H}_{a_{j},b_{j},\theta_{j},12}\\\frac{N}{t}-\tr \tilde{H}_{a_{j},b_{j},\theta_{j},22}\end{pmatrix}\nonumber\\
&+\mathfrak{v}^{\ast}\mathbf{Q}_{\mathrm{univ},\theta_{j}}^{-1}\left[\begin{pmatrix}\frac{N}{t}-\tr \tilde{H}_{a_{j},b_{j},\theta_{j},11}\\-2\tr \tilde{H}_{a_{j},b_{j},\theta_{j},12}\\\frac{N}{t}-\tr \tilde{H}_{a_{j},b_{j},\theta_{j},22}\end{pmatrix}-\mathfrak{v}\right].\nonumber
\end{align}
By \eqref{eq:Vestimate}, the last vector in the second line of \eqref{eq:finalreplace2} has norm $\lesssim N^{\frac12+\tau}\eta_{z,t}^{-D}$ for $\tau>0$ small and $D=O(1)$. By \eqref{eq:finalreplace1}, the same is true of $\mathfrak{v}$. Moreover, as shown in the proof of Lemma \ref{lemma:replaceQ}, we know that $\mathbf{Q}_{\mathrm{univ},j}^{-1}\lesssim N^{-1}\eta_{z,t}^{2}$ with high probability. So, the second line in \eqref{eq:finalreplace2} is $\lesssim N^{\frac12+\tau}N^{\frac12-\kappa}\eta_{z,t}^{-D}N^{-1}\eta_{z,t}^{2}\lesssim N^{-\kappa+\tau}\eta_{z,t}^{-D}$ for possibly different $D=O(1)$ and for small $\tau>0$. By the same token, the same is true for the last line in \eqref{eq:finalreplace2}. Thus, we get
\begin{align*}
\begin{pmatrix}\frac{N}{t}-\tr \tilde{H}_{a_{j},b_{j},\theta_{j},11}\\-2\tr \tilde{H}_{a_{j},b_{j},\theta_{j},12}\\\frac{N}{t}-\tr \tilde{H}_{a_{j},b_{j},\theta_{j},22}\end{pmatrix}^{\ast}\mathbf{Q}_{\mathrm{univ},\theta_{j}}^{-1}\begin{pmatrix}\frac{N}{t}-\tr \tilde{H}_{a_{j},b_{j},\theta_{j},11}\\-2\tr \tilde{H}_{a_{j},b_{j},\theta_{j},12}\\\frac{N}{t}-\tr \tilde{H}_{a_{j},b_{j},\theta_{j},22}\end{pmatrix}-\mathfrak{v}^{\ast}\mathbf{Q}_{\mathrm{univ},\theta_{j}}^{-1}\mathfrak{v}=O(N^{-\kappa+\tau}\eta_{z,t}^{-D}).
\end{align*}
It now suffices to choose $\tau,\epsilon_{0}>0$ small enough to make the RHS $O(N^{-\kappa/2})$, at which point we then exponentiate the resulting bound to conclude the proof.
\end{proof}
%
%
%
\section{Proof of Theorem \ref{theorem:mainwitht}}
Combining \eqref{eq:rhomain} with Lemmas \ref{lemma:replaceconstants}, \ref{lemma:replacedetG}, \ref{lemma:replacetrHH}, \ref{lemma:replaceQ}, and \ref{lemma:finalreplace} shows that 
\begin{align}
\rho_{t}(z,\mathbf{z};A)&=\Phi_{z,t}(z_{1},z_{2})\rho_{\mathrm{GinUE}}^{(2)}(z_{1},z_{2})\left[1+O(N^{-\kappa})\right]+O(\exp[-CN^{\kappa}]),\label{eq:mainwitht1}
\end{align}
where $\kappa>0$ and $\Phi_{z,t}(z_{1},z_{2})$ is as in the statement of Theorem \ref{theorem:mainwitht}. Recall $\rho_{t}(z,\mathbf{z};A)$ is the two-point correlation function for local eigenvalue statistics near $z$; more precisely, we have the following (in which we recall $\mathbf{z}=(z_{1},z_{2})$):
\begin{align*}
\int_{\C^{2}}O(z_{1},z_{2})\rho_{t}(z,\mathbf{z};A)dz_{1}dz_{2}=\E\left[\sum_{i_{1}\neq i_{2}}O(N^{\frac12}\sigma_{z,t}^{\frac12}[z-\lambda_{i_{1}}(t)],N^{\frac12}\sigma_{z,t}^{\frac12}[z-\lambda_{i_{2}}(t)])\right].
\end{align*}
Above, $O(z_{1},z_{2})\in C^{\infty}_{c}(\C^{2})$ is arbitrary. Since $\rho_{t}(z,\mathbf{z};A)$ is a probability density (up to a scaling by a deterministic, $O(1)$ factor) with respect to $dz_{1}dz_{2}$, the estimate \eqref{eq:mainwitht1} gives 
\begin{align*}
&\int_{\C^{2}}O(z_{1},z_{2})\rho_{t}(z,\mathbf{z};A)dz_{1}dz_{2}=\int_{\C^{2}}O(z_{1},z_{2})\rho_{t}(z,\mathbf{z};A)[1+O(N^{-\kappa})]dz_{1}dz_{2}+O(N^{-\kappa})\\
&=\int_{\C^{2}}O(z_{1},z_{2})\Phi_{z,t}(z_{1},z_{2})\rho_{\mathrm{GinUE}}^{(2)}(z_{1},z_{2})dz_{1}dz_{2}+O(N^{-\kappa})
\end{align*}
for some $\kappa>0$. It now suffices to combine the previous two displays. \qed
%
%
%
\section{Proof of Theorem \ref{theorem:main}}
It suffices to prove Theorem \ref{theorem:main} for $\tilde{A}_{t}:=A\sqrt{1+t}$, where $t=N^{-\epsilon_{0}}$ as in Theorem \ref{theorem:mainwitht}. Indeed, let $\rho_{A,z}(z_{1},z_{2})$ be the two-point correlation function for eigenvalues of $A$ near $z$, and $\rho_{\tilde{A}_{t},z}(z_{1},z_{2})$ is the same but for $\tilde{A}_{t}$. Then, for any $O\in C^{\infty}_{c}(\C^{2})$ we have 
\begin{align*}
\int_{\C^{2}}O(z_{1},z_{2})\rho_{A,z}(z_{1},z_{2})dz_{1}dz_{2}&=\frac{1}{(1+t)^{2}}\int_{\C^{2}}O\left(\frac{z_{1}}{\sqrt{1+t}},\frac{z_{1}}{\sqrt{1+t}}\right)\rho_{\tilde{A}_{t},z}(z_{1},z_{2})dz_{1}dz_{2}\\
&=\int_{\C^{2}}O\left(\frac{z_{1}}{\sqrt{1+t}},\frac{z_{1}}{\sqrt{1+t}}\right)\rho_{\tilde{A}_{t},z}(z_{1},z_{2})dz_{1}dz_{2}+O(t)\\
&=\int_{\C^{2}}O(z_{1},z_{2})\rho_{\tilde{A}_{t},z}(z_{1},z_{2})dz_{1}dz_{2}+O(t)
\end{align*}
by change of variables on $\C^{2}$. Since $t=N^{-\epsilon_{0}}\to0$, we have reduced to proving Theorem \ref{theorem:main} for $\tilde{A}_{t}$ instead of $A$. To this end, the main ingredient we require is the following ``three-and-a-half moment matching" theorem, which we state in more generality. Before we state this theorem, we first say that $X,\tilde{X}$ match up to three and a half moments if:
\begin{enumerate}
\item $X$ is an $N\times N$ matrix whose entries $X_{ij}$ are real i.i.d. variables that satisfy $\E|X_{ij}|^{p}\lesssim_{p}N^{-p/2}$ for $1\leq p\leq4$.
\item $\tilde{X}$ is an $N\times N$ matrix whose entries $\tilde{X}_{ij}$ satisfy the same properties.
\item $\E X_{ij}^{p}=\E\tilde{X}_{ij}^{p}$ for $p=1,2,3$, and $\E|X_{ij}|^{4}=\E|\tilde{X}_{ij}|^{4}+O(N^{-2-\delta})$ for some $\delta>0$. (Note that $N^{-2-\delta}$ is the below the natural scale of $N^{-2}$ for fourth moments.)
\end{enumerate}
\begin{lemma}\label{lemma:3.5moment}
Suppose that $X$ and $\tilde{X}$ are $N\times N$ matrices that match up to three and a half moments. Let $\{\lambda^{X}_{i}\}_{i}$ and $\{\lambda^{\tilde{X}}\}_{i}$ be eigenvalues of $X$ and $\tilde{X}$, respectively. Fix any $k\geq1$ and $z\in\C$ such that $|z|\leq1-\tau$ and $\mathrm{Im}(z)\geq\tau$ for some $\tau>0$ fixed. For any $O\in C^{\infty}_{c}(\C^{k})$, we have
\begin{align*}
&\E\left\{\sum_{i_{1}\neq i_{2}\neq\ldots\neq i_{k}}\left[O(N^{\frac12}[z-\lambda^{X}_{i_{1}}],\ldots,N^{\frac12}[z-\lambda^{X}_{i_{k}}])-O(N^{\frac12}[z-\lambda^{\tilde{X}}_{i_{1}}],\ldots,N^{\frac12}[z-\lambda^{\tilde{X}}_{i_{k}}])\right]\right\}\to_{N\to\infty}0.
\end{align*}
\end{lemma}
\begin{proof}[Proof of Theorem \ref{theorem:main} given Lemma \ref{lemma:3.5moment}]
Let $t=N^{-\epsilon_{0}}$ as in Theorem \ref{theorem:mainwitht}. By Lemma 3.4 in \cite{EYY11}, we can find $N\times N$ matrices $\mathbf{A}$ and $B$ such that $B$ is real Ginibre (entries are independent $N(0,N^{-1})$ random variables), such that $\mathbf{A}$ and $B$ are independent, such that $\mathbf{A}+\sqrt{t}B$ and $\tilde{A}_{t}$ match up to three and a half moments, and such that $\E|\mathbf{A}_{ij}|^{p}\lesssim_{p} N^{-p/2}$ for all $p\geq1$. (Indeed, for this last property, note that it is true for $\tilde{A}_{t}$ by assumption and for $B$.) We necessarily have $\E|\mathbf{A}_{ij}|^{2}=N^{-1}$ (again, since $\E|\tilde{A}_{t}|^{2}=(1+t)N^{-1}$ and $\E|B_{ij}|^{2}=N^{-1}$. Ultimately, by Lemma \ref{lemma:3.5moment} (for the test function $O$ precomposed with scaling by a universal $O(1)$ constant), it suffices to prove Theorem \ref{theorem:main} for $\mathbf{A}+\sqrt{t}B$ instead of for $\tilde{A}_{t}$. This amounts to replacing $\sigma_{z,t}$ in Theorem \ref{theorem:main} by its universal approximation $\sigma_{\mathrm{univ},z,t}$ from Lemma \ref{lemma:replaceconstants}, and then using Theorem \ref{theorem:mainwitht}. We give the details below.

Let $\{\lambda_{j}^{\mathbf{A}+\sqrt{t}B}\}_{j}$ denote the eigenvalues of $\mathbf{A}+\sqrt{t}B$. Fix $z\in\C$ as in the statement of Theorem \ref{theorem:main}, and let $\sigma_{z,t}$ be defined as in the introduction but for $\mathbf{A}$ and its resolvents. We first claim that 
\begin{align}
&\E\left[\sum_{i_{1}\neq i_{2}}O(N^{\frac12}\sigma_{z,t}^{\frac12}[z-\lambda_{i_{1}}],N^{\frac12}\sigma_{z,t}^{\frac12}[z-\lambda_{i_{1}}])\right]\label{eq:main1}\\
&=\E\left[\sum_{i_{1}\neq i_{2}}O(N^{\frac12}\sigma_{\mathrm{univ},z,t}^{\frac12}[z-\lambda_{i_{1}}],N^{\frac12}\sigma_{\mathrm{univ},z,t}^{\frac12}[z-\lambda_{i_{1}}])\right]+o(1),\nonumber
\end{align}
where $\sigma_{\mathrm{univ},z,t}$ is from Lemma \ref{lemma:replaceconstants}. Indeed, with high probability, we know that $\sigma_{z,t}=\sigma_{\mathrm{univ},z,t}[1+O(N^{-\kappa})]$ for some $\kappa>0$. If $\rho_{\mathbf{A}+\sqrt{t}B,z}(z_{1},z_{2})$ denotes the two-point correlation function for $\mathbf{A}+\sqrt{t}B$ near $z$, then 
\begin{align*}
&\E\left[\sum_{i_{1}\neq i_{2}}O(N^{\frac12}\sigma_{\mathrm{univ},z,t}^{\frac12}[z-\lambda_{i_{1}}],N^{\frac12}\sigma_{\mathrm{univ},z,t}^{\frac12}[z-\lambda_{i_{1}}])\right]\\
&=\int_{\C^{2}}O\left(\sigma_{\mathrm{univ},z,t}^{\frac12}\sigma_{z,t}^{-\frac12}z_{1},\sigma_{\mathrm{univ},z,t}^{\frac12}\sigma_{z,t}^{-\frac12}z_{2}\right)\frac{1}{N^{2}\sigma_{z,t}^{2}}\rho_{\mathbf{A}+\sqrt{t}B}\left(z+\frac{z_{1}}{\sqrt{N\sigma_{z,t}}},z+\frac{z_{2}}{\sqrt{N\sigma_{z,t}}}\right)dz_{1}dz_{2},
\end{align*}
at which point \eqref{eq:main1} follows by $\sigma_{\mathrm{univ},z,t}\sigma_{z,t}^{-1}=1+O(N^{-\kappa})$ and an elementary Taylor expansion of $O$. To conclude the proof, we now combine \eqref{eq:main1} with Theorem \ref{theorem:mainwitht} (applied to $\mathbf{A}$).
\end{proof}
\subsection{Proof of Lemma \ref{lemma:3.5moment}}
By a standard approximation procedure, it suffices to prove Lemma \ref{lemma:3.5moment} by functions of the form $O(z_{1},\ldots,z_{k})=f^{(1)}(z_{1})\ldots f^{(k)}(z_{k})$. The idea behind the following calculation is to use Girko's Hermitization formula to bring Lemma \ref{lemma:3.5moment} into the realm of real symmetric matrices; this computation was also explored in \cite{CES20}. Before we proceed, we introduce some notation. Define
\begin{align*}
\mathsf{G}_{z}(\eta):=\begin{pmatrix}-i\eta&X-z\\X^{\ast}-\overline{z}&-i\eta\end{pmatrix},\quad \tilde{\mathsf{G}}_{z}(\eta):=\begin{pmatrix}-i\eta&X-z\\\tilde{X}^{\ast}-\overline{z}&-i\eta\end{pmatrix}.
\end{align*}
We use different font to distinguish from $G_{z}$, which was meant for $A$ earlier in this paper. We also define $m^{z}(w)$ to be the unique solution to
\begin{align}\label{eq:mz-equation}
-\frac{1}{m^{z}(w)}=w+m^{z}(w)-\frac{|z|^{2}}{w+m^{z}(w)},\quad \mathrm{Im}[m^{z}(w)]\mathrm{Im}[w]>0.
\end{align}
By the inclusion-exclusion principle, it suffices to show that if $\sigma_{i}$ and $\tilde{\sigma}_{i}$ denote eigenvalues of $X$ and $\tilde{X}$, respectively, we have
$$
\begin{aligned}
\mathbf{E} \prod_{j=1}^k & \left(\frac{1}{N} \sum_{i=1}^{N} f_{z_j}^{(j)}\left(\sigma_i\right)-\frac{1}{\pi} \int_{\mathbf{D}} f_{z_j}^{(j)}(z)  d^{2} z\right) =\mathbf{E} \prod_{j=1}^k\left(\frac{1}{N} \sum_{i=1}^{N} f_{z_j}^{(j)}\left(\widetilde{\sigma}_i\right)-\frac{1}{\pi} \int_{\mathbf{D}} f_{z_j}^{(j)}(z)  d^{2} z\right)+O(N^{-c})
\end{aligned}
$$
where we introduced the rescaled test functions
$$
f_{z_j}^{(j)}(z):=Nf^{(j)}\left(\sqrt{N}\left(z-z_j\right)\right), \quad z \in \mathbf{C},
$$
and the implicit constant in $O(\cdot)$ depends on $\|\Delta f^{(j)}\|_{L^{1}(\C)}$ for $j=1, \ldots, k$. We adjusted notation slightly to fit that of \cite{CES20}. By Theorem 2.4 in \cite{CES20}, we have
\begin{align*}
\left(\frac{1}{N} \sum_{i=1}^{N} f_{z_j}^{(j)}\left(\sigma_i\right)-\frac{1}{\pi} \int_{\mathbf{D}} f_{z_j}^{(j)}(z)  d^{2} z\right)=\mathcal{I}_{\epsilon}(X,f_{z_{j}}^{(j)})+\mathcal{E}_{\epsilon}
\end{align*}
where $\epsilon>0$ is arbitrary, where $\E|\mathcal{E}_{\epsilon}|\lesssim_{\tau}N^{-\frac14\epsilon}\|\Delta f^{(j)}\|_{L^{1}(\C)}$, and where 
\begin{align*}
\mathcal{I}_\epsilon\left(X, f_{z_{j}}^{(j)}\right):=\frac{1}{2 \pi} \int_{\mathbf{C}} \Delta f_{z_{j}}^{(j)}(z) \int_{N^{-1-\epsilon}}^{N^{-1+\epsilon}}\left\langle\mathrm{Im} \mathsf{G}_{z}(\eta)-\mathrm{Im}m^z(i\eta)\right\rangle d\eta d^{2}z.
\end{align*}
(Recall that $\langle\rangle$ denotes normalized trace.) The same computations and estimates hold for $\tilde{X}$ in place of $X$. Thus, the proof of Lemma \ref{lemma:3.5moment} reduces to proving the following. (In Lemma \ref{lemma:3.5momentreduced} below, the constant $\delta>0$ is the exponent in the moment matching $\E X_{ij}^{p}=\E\tilde{X}_{ij}^{p}$ for $p=1,2,3$, and $\E|X_{ij}|^{4}=\E|\tilde{X}_{ij}|^{4}+O(N^{-2-\delta})$.)
\begin{lemma}\label{lemma:3.5momentreduced}
Take the assumptions in Lemma \ref{lemma:3.5moment}. Fix any $k\geq1$. There exists small  $\epsilon=\epsilon(k,\delta)>0$ such that uniformly over $N^{-1-\epsilon}\leq\eta_{\ell}\leq N^{-1+\epsilon}$ for $\ell=1,\ldots,k$, we have
\begin{align}
\E\prod_{l=1}^k\left\langle\mathrm{Im}\mathsf{G}_{z}(i\eta_{\ell})-\mathrm{Im}m^{z_l}(i\eta_{\ell})\right\rangle-\E\prod_{l=1}^k\left\langle\mathrm{Im}\widetilde{\mathsf{G}}_{z_{\ell}}(i\eta_{\ell})-\mathrm{Im}m^{z_{\ell}}(i\eta_{\ell}) \right\rangle=\mathcal{O}\left(N^{-c\delta}\right).\label{eq:3.5momentreduced}
\end{align}
\end{lemma}
The proof of Lemma \ref{lemma:3.5momentreduced} uses a standard Green's function comparison argument. We adopt the continuous comparison method introduced in \cite{KY16}, which is based on the following construction.
    
\begin{defn}[Interpolating matrices] 
Define the following matrices:
$$\mathcal{H}^0:=\widetilde{\mathcal{H}}=\begin{pmatrix}
   0 & \widetilde X   \\
  \widetilde  X^T & 0
    \end{pmatrix},\quad \mathcal{H}^1:=\mathcal{H}=\begin{pmatrix}
   0 &   X   \\
     X^T & 0
    \end{pmatrix}$$
Let $\rho_{i j}^0$ and $\rho_{i j}^1$ denote the laws of $\widetilde{\mathcal{H}}_{i j}$ and $\mathcal{H}_{i j}$, respectively. For $\theta \in[0,1]$, we define the interpolated laws $\rho_{i j}^\theta:=(1-\theta) \rho_{i j}^0+\theta \rho_{i j}^1$. Let $\left\{\mathcal{H}^\theta: \theta \in(0,1)\right\}$ be a collection of random matrices that satisfy the following properties. For any fixed $\theta \in(0,1)$, the triple $\left(\mathcal{H}^0, \mathcal{H}^\theta, \mathcal{H}^1\right)$ of $2 N \times 2 N$ random matrices are jointly independent, and the matrix $\mathcal{H}^\theta=\left(\mathcal{H}_{i j}^\theta\right)$ has law
$$
\prod_{i \leq j } \rho_{i j}^\theta\left(\mathrm{d} \mathcal{H}_{i j}^\theta\right)
$$
(We do not require $\mathcal{H}^{\theta_1}$ to be independent of $\mathcal{H}^{\theta_2}$ for $\theta_1 \neq \theta_2 \in(0,1)$.) For $\lambda \in \mathbb{C}$ and indices $i,j$, we define the matrix $\mathcal{H}_{(i j)}^{\theta, \lambda}$ as
$$
\left(\mathcal{H}_{(i j)}^{\theta, \lambda}\right)_{k l}:= \begin{cases}\mathcal{H}_{i j}^\theta, & \text { if }\{k, l\} \neq\{i, j\}, \\ \lambda, & \text { if }\{k, l\}=\{i, j\} .\end{cases}
$$
Correspondingly, we define the resolvents
$$
\mathsf{G}^\theta_{\ell}:=\left(\mathcal{H}^\theta-\begin{pmatrix}
   0 & z_{\ell}   \\
  z_{\ell}^*&  0
    \end{pmatrix}-i\eta_{\ell} \right)^{-1} , \quad \mathsf{G}_{\ell, (ij)}^{\theta, \lambda}(z):=\left(\mathcal{H}_{(ij )}^{\theta, \lambda}-\begin{pmatrix}
   0 & z_{\ell}   \\
  z_{\ell}^*&  0
    \end{pmatrix}-i\eta_{\ell} \right)^{-1}  .
$$
\end{defn}
For any function $F: \mathbb{R}^{2 N \times 2 N} \rightarrow \mathbb{C}$, we have the basic interpolation formula
\begin{equation}\label{expH}
   \frac{\mathrm{d}}{\mathrm{d} \theta} \mathbb{E} F\left(\mathcal{H}^\theta\right)
   =\sum_{i,j} \left[\mathbb{E} F\left(\mathcal{H}_{(ij)}^{\theta, \mathcal{H}_{ij}^1}\right)-\mathbb{E} F\left(\mathcal{H}_{(ij)}^{\theta, \mathcal{H}_{ij}^0}\right)\right] 
\end{equation}
provided all the expectations exist. 
\begin{proof}[Proof of Lemma \ref{lemma:3.5momentreduced}]
We will use \eqref{expH} with the choice 
\begin{equation}\label{FIM}
    F\left(\mathcal{H}^\theta\right)=\prod_{l=1}^k\left\langle\mathrm{Im}\mathsf{G}_{\ell}^{\theta }-\mathrm{Im}m^{z_{\ell}} \right\rangle  
\end{equation}
We omit $\theta$ from the notation $m^{z_{\ell}}$ since the entries of the $\mathcal{H}^\theta$ have the same variances for all $\theta$. Hence $m_{\ell}^\theta$ is independent of $\theta$.  In the remaining, we will prove the following bound for our choice of $F$ in \eqref{FIM}, which implies Lemma \ref{lemma:3.5momentreduced}: for all $\theta \in[0,1]$ and $1\le i,j\le 2N$,  
\begin{equation}\label{FH2ed}
    \mathbb{E} F\left(\mathcal{H}_{(i,j)}^{\theta, \mathcal{H}_{ij}^1}\right)-\mathbb{E} F\left(\mathcal{H}_{(i,j)}^{\theta, \mathcal{H}_{ij}^0}\right) \prec N^{-2+ C\, \epsilon-\delta}.
\end{equation}
Combining \eqref{expH} and \eqref{FH2ed}, if we choose $\epsilon>0$ small enough, then we deduce Lemma \ref{lemma:3.5momentreduced}. In particular, are left to prove \eqref{FH2ed}. By Theorem 5.2 in \cite{AET21} (see also \cite{CES20}), we have the following for any $\tau>0$ with very high probability (i.e. probability at least $1-O(N^{-D})$ for any $D=O(1)$):
$$
\begin{gathered}
 \|\mathsf{G}_{\ell}^{\theta }-  m^{z_{\ell}}\|_{\max}\lesssim   N^{\epsilon+\tau}. 
\end{gathered}
$$
The $\max$-norm is sup-norm over entries. By resolvent expansion, for any $\lambda$, $\lambda'\in \mathbb R$, we have the following (in which we drop subscript $\ell$ for simplicity):
 \begin{equation}\label{Gexp}
 \mathsf{G}_{(i j)}^{\theta, \lambda^{\prime}}=\mathsf{G}_{(i j)}^{\theta, \lambda}   +\sum_{k=1}^K \mathsf{G}_{(i j)}^{\theta, \lambda}\left\{\left[ \left(\lambda-\lambda^{\prime}\right) \Delta_{i j}  \right] \mathsf{G}_{(i j)}^{\theta, \lambda}\right\}^k   +\mathsf{G}_{(i j)}^{\theta, \lambda^{\prime}}\left\{\left[ \left(\lambda-\lambda^{\prime}\right) \Delta_{i j} \right] \mathsf{G}_{(i j)}^{\theta, \lambda}\right\}^{K+1}.
\end{equation} 
Above, $\Delta_{i j}$  are $2 N \times 2 N$ matrices defined as 
$$\left(\Delta_{i j}\right)_{k l}:=\delta_{k i} \delta_{l j}+\delta_{k j} \delta_{l i}$$ Combining the above two statements and the trivial bound $\|\mathsf{G}_{\ell}\|\le (1/\eta_{\ell})$, we get
$$
\begin{gathered}
\left\|\mathsf{G} ^{\theta, 0 }_{l, (i,j)}-  m^{z_{\ell}}\right\|_{\max}\lesssim   N^{\epsilon+\tau} 
\end{gathered}
$$
for any $\tau>0$ with very high probability. Now, use \eqref{Gexp} with $\lambda=0$ and $\lambda'=\mathcal{H}_{ij}^{\gamma}$ and $K=7$. In doing so, we use $\E|\mathcal{H}_{ij}^{\gamma}|^{p}\lesssim_{p} N^{-p/2}$. We also apply the local law (Lemma \ref{1GE}) to control the Green's function entries. Ultimately, we get the following in which we again drop the subscript $\ell$:
\begin{align}
\left[ \mathsf{G}_{(i j)}^{\theta, H_{i j}^\gamma}\right]_{x y}=\left[\mathsf{G}_{(i j)}^{\theta, 0}\right]_{x y}+\sum_{k=1}^4 \mathcal{X}_{x y}^{(i j)}( \gamma, k)+\mathcal{E}, \quad 1\le x, y\le 2N,   \gamma \in\{0,1\}.\label{eq:gexpansion}
\end{align}
Above, $\mathcal{E}=O(N^{-5/2+\tau})$ with very high probability (for any fixed $\tau>0$), and
$$
\mathcal{X}^{(i j)}( \gamma, k):=\left(-\mathcal{H}_{i j}^\gamma\right)^k \mathsf{G}_{(i j)}^{\theta, 0}\left[\Delta_{i j}\mathsf{G}_{(i j)}^{\theta, 0}\right]^k, \quad \text { with } \quad  \left\|\mathcal{X}  (\gamma, k)\right\|_{\max} \prec N^{-k / 2+C \epsilon} .
$$
Notice that $\mathsf{G}_{(i j)}^{\theta, 0}$ is independent of $\mathcal{H}_{i j}^\gamma, \gamma \in\{0,1\}$. Moreover, note that 
\begin{align*}
\E\left|\left(\mathsf{G}_{(i j)}^{\theta, 0}\left[\Delta_{i j}\mathsf{G}_{(i j)}^{\theta, 0}\right]^k\right)_{\mathfrak{a}\mathfrak{b}}\right|\lesssim N^{C\epsilon+\tau}
\end{align*}
for any entry indices $\mathfrak{a},\mathfrak{b}$, since we have shown that the entries on the LHS are $\prec N^{C\epsilon+\tau}$ for any $\tau>0$, but they are also deterministically bounded by $O(\eta^{-D})$ for some $D=O(1)$, and $\eta\geq N^{-2}$.
Thus, with the expansion \eqref{eq:gexpansion}, we can write
$$\mathbb{E} F\left(\mathcal{H}_{(i j)}^{\theta, X_{i j}^\gamma}\right)
=\sum_s\sum_{k_1,k_2\cdots k_s}1\left(\sum_{t=1}^sk_t\le 4\right) \E[{\cal C}_{k_1,k_2\cdots k_s }] \prod_{t=1}^s \left(\mathbb E  (\mathcal{H}_{i j}^\gamma)^{k_t}\right)+O(N^{-5/2+C \epsilon})
$$
where $ {\cal C}_{k_1,k_2\cdots k_s }=O(N^{C\epsilon+\tau}) $ for any $\tau>0$, and ${\cal C}_{k_1,k_2\cdots k_s }$ does not depend on $\gamma$. We now use this for $\gamma=0,1$. The LHS of \eqref{FH2ed} is thus given by 
\begin{align*}
\sum_s\sum_{k_1,k_2\cdots k_s}1\left(\sum_{t=1}^sk_t\le 4\right) {\cal C}_{k_1,k_2\cdots k_s } \left[\prod_{t=1}^s \left(\mathbb E  (\mathcal{H}_{i j}^{1})^{k_t}\right)-\prod_{t=1}^s \left(\mathbb E  (\mathcal{H}_{i j}^{0})^{k_t}\right)\right]+O(N^{-5/2+C \epsilon}).
\end{align*}
By the moment matching condition, the first term above is $O(N^{-2+C\epsilon+\delta})$, so \eqref{FH2ed} follows. 
\end{proof}

\appendix
\section{Resolvent estimates for \texorpdfstring{$G_{a,b,\theta}(\eta)$}{G\_a b theta}}

For $\kappa>0$ let $\theta\in\left[\kappa, \frac{\pi}{2}-\kappa\right]$, $a\in (-1,1)$, $b\in[\kappa,1)$. Suppose $A \in M_N(\mathbb R)$ is a random matrix with i.i.d., centered, variance $1/N$ entries. Define
\[
G(\eta) = G_{a, b, \theta} (\eta) = 
\begin{pmatrix}
    -i\eta & I_2 \otimes A - \Lambda_{a, b, \theta} \otimes I_N \\
    I_2 \otimes A^T - \Lambda_{a, b, \theta}^T \otimes I_N & -i\eta
\end{pmatrix}^{-1},
\]
where
\[
\Lambda = \Lambda_{a, b, \theta} = \begin{pmatrix}
    a & b\tan\theta \\
    -\frac{b}{\tan\theta} & a
\end{pmatrix}
\]
and let $w = a+ ib$. We will also use the notation
\[
W = \begin{pmatrix}0 & I_2\otimes A\\ I_2\otimes A^T & 0 \end{pmatrix}, \quad Z = Z_{a,b,\theta}=\begin{pmatrix}
    0 &    \Lambda_{a, b, \theta} \otimes I_N \\
      \Lambda_{a, b, \theta}^T \otimes I_N & 0
\end{pmatrix}.
\]
In this notation $G_{a,b,\theta}(\eta) = (W - Z_{a,b,\theta} -i\eta)^{-1}$. We recall
\[
A_{a,b,\theta}:= I_2 \otimes A - \Lambda_{a, b, \theta} \otimes I_N
\]
Then
\[
G_{a, b, \theta}(\eta) = \begin{pmatrix}
     i\eta H_{a,b,\theta}  & H_{a,b,\theta} A_{a,b,\theta} \\
    A_{a,b,\theta}^T H_{a,b,\theta}  &  i\eta \tilde{H}_{a,b,\theta}\\
\end{pmatrix},
\]
where
\begin{align*}
\tilde{H}_{a,b,\theta}:=\tilde{H}_{a,b,\theta}(\eta) &= \left(A_{a,b,\theta}^T A_{a,b,\theta} + \eta^2\right)^{-1},\\
H_{a,b,\theta}:=H_{a,b,\theta}(\eta) &=
\left(A_{a,b,\theta} A_{a,b,\theta}^T  + \eta^2\right)^{-1}.
\end{align*}
Define a linear operator $\mathcal{S}:M_{4N}(\C)\rightarrow M_4(\C)$ as follows. Given a matrix $T\in M_{4N}(\C)$ consisting of $16$ blocks $T_{ij}\in M_N(\C)$ for $i,j\in[[1,4]]$, let
\[
\mathcal{S}(T) = 
\begin{pmatrix}
    \langle T_{33} \rangle & \langle T_{34} \rangle & 0 & 0 \\
    \langle T_{43} \rangle & \langle T_{44} \rangle & 0 & 0 \\
    0 & 0 & \langle T_{11} \rangle & \langle T_{12} \rangle \\
    0 & 0 & \langle T_{21} \rangle & \langle T_{22} \rangle \\
\end{pmatrix}.
\]

In this Appendix we use the standard technique of cumulant expansion to derive the deterministic approximations of the resolvent $G_{a,b,\theta}$ and products of two resolvents $G_{a,b,\theta} T G_{a,b,\theta}$ for some deterministic matrices $T$. See e.g. \cite{CES2020} for a similar argument for the resolvent of Wigner ensemble. We are not concerned with the optimality of the error bounds is $\eta$ since we only apply these estimates with $\eta=N^{-\epsilon}$. To state the estimates we use the notion of stochastic domination defined here.
\begin{defn}[Stochastic domination]\label{defn:prec}
Suppose that $X = \left\{X_N(s) : N\in \zz_+, s\in S_N\right\}$ and $Y = \left\{Y_N(s) : N\in \zz_+, s\in S_N\right\}$ are sequences of random variables, possibly parametrized by $s$. We say that $X$ is stochastically dominated by $Y$ uniformly in $s$ and write $X\prec Y$ or $X = \oh(Y)$ if for any $\varepsilon, D>0$ we have
\[
\sup_{S_N} \pp\left(X_N(s) > N^{\varepsilon} Y_N(s)\right) < N^{-D}
\]
for large enough $N$.
\end{defn}

\begin{lemma}\label{1GE}
For $\eta > N^{-1/2+\varepsilon}$, and any $\alpha, \beta\in[[1,4]]$, 
\[
\left|\langle\left(G_{a,b,\theta}(\eta) - M_{a,b,\theta}(\eta)\right) E_{\alpha,\beta}\rangle\right| \prec \frac{1}{N\eta^2}
\]
uniformly in $a,b,\theta$, where $E_{\alpha,\beta} =e_{\alpha,\beta}\otimes I_N \in M_{4N}(\R)$ and 
\[
e_{\alpha,\beta}\in M_4(\R),\qquad  (e_{\alpha,\beta})_{\alpha',\beta'}=\delta_{\alpha'=\alpha}\delta_{\beta'=\beta}
\]
and $M(\eta) = M_{a,b,\theta}(\eta) \in M_{4N}(\C)$ is the solution to the matrix Dyson equation (MDE)
\begin{equation}\label{MDE}
    \left[i\eta+Z+\mathcal{S}\left(M(\eta)\right) \right] M(\eta) +I=0,
\end{equation}
satisfying $\eta\im M(\eta)>0$.
\end{lemma}

The existence and uniqueness of the solution to matrix Dyson equation \eqref{MDE} satisfying $\eta\im M(\eta)>0$ was shown in \cite{helton2007operator}. Before we prove this lemma we introduce the notation used to control the error terms in the proof.

\begin{defn}[Renormalized term]
For any smooth function $f:M_{2N}(\R)\rightarrow M_{2N}(\R)$ define
\[
\underline{Wf(W)} = Wf(W) - \mathbb  E_{\widetilde{W}} \widetilde{W} \left(\partial_{\widetilde{W}} f\right)(W),
\]
where $\widetilde{W}$ is an independent copy of $W$. 
\end{defn}

Note that $\mathbb E_{\widetilde{W}} \widetilde{W} \left(\partial_{\widetilde{W}} f\right)(W)$ is the first order term in the cumulant expansion of $\mathbb E Wf(W)$ with respect to $W$. The following lemma lets us control the renormalized terms. The proof is deferred until the end of this Appendix.

\begin{lemma}[Renormalized term bounds]\label{renorm-term-bound}
For $\eta\ge N^{-\frac12+\varepsilon}$ and any $\alpha, \beta, \alpha', \beta'\in[[1,4]]$, we have
\begin{align}
\left|\ntr{\underline{WG}E_{\alpha,\beta}}\right| &\prec \frac{1}{N\eta^2},\label{underline-1G}\\ 
\left|\ntr{\underline{WGE_{\alpha,\beta}G}E_{\alpha',\beta'}}\right| &\prec \frac{1}{N\eta^3}\label{underline-2G}
\end{align}
uniformly in $a,b,\theta$.
\end{lemma}

\begin{proof}[Proof of Lemma \ref{1GE}]
From the definition of renormalized term we see that
\[
WG = -\mathcal{S}(G)G-\frac{1}{N}\mathcal{T}(G)G + \underline{WG}.
\]
Then
\[
-I = -(W-Z-i\eta)G = \left(Z+i\eta+\mathcal{S}(G)\right)G + \frac{1}{N}\mathcal{T}(G)G-\underline{WG},
\]
where
\[
\mathcal{T}(T) = \begin{pmatrix}
    0 & 0 & T_{13} & T_{23} \\
    0 & 0 & T_{14} & T_{24} \\
    T_{31} & T_{41} & 0 & 0 \\
    T_{32} & T_{42} & 0 & 0
\end{pmatrix}.
\]
We take a normalized trace separately for every $N\times N$ block in this equation to get
\[
-\frac14\delta_{\alpha\beta} = \sum_{\gamma=1}^4\left(Z+i\eta+\mathcal{S}(G)\right)_{\alpha\gamma}\ntr{GE_{\gamma\beta}} + \frac{1}{N}\ntr{\mathcal{T}(G)GE_{\alpha\beta}} + \ntr{\underline{WG}E_{\alpha\beta}}.
\]
Note that $\frac{1}{N}\ntr{\mathcal{T}(G)GE_{\alpha\beta}} \prec \frac{1}{N\eta^2}$ by a trivial bound and $\ntr{\underline{WG}E_{\alpha\beta}} \prec \frac{1}{N\eta^2}$ by Lemma \ref{renorm-term-bound}. Then
\[
-\frac14\delta_{\alpha\beta} = \sum_{\gamma=1}^4\left(Z+i\eta+\mathcal{S}(G)\right)_{\alpha\gamma}\ntr{GE_{\gamma\beta}} + \oh\left(\frac{1}{N\eta^2}\right)
\]
and the result of Lemma \ref{1GE} follows by stability of the solution of MDE \eqref{MDE}.
\end{proof}

\begin{lemma}\label{2GE}
For $\eta > N^{-\frac16+\varepsilon}$, and any $\alpha, \beta, \alpha', \beta'\in[[1,4]]$,
\[
\ntr{G_{a,b,\theta} E_{\alpha\beta}G_{a,b,\theta} E_{\alpha'\beta'}}= \ntr{M_{a,b,\theta} E_{\alpha\beta}M_{a,b,\theta} \mathcal{X}^{-1}(E_{\alpha'\beta'})}+\oh\left(\frac{1}{N\eta^6}\right)
\]
uniformly in $a,b,\theta$, where operator linear $\mathcal{X} = \mathcal{X}_{a,b,\theta}$ acting on $M_4(\R)$ by
\[
\mathcal{X}_{a,b,\theta}(\cdot) = I - \mathcal{S}(M_{a,b,\theta} \cdot M_{a,b,\theta}).
\]
\end{lemma}

\begin{proof}
From the definition of $G$ and \eqref{MDE} we see that
\begin{equation}\label{IBP-for-G}
G = M + M\mathcal{S}(G - M) G - M\underline{WG} + \frac{1}{N} \mathcal{T}(G)G,
\end{equation}

Now we consider $GE_{\alpha,\beta}G$. The integration by parts formula gives
\begin{align}\label{2G-IBP}
G E_{\alpha,\beta} G &= M E_{\alpha,\beta} M + M \mathcal{S}(G E_{\alpha,\beta}G) M + M E_{\alpha,\beta} (G - M) \\ 
&+ M \mathcal{S}(G - M) G E_{\alpha,\beta} G + M \mathcal{S}(G E_{\alpha,\beta} G) (G - M)\nonumber \\
&+ \frac{1}{N} M \mathcal{T}(G) G E_{\alpha,\beta} G + \frac{1}{N} \mathcal{T}(G E_{\alpha,\beta} G)G - M \underline{WG E_{\alpha,\beta} G}.\nonumber
\end{align}
Define a linear operator $\mathcal{B}$ acting on $M_{4N}(\C)$ by
\[
\mathcal{B}(\cdot) = I - M \mathcal{S}(\cdot) M.
\]
We move the second term of \eqref{2G-IBP} to the left, apply $\mathcal{B}^{-1}$ on both sides, multiply by $E_{\alpha',\beta'}$ on the right and take the normalized trace to get
\begin{align*}
\ntr{G E_{\alpha,\beta} G E_{\alpha',\beta'}} &= \ntr{\mathcal{B}^{-1}(M E_{\alpha,\beta} M) E_{\alpha',\beta'}} \\
&+ \ntr{\mathcal{B}^{-1}(M E_{\alpha,\beta} (G - M)) E_{\alpha',\beta'}} \\
&+ \ntr{\mathcal{B}^{-1}(M \mathcal{S}(G - M) G E_{\alpha,\beta} G) E_{\alpha',\beta'}} \\
&+ \ntr{\mathcal{B}^{-1}(M \mathcal{S}(G E_{\alpha,\beta} G) (G - M)) E_{\alpha',\beta'}} \\
&+ \frac{1}{N}\ntr{\mathcal{B}^{-1}(M \mathcal{T}(G) G E_{\alpha,\beta} G) E_{\alpha',\beta'}} \\
&+ \frac{1}{N} \ntr{\mathcal{B}^{-1}(\mathcal{T}(G E_{\alpha,\beta} G)G) E_{\alpha',\beta'}} \\
&- \ntr{\mathcal{B}^{-1}(M \underline{WG E_{\alpha,\beta} G}) E_{\alpha',\beta'}}.
\end{align*}
Now we notice that operators $\mathcal{B}$ and $\mathcal{X}$ are related thorough the following identity. For any $B_1, B_2\in M_{4N}(\C)$, we have
\[
\ntr{\mathcal{B}^{-1}(B_1)B_2} = \ntr{B_1\mathcal{X}^{-1}(B_2)}.
\]
Thus, we have
\begin{align*}
\ntr{G E_{\alpha,\beta} G E_{\alpha',\beta'}} &= \ntr{M E_{\alpha,\beta} M \mathcal{X}^{-1}(E_{\alpha',\beta'})} \\
&+ \ntr{M E_{\alpha,\beta} (G - M) \mathcal{X}^{-1}(E_{\alpha',\beta'})} \\
&+ \ntr{M \mathcal{S}(G - M) G E_{\alpha,\beta} G \mathcal{X}^{-1}(E_{\alpha',\beta'})} \\
&+ \ntr{M \mathcal{S}(G E_{\alpha,\beta} G) (G - M) \mathcal{X}^{-1}(E_{\alpha',\beta'})} \\
&+ \frac{1}{N}\ntr{M \mathcal{T}(G) G E_{\alpha,\beta} G \mathcal{X}^{-1}(E_{\alpha',\beta'})} \\
&+ \frac{1}{N} \ntr{\mathcal{T}(G E_{\alpha,\beta} G)G \mathcal{X}^{-1}(E_{\alpha',\beta'})} \\
&- \ntr{M \underline{WG E_{\alpha,\beta} G} \mathcal{X}^{-1}(E_{\alpha',\beta'})} 
\end{align*}
To bound the error terms we use that $\|\mathcal{S}\|_{op}\lesssim 1$, $\|\mathcal{T}\|_{op}\lesssim 1$ and $\|\mathcal{X}\|_{op}\gtrsim \eta$ (see section \ref{appx:operatorX}). Note that in the 2nd and 4th terms $G-M$ appears and it is multiplied by a $4\times 4$ block matrix. In the 3rd term $G-M$ appears inside of the operator $\mathcal{S}$. Thus every instance of $G-M$ on the right can be bound using Lemma \ref{1GE} by $\frac{1}{N\eta^2}$. Every other instance of $G$ or $M$ on the right we control using the operator norm bound $\|G\|, \|M\|\prec \frac{1}{\eta}$. In the final term we use Lemma \ref{renorm-term-bound} for the renormalized term. Putting this together we get
\begin{align*}
&\ntr{M E_{\alpha,\beta} (G - M) \mathcal{X}^{-1}(E_{\alpha',\beta'})} \prec \frac{1}{N\eta^4},\\
&\ntr{M \mathcal{S}(G - M) G E_{\alpha,\beta} G \mathcal{X}^{-1}(E_{\alpha',\beta'})} \prec \frac{1}{N\eta^6}, \\
&\ntr{M \mathcal{S}(G E_{\alpha,\beta} G) (G - M) \mathcal{X}^{-1}(E_{\alpha',\beta'})} \prec \frac{1}{N\eta^6}, \\
&\frac{1}{N}\ntr{M \mathcal{T}(G) G E_{\alpha,\beta} G \mathcal{X}^{-1}(E_{\alpha',\beta'})}\prec \frac{1}{N\eta^5}, \\
&\frac{1}{N} \ntr{\mathcal{T}(G E_{\alpha,\beta} G)G \mathcal{X}^{-1}(E_{\alpha',\beta'})} \prec \frac{1}{N\eta^4},\\
&\ntr{M \underline{WG E_{\alpha,\beta} G} \mathcal{X}^{-1}(E_{\alpha',\beta'})}\prec \frac{1}{N\eta^5}.
\end{align*}
This concludes the proof.
\end{proof}

\begin{proof}[Proof of Lemma \ref{renorm-term-bound}]
We show the proof of \eqref{underline-1G} here. The proof of \eqref{underline-2G} is analogous. Every extra $G$ in the trace gives rise to an additional $\frac{1}{\eta}$ in the bound.

The proof of \eqref{underline-1G} follows closely the proof of Theorem 4.1 of \cite{CES2020}, so we outline the differences. Similarly to \cite{CES2020}, we use the cumulant expansion
\[
\E A_{ij} f(W) = \sum_{m\ge 1} \frac{\kappa_{m+1}}{m! N^{\frac12(m+1)}} \E\partial^m_{A_{ij}} f(W),
\]
where $\kappa_{m}$ is $m$th cumulant of $\sqrt{N}A_{ij}$. 

We introduce matrices $\Delta^{ij}\in M_{4N}(\R)$ for $i,j\in[[1,N]]$, such that 
\[
(\Delta^{ij})_{xy} = \delta_{x=2N+i}\delta_{y=j}+\delta_{x=3N+i}\delta_{y=N+j}+\delta_{x=j}\delta_{y=2N+i}+\delta_{x=N+j}\delta_{y=3N+i}.
\]
Then $W = \sum_{i,j = 1}^N A_{ij} \Delta^{ij}$. Consider the second moment of $\ntr{\underline{WG}E_{\alpha,\beta}}$ and use the cumulant expansion with respect to both $W$. The first derivative terms of of the second moment are
\begin{align*}
&\E|\ntr{\underline{WG}E_{\alpha,\beta}}|^2 = \sum_{i,j=1}^N \frac{\kappa_2}{N} \E\ntr{\Delta^{ij}GE_{\alpha,\beta}}\ntr{E_{\alpha,\beta}^\ast G^\ast \Delta^{ij}} \\
&+ \sum_{i,j=1}^N \sum_{i',j'=1}^N \frac{\kappa^2_2}{N^2} \E\ntr{\Delta^{ij}G\Delta^{i'j'}G E_{\alpha,\beta}}\ntr{E_{\alpha,\beta}^\ast G^\ast\Delta^{ij}G^\ast\Delta^{i'j'}} + \ldots 
\end{align*}
The first term can be written as a sum of $16$ terms of the type 
\[
\frac{\kappa_2}{N^2}\E\ntr{G^{(\ast)}B_1G^{(\ast)}B_2},
\]
where $\|B_1\|, \|B_2\|\le 1$ and $G^{(\ast)}$ is indicating $G$ or $G^\ast$. Since we do not need to control \eqref{underline-1G} by $\frac{1}{N\eta}$, we bound $\ntr{GB_1GB_2}$ trivially by $\frac{1}{\eta^2}$. Similarly, the second term can be written as a sum of $16^2$ terms of three types
\begin{align*}
&\frac{\kappa_2^2}{N^3}\ntr{G^{(\ast)}B_1G^{(\ast)}B_2G^{(\ast)}B_3G^{(\ast)}B_4}\\
&\frac{\kappa_2^2}{N^2}\ntr{G^{(\ast)}B_1}\ntr{G^{(\ast)}B_2G^{(\ast)}B_3G^{(\ast)}B_4}\\
&\frac{\kappa_2^2}{N^2}\ntr{G^{(\ast)}B_1G^{(\ast)}B_2}\ntr{G^{(\ast)}B_3G^{(\ast)}B_4}
\end{align*}
All these terms can be bound by $\frac{\kappa_2^2}{N^2\eta^4}$. Similarly for higher order terms in the cumulant expansion each extra derivative adds another $G$ into the term, which we bound by $\frac{1}{\eta}$ and another $\frac{1}{\sqrt{N}}$ due to taking a higher order cumulant of the entry of $A$. Analogous calculation works for $2p$-th moment. We refer the reader to \cite{CES2020} for details.
\end{proof}

\subsection{Properties of the operator \texorpdfstring{$\mathcal{X}$}{X}.}\label{appx:operatorX}

Note that matrix Dyson equation \eqref{MDE} invariant under transposition of $M$ and under conjugation of $M$ by a permutation matrix
\[
\Pi = \begin{pmatrix}
    0 & 0 & 0 & 1\\
    0 & 0 & 1 & 0\\
    0 & 1 & 0 & 0\\
    1 & 0 & 0 & 0
\end{pmatrix}.
\]
The solution to MDE satisfying the condition $\eta \im M(\eta)>0$ is unique by Theorem 2.1 of \cite{helton2007operator}. Thus for $\eta>0$ the solution is
\[
M=\begin{pmatrix}
    M_{11} &   M_{11} \Lambda ( i\eta + M_{11})^{-1} \\
    (i\eta + M_{11})^{-1} \Lambda^T M_{11} & F M_{11} F
\end{pmatrix},
\]
where
\[
F = \begin{pmatrix}0&1\\1&0\end{pmatrix}
\]
and $M_{11} \in M_2(\mathbb C)$ is satisfies $iM_{11} < 0$ and 
\begin{align*}
    - M_{11}^{-1} = i\eta + FM_{11}F - \Lambda(i\eta + M_{11})^{-1} \Lambda^T 
\end{align*}
By a trivial computation we can derive the expansion of the solution in $\eta$.
\begin{align*}
M_{11}(\eta) &= i\sqrt{1-a^2-b^2}\begin{pmatrix} \tan^{-1}\theta & 0\\ 0 & \tan\theta \end{pmatrix} + i\eta S + O(\eta^2),
\end{align*}
where $S\in M_2^{sa}(\mathbb R)$ has entries 
\begin{align*}
S_{11} &= -1 + \frac{1 + \tan^{-2}\theta}{4(1-a^2-b^2)} + \frac{(1-a^2)(1-\tan^{-2}\theta)}{4b^2} \\
S_{22} &= -1 + \frac{1 + \tan^2\theta}{4(1-a^2-b^2)} + \frac{(1-a^2)(1-\tan^2\theta)}{4b^2} \\
S_{12} &= S_{21} = -\frac{a}{2b}(\tan\theta - \tan^{-1}\theta).
\end{align*}

Recall that the operator $\mathcal{X}$ acts on $M_4(\C)$ by $\mathcal{X}(\cdot) = I - \mathcal{S}(M(\eta)\cdot M(\eta))$ and $\mathcal{X}^\ast(\cdot) = I - M(-\eta)\mathcal{S}(\cdot)M(-\eta)$. Using the formulas for $M(\eta)$ above, it is easy to verify the following.
\begin{itemize}
    \item Operator $\mathcal{X}$ has an eigenvalue $1$ with multiplicity $8$. In particular, the left eigenspace of $\mathcal{X}$ corresponding to this eigenvalue is the space of $2\times 2$ block matrices with $0$ diagonal blocks. 

    \item Define 
    \[
    \beta_+ = \frac{1}{2\sqrt{1-|w|^2}}(\tan\theta+\tan^{-1}\theta)\eta,\quad \beta_- = 2(1-a^2-b^2)
    \]
    and
    \[
    \gamma_\pm = 1-a^2+b^2\pm \sqrt{(1 - a^2 + b^2)^2 - 4b^2}.
    \]
    Operator $\mathcal{X}$ has two eigenvalues of the form $\beta_-+O(\eta)$, two eigenvalues of the form $\beta_+ +O(\eta)$, two eigenvalues of the form $\gamma_-+O(\eta)$, two eigenvalues of the form $\gamma_+ +O(\eta)$.

    \item As a consequence, we have $\|\mathcal{X}\|_{op} \gtrsim \eta$.
\end{itemize}

\section{Computation of resolvent quantities at \texorpdfstring{$\theta=\frac{\pi}{4}$}{theta=pi/4}.}\label{section:resolventsatpi/4}

In this section we compute certain functions of $G_{a,b,\frac{\pi}{4}}(\eta)$ explicitly based on the observation that at $\theta = \frac{\pi}{4}$ this resolvent reduces to the regular resolvent $G_w(\eta)$ as follows. It is straightforward to check that for $w=a+ib$
\begin{equation*}
G_{a,b,\frac{\pi}{4}}(\eta) = U^\ast\begin{pmatrix}G_w(\eta)&0\\0&G_{\overline{w}}(\eta)\end{pmatrix}U,
\end{equation*}
where
\begin{equation}\label{3vec-at-z}
U = \frac{1}{\sqrt{2}}\begin{pmatrix}
    1 & -i & 0 & 0\\
    0 & 0 & 1 & -i\\
    1 & i & 0 & 0\\
    0 & 0 & 1 & i
\end{pmatrix}
\end{equation}
Note also that
\begin{align*}
UE_{3,3}U^\ast &= \frac{1}{2}\left(E_{2,2}+E_{2,4}+E_{4,2}+E_{4,4}\right)\\
UE_{4,4}U^\ast &= \frac{1}{2}\left(E_{2,2}-E_{2,4}-E_{4,2}+E_{4,4}\right)\\
UE_{3,4}U^\ast &= \frac{i}{2}\left(E_{2,2}-E_{2,4}+E_{4,2}-E_{4,4}\right)\\
UE_{4,3}U^\ast &= \frac{i}{2}\left(-E_{2,2}-E_{2,4}+E_{4,2}+E_{4,4}\right)\\
\end{align*}
Suppose $z=a+ib$ and $\eta_{z,t}$ is defined by $t\ntr{H_z(\eta_{z,t})} = 1$. We show here that
\[
\begin{pmatrix}\frac{N}{t}-\tr \tilde{H}_{a,b,\frac{\pi}{4},11}\\-2\tr \tilde{H}_{a,b,\frac{\pi}{4},12}\\\frac{N}{t}-\tr \tilde{H}_{a,b,\frac{\pi}{4},22}\end{pmatrix} = 0.
\]
Indeed,
\begin{align*}
\tr \tilde{H}_{a,b,\frac{\pi}{4},11} &= \frac{1}{i\eta_{z,t}}\tr G_{a,b,\frac{\pi}{4}} E_{3,3} = \frac{1}{i\eta_{z,t}}\tr \begin{pmatrix}G_w(\eta)&0\\0&G_{\overline{w}}(\eta)\end{pmatrix} U E_{3,3} U^\ast \\
&= \frac{1}{2i\eta_{z,t}}\tr \begin{pmatrix}G_z(\eta)&0\\0&G_{\overline{z}}(\eta)\end{pmatrix} \left(E_{2,2}+E_{2,4}+E_{4,2}+E_{4,4}\right) \\
&= \frac12\tr \tilde{H}_z + \frac12\tr \tilde{H}_{\overline{z}} = \frac{N}{t}.
\end{align*}
Similarly, $\tr \tilde{H}_{a,b,\frac{\pi}{4},12} = 0$ and $\tr \tilde{H}_{a,b,\frac{\pi}{4},22} = \frac{N}{t}$.

\subsection{Computing \texorpdfstring{$\mathbf{Q}_{a,b,\frac{\pi}{4}}$}{Q\_a b pi/4}.}\label{subsection:Qestimatepi/4}

\begin{lemma}
For $a\in[-1,1]$, $b\in[\kappa, 1)$ and $\eta > N^{-\delta}$ we have $\mathbf{Q}_{a,b,\frac{\pi}{4}} \gtrsim \frac{N}{\eta^2}$ uniformly in $a,b$.
\end{lemma}

\begin{proof}
Note that as a quadratic form on $M^{sa}_2(\R)$,
\[
\mathbf{Q}_{a,b,\frac{\pi}{4}} = -\frac{4N}{\eta^2}\ntr{\left(G_{a,b,\frac{\pi}{4}}\begin{pmatrix}0&0\\0&P\end{pmatrix}\right)^2}
\]
Thus, we can compute the entries of $Q$ in the basis
\[
\begin{pmatrix}1&0\\0&0\end{pmatrix},\quad \frac12\begin{pmatrix}0&1\\1&0\end{pmatrix},\quad \begin{pmatrix}0&0\\0&1\end{pmatrix}.
\]
We get
\[
\mathbf{Q}_{a,b,\frac{\pi}{4}} = \frac{N}{4}\begin{pmatrix}
    \ntr{\tilde{H}_z^2}+\ntr{\tilde{H}_z\tilde{H}_{\bar{z}}} & 0 & \ntr{\tilde{H}_z^2}-\ntr{\tilde{H}_z\tilde{H}_{\bar{z}}}\\
    0 & 4\ntr{\tilde{H}_z \tilde{H}_{\bar{z}}} & 0\\
    \ntr{\tilde{H}_z^2}-\ntr{\tilde{H}_z\tilde{H}_{\bar{z}}} & 0 & \ntr{\tilde{H}_z^2}+\ntr{\tilde{H}_z\tilde{H}_{\bar{z}}}\\
\end{pmatrix}.
\]
This matrix has eigenvalues $2N\ntr{\tilde{H}_z^2}, 2N\ntr{\tilde{H}_z\tilde{H}_{\bar{z}}}, N\ntr{\tilde{H}_z\tilde{H}_{\bar{z}}}$, all of which are bounded below by $N\eta^{-2}$ by \cite{cipolloni2022optimal}.
\end{proof}

\section{Jacobian calculation}\label{section:jacobian}

Let us recall the Schur decomposition construction. Define the manifold 
\[
\Omega = \R \times \R_+ \times [0, \pi/2) \times V^2(\R^N) \times M_{(N-2)\times 2}(\R),
\]
where $V^2(\R^N)$ is a Stiefel manifold, i.e.
\[
V^2(\R^N) = \mathbf{O}(N)/\mathbf{O}(N-2) = \{(v_1,v_2) \in \mathbb{S}^{N-1}\times\mathbb{S}^{N-1}:\,v_1^T v_2 = 0\}.
\]
Choose a smooth map $R: V^2(\R^N) \rightarrow \mathbf{O}(N)$ such that for any $\mathbf{v} = (v_1, v_2) \in V^2(\R^N)$  we have $R(\mathbf{v})\mathbf{e}_i = v_i$ for $i = 1, 2$. For any $(a, b, \theta) \in \R \times \R_+ \times [0, \pi/2)$ define 
\[
\Lambda_{a, b, \theta} = \begin{pmatrix}
    a & b\tan\theta \\
    -\frac{b}{\tan\theta} & a
\end{pmatrix}.
\]

Define a map
\[
\Phi: \Omega \times M_{N-2}(\R) \rightarrow M_N(\R)
\]
such that
\[
\Phi(a, b, \theta, \mathbf{v}, W, M^{(1)}) = M = R(\mathbf{v}) \begin{pmatrix}
    \Lambda_{a, b, \theta} & W^T \\
    0 & M^{(1)}
\end{pmatrix} R(\mathbf{v})^T.
\]
Note that if $\mathbf{v} = (v_1, v_2)$, then $\cos\theta v_1 \pm i\sin\theta v_2$ are eigenvectors of the right-hand side with corresponding eigenvalues $\lambda, \overline{\lambda} = a \pm bi$.

\begin{lemma}
The Jacobian of $\Phi$ is 
\begin{align}
    J(\Phi) &= 16b^2\frac{|\cos 2\theta|}{\sin^2 2\theta} \left|\det\left(M^{(1)} - \lambda\right) \right|^2 \label{real-non-sym-jacobian}
\end{align}
\end{lemma}
 
\begin{proof}
Consider a smooth atlas on $V^2(\R)$ and let $\varphi: U\rightarrow \tilde{U}\subset \R^{2N-3}$, where $U\subset V^2(\R)$, be a chart in this atlas. Denote the standard coordinates in $\R^{2N-3}$ by $(u^1,\ldots, u^{2N-3})$. By abusing notation, in the the rest of the proof we view $v_i$ as 
\[
v_i = v_i(u^1,\ldots, u^{2N-3}) = \left(\varphi^{-1}(u^1,\ldots, u^{2N-3})\right)_{i}
\]
for $i=1,2$. Similarly, we will view $R(\mathbf{v})$ as a function from $\tilde{U}$ to $M_N(\R)$. Note that 
\[
J(\Phi) = J(\Phi\circ \varphi^{-1}) J(\varphi).
\]
We start by computing $J(\Phi\circ \varphi^{-1})$. By differentiating the definition of $\Phi$, we get
\begin{align*}
dM &= dR(\mathbf{v}) \begin{pmatrix}
    \Lambda_{a, b, \theta} & W^T \\
    0 & M^{(1)}
\end{pmatrix} R(\mathbf{v})^T + R(\mathbf{v}) \begin{pmatrix}
    \Lambda_{a, b, \theta} & W^T \\
    0 & M^{(1)}
\end{pmatrix} dR(\mathbf{v})^T \\
&+ R(\mathbf{v}) \begin{pmatrix}
    d\Lambda_{a, b, \theta} & dW^T \\
    0 & dM^{(1)}
\end{pmatrix} R(\mathbf{v})^T.
\end{align*}
Consider $d\tilde{M} = R(\mathbf{v})^T dM R(\mathbf{v})$. Since $R(\mathbf{v})$ is orthogonal, the volume form on $M_N(\R)$ can be expressed in terms of $d\tilde{M}$ as $\bigwedge_{i=1}^N\bigwedge_{j=1}^N dM_{ij} = \bigwedge_{i=1}^N\bigwedge_{j=1}^N d\tilde{M}_{ij}$. Note that
\begin{align*}
d\tilde{M} &= R(\mathbf{v})^T dR(\mathbf{v}) \begin{pmatrix}
    \Lambda_{a, b, \theta} & W^T \\
    0 & M^{(1)}
\end{pmatrix} + \begin{pmatrix}
    \Lambda_{a, b, \theta} & W^T \\
    0 & M^{(1)}
\end{pmatrix} dR(\mathbf{v})^T R(\mathbf{v})\\
&+ \begin{pmatrix}
    d\Lambda_{a, b, \theta} & dW^T \\
    0 & dM^{(1)}
\end{pmatrix}.
\end{align*}
Since $R(\mathbf{v})\mathbf{e}_i = v_i$ for $i=1,2$, the first two columns of $R(\mathbf{v})$ are $v_1$ and $v_2$. Denote the rest of the matrix $R(\mathbf{v})$ by $\tilde{R}(\mathbf{v})$. Since $R(\mathbf{v})$ is orthogonal, $\tilde{R}(\mathbf{v})^T v_i = 0$ for $i=1,2$ and, thus, $d\tilde{R}(\mathbf{v})^T v_i = -\tilde{R}(\mathbf{v})^T dv_i$. Then
\[
R(\mathbf{v})^T dR(\mathbf{v}) = -dR(\mathbf{v})^T R(\mathbf{v}) = \begin{pmatrix}
df E & -dH^T\\
dH & \tilde{R}(\mathbf{v})^T d\tilde{R}(\mathbf{v})
\end{pmatrix},
\]
where
\begin{align*}
df &= v_2^T dv_1,\\
E &= \begin{pmatrix}0&1\\-1&0\end{pmatrix}\\
dH &= (dH_1, dH_2) = (\tilde{R}(\mathbf{v})^T dv_1, \tilde{R}(\mathbf{v})^T dv_2).
\end{align*}
Plugging this into the expression for $d\tilde{M}$, we get
\[
d\tilde{M} = \begin{pmatrix}
     d\tilde{M}^{(11)} & d\tilde{M}^{(12)} \\
     d\tilde{M}^{(21)} & d\tilde{M}^{(22)}
\end{pmatrix},
\]
where
\begin{align*}
d\tilde{M}^{(11)} &= d\Lambda_{a,b,\theta}+df(E\Lambda_{a,b,\theta}-\Lambda_{a,b,\theta}E)-W^TdH,\\
d\tilde{M}^{(12)} &= dW^T + df EW^T-dH^T M^{(1)}+ \Lambda_{a,b,\theta}dH^T - W^{T}\tilde{R}(\mathbf{v})^Td\tilde{R}(\mathbf{v}),\\
d\tilde{M}^{(21)} &= dH \Lambda_{a,b,\theta} - M^{(1)}dH, \\
d\tilde{M}^{(22)} &= dM^{(1)} + dH W^T + \tilde{R}(\mathbf{v})^Td\tilde{R}(\mathbf{v}) M^{(1)} - M^{(1)} \tilde{R}(\mathbf{v})^Td\tilde{R}(\mathbf{v}).
\end{align*}
Now we compute the volume form by changing the order and noticing that $dW$ and $dM^{(1)}$ appear only in $d\tilde{M}^{(12)}$ and $d\tilde{M}^{(22)}$ respectively.
\begin{align*}
\bigwedge_{i=1}^N\bigwedge_{j=1}^N d\tilde{M}_{ij} &=  \bigwedge_{i=1}^2\bigwedge_{j=1}^2 d\tilde{M}^{(11)}_{ij} \wedge \bigwedge_{i=1}^{N-2}\bigwedge_{j=1}^2 d\tilde{M}^{(21)}_{ij} \wedge \bigwedge_{i=1}^2\bigwedge_{j=1}^{N-2} d\tilde{M}^{(12)}_{ij} \wedge \bigwedge_{i=1}^{N-2}\bigwedge_{j=1}^{N-2} d\tilde{M}^{(22)}_{ij} \\
&= \bigwedge_{i=1}^2\bigwedge_{j=1}^2 d\tilde{M}^{(11)}_{ij} \wedge \bigwedge_{i=1}^{N-2}\bigwedge_{j=1}^2 d\tilde{M}^{(21)}_{ij} \wedge \bigwedge_{i=1}^2\bigwedge_{j=1}^{N-2} dW_{ji} \wedge \bigwedge_{i=1}^{N-2}\bigwedge_{j=1}^{N-2} dM^{(1)}_{ij}.
\end{align*}
Notice that 
\[
E\Lambda_{a,b,\theta}-\Lambda_{a,b,\theta}E = b(\tan\theta-\tan^{-1}\theta)\begin{pmatrix}1&0\\0&-1\end{pmatrix}
\]
and
\[
d\Lambda_{a,b,\theta} = \begin{pmatrix} da & \tan\theta db + \frac{b}{\cos^2\theta}d\theta\\ -\tan^{-1}\theta db + \frac{b}{\sin^2\theta}d\theta & da\end{pmatrix}.
\]
Thus we can simplify the first two terms of the volume form:
\begin{align*}
&\bigwedge_{i=1}^2\bigwedge_{j=1}^2 d\tilde{M}^{(11)}_{ij} \wedge \bigwedge_{i=1}^{N-2}\bigwedge_{j=1}^2 d\tilde{M}^{(21)}_{ij} = (da + b(\tan\theta-\tan^{-1}\theta)df)\\
&\wedge \left(\tan\theta db + \frac{b}{\cos^2\theta}d\theta\right)\wedge (da - b(\tan\theta-\tan^{-1}\theta)df)\\
&\wedge \left(-\tan^{-1}\theta db + \frac{b}{\sin^2\theta}d\theta\right)\wedge \bigwedge_{i=1}^{N-2}\bigwedge_{j=1}^2 d\tilde{M}^{(21)}_{ij}\\
&= 4b^2\frac{\tan\theta-\tan^{-1}\theta}{\sin\theta\cos\theta} da\wedge db\wedge d\theta\wedge df\wedge \bigwedge_{i=1}^{N-2} \bigwedge_{j=1}^2 d\tilde{M}^{(21)}_{ij}\\
&= -16b^2\frac{\cos 2\theta}{\sin^2 2\theta} da\wedge db\wedge d\theta\wedge df\wedge \bigwedge_{i=1}^{N-2} \bigwedge_{j=1}^2 d\tilde{M}^{(21)}_{ij}.
\end{align*}
To simplify $d\tilde{M}^{(22)}$ consider its two columns $d\tilde{M}^{(22)}_1$ and $d\tilde{M}^{(22)}_2$. Then
\[
\begin{pmatrix}d\tilde{M}^{(21)}_1\\ d\tilde{M}^{(21)}_2\end{pmatrix} = 
\begin{pmatrix}
a-M^{(1)} & -b\tan^{-1}\theta\\
b\tan\theta & a-M^{(1)}
\end{pmatrix}
\begin{pmatrix}dH_1\\dH_2\end{pmatrix}
\]
Thus 
\[
\bigwedge_{j=1}^2 \bigwedge_{i=1}^{N-2} d\tilde{M}^{(21)}_{ij} = \det\left\{\Lambda_{a,b,\theta}^T\otimes I_{N-2} - I_2\otimes M^{(1)}\right\}\bigwedge_{j=1}^2 \bigwedge_{i=1}^{N-2} dH_{ij}.
\]
Consider matrices $L_\theta, R_\theta \in M_2(\C)$ consisting of left and right eigenvectors of $\Lambda_{a, b, \theta}$ such that $L^\ast_\theta R_\theta = I$. Then $\Lambda_{a,b,\theta}^T = L_\theta^\ast\begin{pmatrix}\overline{\lambda}&0\\0&\lambda\end{pmatrix}R_\theta$. Then the determinant above is equal to
\begin{align*}
    &\det\left\{\Lambda_{a,b,\theta}^T\otimes I_{N-2} - I_2\otimes M^{(1)}\right\} \\
    &= \det\left\{L_{\theta}^\ast \otimes I_{N-2}\begin{pmatrix}\overline{\lambda}-M^{(1)}&0\\0&\lambda-M^{(1)}\end{pmatrix}R_\theta \otimes I_{N-2}\right\} \\
    &=\left|\det(M^{(1)}-\lambda)\right|^2.
\end{align*}
It remains to show that 
\[
J(\varphi) df\wedge \bigwedge_{j=1}^2 \bigwedge_{i=1}^{N-2} dH_{ij} = d\mathbf{v},
\]
where $d\mathbf{v}$ is the rotationally invariant volume form on $V^2(\R)$. First, choose a translation invariant local coordinate chart, so that $J(\varphi)$ is constant in $\mathbf{v}$ (we can always do this because the action of $\mathbf{O}(N)$ on $V^{2}(\R^{N})=\mathbf{O}(N)/\mathbf{O}(N-2)$ is transitive). For any $(v_{1},v_{2})\in V^{2}(\R^{N})$, choose a smooth lift to $(v_{1},v_{2},\ldots,v_{N})\in \mathbf{O}(N)$. We have 
\begin{align*}
df\wedge\bigwedge_{j=1,2}\bigwedge_{i=1}^{N-2}dH_{ij}=\bigwedge_{\substack{i=1,\ldots,N\\j=1,2\\j<i}}v_{i}^{T}dv_{j}
\end{align*}
We claim this is translation invariant; this would complete the proof. To this end, note
\begin{align*}
\bigwedge_{\substack{i,j=1,\ldots,N\\j<i}}v_{i}^{T}dv_{j}
\end{align*}
is translation invariant (see the bottom of page 16 of \cite{Edelman}). But this is a smooth lift of $\bigwedge_{\substack{i=1,\ldots,N\\j=1,2\\j<i}}v_{i}^{T}dv_{j}$, so the proposed translation invariance follows.
\end{proof}
\bibliographystyle{abbrv}
\bibliography{realnonhermitian_v2}

\end{document}